\documentclass[11pt,a4paper]{amsart}

\usepackage[margin=3cm]{geometry}
\usepackage[pdftex]{graphicx}
\DeclareGraphicsExtensions{.pdf, .jpg, .png}
\usepackage{color}
\graphicspath{{fig/}}
\usepackage{subcaption}
\usepackage{algorithm,placeins}
\usepackage{algorithmic,amsmath}

\renewcommand{\vec}{\mathbf}

\newcommand{\diag}{\operatorname{diag\,}}
\newcommand{\MCov}{M_\mathrm{Cov}}
\newcommand{\Tau}{\mathcal{T}}
\newcommand{\Kplus}{K^+}
\newcommand{\Kminus}{K^-}

\title[Anisotropic a posteriori error estimate for VEM]{Anisotropic a posteriori error estimate for the Virtual Element Method}

 \author{P.~F.~Antonietti}
  \author{S.~Berrone}
  \author{A.~Borio}
  \author{A.~D'Auria}
  \author{M.~Verani}
  \author{S.~Weisser}

  \address[P.F. Antonietti]{
    MOX, Dipartimento di Matematica,
    Politecnico di Milano, Italy;
    \emph{e-mail: paola.antonietti@polimi.it}
  }
  
  \address[S. Berrone]{Dipartimento di Scienze Matematiche, Politecnico di Torino, Italy; \emph{e-mail:stefano.berrone@polito.it}}
  \address[A. Borio]{Dipartimento di Scienze Matematiche, Politecnico di Torino, Italy; \emph{e-mail:andrea.borio@polito.it}}
  \address[A. D'Auria]{Dipartimento di Scienze Matematiche, Politecnico di Torino, Italy; \emph{e-mail:alessandro.dauria@polito.it}}
  
  \address[M. Verani]{
    MOX, Dipartimento di Matematica,
    Politecnico di Milano, Italy;
    \emph{e-mail: marco.verani@polimi.it}
  }

\address[S. Weisser]{FR Mathematik,
Universit\"at des Saarlandes, Germany;
\emph{e-mail:weisser@num.uni-sb.de}
}

\usepackage{comment}
\usepackage{amsmath}       %   AMS equation environments
\usepackage{amssymb}       %   AMS symbol fonts (e.g., \boldsymbol{}
\usepackage{chapterbib}    %   Bibtex
\usepackage{color}
\usepackage{comment}
\usepackage{bm,bbm}
\usepackage{overpic}
\usepackage{times}
\usepackage{epsfig}
\usepackage{graphicx,graphics,rotating}
\usepackage{multicol}
\usepackage{multirow}

\theoremstyle{plain}
\newtheorem{theorem}{Theorem}[section]
\newtheorem{lemma}[theorem]{Lemma}
\newtheorem{corollary}[theorem]{Corollary}
\newtheorem{proposition}[theorem]{Proposition}

\theoremstyle{definition}
\newtheorem{definition}[theorem]{Definition}
\newtheorem{remark}[theorem]{Remark}

\theoremstyle{plain}
\newtheorem{assumption}[theorem]{Assumption}

\definecolor{MyDarkGreen}{rgb}{0,0.45,0}

\def\trait #1 #2 #3 {\vrule width #1pt height #2pt depth #3pt}
\def\fin{\hfill
        \trait .3 5 0
        \trait 5 .3 0
        \kern-5pt
        \trait 5 5 -4.7
        \trait 0.3 5 0
\medskip}

\newcommand{\REAL}{\mathbbm{R}}

\newcommand{\xv}{\mathbf{x}}

\newcommand{\xs}{x}

\newcommand{\calM}{\mathcal{M}}

\newcommand{\LTWO}  {L^2}

\newcommand{\HS}[1] {H^{#1}}

\newcommand{\PS}[1] {\mathbbm{P}_{#1}}

\newcommand{\hh}{h}

\newcommand{\DIM} {d}              % space dimension

\newcommand{\nlen}{\hspace{-0.2mm}}

\newcommand{\snorm}  [2]{|#1|_{#2}}

\newcommand{\norm}   [2]{|\nlen|#1|\nlen|_{#2}}

\newcommand{\abs}    [1]{|#1|}

\newcommand{\EOD}{\end{document}}

\begin{document}

\maketitle

\begin{abstract}
{We derive an anisotropic a posteriori error estimate for the adaptive conforming Virtual Element approximation of a paradigmatic two-dimensional elliptic problem. In particular, we introduce a quasi-interpolant operator and exploit its approximation results to prove the reliability of the error indicator. We design and implement the corresponding  adaptive polygonal anisotropic algorithm.  Several numerical tests assess the superiority of the proposed algorithm in comparison with standard polygonal isotropic mesh refinement schemes. }

\vspace{0.2cm}
\end{abstract}

\section{Introduction and Notation}
In recent years, the numerical approximation of partial differential equations on computational meshes composed by arbitrarily-shaped
polygonal/polyhedral (polytopal, for short) elements has been 
the subject of an intense research activity.
Examples of polytopal element methods (POEMs) include
the Mimetic Finite Difference method,
%%\cite{%HymanShashkovSteinberg_1997,%
%  BreLipSha05,%
%  Gyrya-Lipnikov-Manzini:2016,%
%  Brezzi-Buffa-Lipnikov:2009,%
%  BeiraodaVeiga-Lipnikov-Manzini:2011,%
%  Lipnikov-Manzini-Moulton-Shashkov:2016,%
%  Lipnikov-Manzini:2014,%
%  BeiraoManziniLipnikov_2014,%
%  Lipnikov-Manzini-Shashkov:2014,%
%  Manzini-Lipnikov-Moulton-Shashkov:2017,%
%  Lipnikov-Manzini-Moulton-Shashkov:2016,%
%  AntoniettiBigoniVerani_2013,%
%  AntoniettiFormaggiaScottiVeraniVerzotti_2016}, 
%%
the Polygonal Finite Element Method,
%\cite{SukumarTabarraei_2004,%
%  SukumarTabarraei_2008,%
%  Manzini-Russo-Sukumar:2014}, 
%%
the Polygonal Discontinuous Galerkin Finite Element Method,
%\cite{AntBreMar09,dg_cfes_2012,%
%  BasBotColSub,%
%  Canetal14,%
%  CangianiDongGeorgoulisHouston_2016,%
%  AntoniettiHoustonet_al_2016,%
%  AntoniettiFacciolaRussoVerani_2016,%
%  CangianiDongGeorgoulis_2017,%
%  AntoniettiPennesi_2017,%
%  AntoniettiPennesiHouston_2018,%
%  DiPietro-Ern:2011,%
%  Brezzi:2000,%
%  Arnold:2001},
%%%
the Hybridizable Discontinuous Galerkin and Hybrid High-Order Methods,
%\cite{Cockburn-Gopalakrishnan-Lazarov:2009,%
%  CockburnDongGuzman_2008,%
%  DiPietroErnLemaire_2014,%
%  CockburnDiPietroErn_2016},
%%
the Gradient Discretization method,
%\cite{EymardGuichardHerbin_2012,%
%  DroniouEymardGallouetHerbin_2013,%
%  DiPietro-Droniou-Manzini:2018}, 
%%%
the Finite Volume Method,
%\cite{Droniou:2014},
%%
the BEM-based FEM,
% \cite{Weisser_basic}
%%
 the Weak Galerkin method 
%(see, e.g., \cite{Wang:2014}) 
 and the {V}irtual {E}lement method (VEM). For more details see the special issue \cite{special-issue} and the references therein. 

The novelty and recent surge of interest in POEMs stems from their ability to describe a physical domain using not only standard shapes (triangles, tetrahedra, square, hexahedra,...) but also highly irregular and arbitrary geometries.   
This flexibility of essentially arbitrary polytopal meshes is naturally very attractive for designing adaptive algorithms based on mesh refinement (and derefinement/agglomeration) driven by suitable a posteriori error estimates. However, while (isotropic and anisotropic) {error estimates} and a posteriori error estimates and adaptive finite element methods (AFEMs) have been intensively investigated during  the last decades (see, e.g., for the isotropic case the monographs \cite{Verfurth:2013, Nochetto-Veeser:2012} and the references therein and for the anisotropic case \cite{apel:book,FormaggiaPerotto2001,FormaggiaPerotto2003,Georgoulis:phd,Georgoulis2006,Houston:anisDG:2007c,Houston:anisDG:2007d} and the references therein), the corresponding study of a posteriori error estimates and adaptivity for polytopal methods is still in its infancy.  See, for example, \cite{Beirao:2008,BM:2008,ABLV:2013} for the study of a posteriori error estimates in the context of Mimetic Finite Differences, \cite{BM:2015,Berrone-Borio:2017,Cangiani_et_al-apost:2017,Mora-Rivera-Rodriguez:2017,Beirao-Manzini-Mascotto:2019,Chi-Beirao-Paulino:2019,cangiani2019posteriori} for the Virtual Element Method, \cite{Weisser:2011,Weisser:2017,Weisser-Wick:2018,Weisser:book} for polygonal BEM-based FEM, \cite{Botti-et-al:2017} for the polygonal Discontinuous Galerkin method, \cite{DiPietro-Specogna:2016} for the Mixed High Order method, \cite{Mu:2019} for the Weak Galerkin method and \cite{Vohralik-Yousef:2018}  for lowest-order locally conservative methods on polytopal meshes. 
Moreover, despite the great flexibility provided by polytopal meshes, the above works focused on the isotropic case, only. The anisotropic adaptive polytopal mesh refinement, to our knowledge, has been addressed only in \cite{Antonietti-et-al:anis} for the Virtual Element Method in two dimensions. For completeness, see  also the recent work \cite{Chen-anis:2019} for nonconforming VEM {\em a priori} anisotropic error analysis.   
Aim of this paper is to push forward the research of \cite{Antonietti-et-al:anis}  providing a rigorous polygonal anisotropic {\em a posteriori} error estimate for conforming VEM and numerically assessing its efficacy in driving polygonal adaptive anisotropic mesh refinement strategies for the virtual element approximation of a paradigmatic two-dimensional elliptic problem. \\

The outline of the paper is as follows. In Section \ref{sec:VEM} we introduce the continuous elliptic  problem together with its lowest order virtual element approximation. In Section \ref{S:preliminaries} we first make precise the notion of {\it polygonal anisotropy}, then we state the anisotropic mesh regularity  assumptions under which our theoretical results will be obtained. In the same section we also collect a series of instrumental results that will be employed in the subsequent analysis. In Section \ref{S:quasi-ineterp} we introduce a quasi-interpolant operator and prove approximation results that will be employed in Section \ref{S:apos} where a novel  polygonal anisotropic a posteriori error estimate  is obtained. Finally, in Section \ref{sec:numres} we present a set of numerical results assessing the validity of our theoretical error estimates and the capability of our anisotropic error indicators to drive an adaptive polygonal anisotropic mesh refinement strategy for the solution of an elliptic problem.

\subsection{Notation of functional spaces and technical results}
We use the standard definition and notation of Sobolev spaces, norms
and seminorms as given in~\cite{Adams-Fournier:2003}.
Hence, the Sobolev space $\HS{s}(\omega)$ consists of functions
defined on the open bounded connected subset $\omega$ of $\REAL^{2}$
that are square integrable and whose weak derivatives up to order $s$
are square integrable.
As usual, if $s=0$, we prefer the notation $\LTWO(\omega)$.
Norm and seminorm in $\HS{s}(\omega)$ are denoted by
$\norm{\cdot}{s,\omega}$ and $\snorm{\cdot}{s,\omega}$, respectively,
and $(\cdot,\cdot)_{\omega}$ denote the $\LTWO$-inner product.
The subscript $\omega$ may be omitted when $\omega$ is the whole computational
domain $\Omega$. 
%%
%If $\Gamma$ is the boundary of $\Omega$, we denote the space of the
%traces of the functions in $\HONE(\Omega)$, which are compactly
%supported in $\LTWO(\Gamma)$, by the notation
%$H^{\frac{1}{2}}_{c}(\Gamma)$, and a similar notation holds for the
%traces of the element boundary.

If $\ell\geq0$ is an integer number, $\PS{\ell}(\omega)$ is the space
of polynomials of degree up to $\ell$ defined on $\omega$, with the
convention that $\PS{-1}(\omega)=\{0\}$.
The $\LTWO$-orthogonal projection onto the polynomial space
$\PS{\ell}(\omega)$ is denoted by
$\Pi^{0,\omega}_{\ell}\,:\,\LTWO(\omega)\to\PS{\ell}(\omega)$.
The space $\PS{\ell}(\omega)$ is the span of the finite set of
\emph{scaled monomials of degree up to $\ell$}, that are given by
\begin{align}
  \calM_{\ell}(\omega) =
  \bigg\{\,
    \left( \frac{\xv-\bar\xv_{\omega}}{\hh_{\omega}} \right)^{\alpha}
    \textrm{~with~}\abs{\alpha}\leq\ell
    \,\bigg\},
\end{align}
where 
\begin{itemize}
\item $\bar\xv_{\omega}$ denotes the center of gravity of $\omega$ and
  $\hh_{\omega}$ its characteristic length, as, for instance, the edge
  length or the cell diameter for $\DIM=1,2$;
\item $\alpha=(\alpha_1,\alpha_2)$ is the
  two-dimensional multi-index of nonnegative integers $\alpha_i$
  with degree $\abs{\alpha}=\alpha_1+\alpha_{2}\leq\ell$ and
  such that
  $\xv^{\alpha}=\xs_1^{\alpha_1}\xs_{2}^{\alpha_{2}}$ for
  any $\xv\in\REAL^{2}$.
\end{itemize}

{\color{black} Finally, we use the symbols $\lesssim$ and $\gtrsim$ to denote inequalities holding  up to a positive constant that is independent of the 
characteristic length of mesh elements, but may depend on the problem constants, like the
coercivity and continuity constants, or other discretization constants
like the mesh regularity constant, the stability constants, etc.
Accordingly, $a\simeq b$ means $a\lesssim b \lesssim a$.
The hidden constant generally has  a different value at each occurrence.}

\section{Model problem and Virtual Element discretization}
\label{sec:VEM}
Let $\Omega\subset \mathbb{R}^2$ be a bounded polygonal domain. In this paper we are interested in deriving anisotropic error estimates for the virtual element approximation of the following elliptic problem:
\begin{equation}\label{pb}
 -\Delta u = f \quad \text{in~}\Omega,\qquad u=0\quad \text{on~}\partial\Omega
\end{equation}
with $f\in L^2(\Omega)$. The variational formulation of \eqref{pb} reads as: Find $u\in H^1_0(\Omega)$ such that 
\begin{equation}\label{pb:weak}
a(u,v)=F(v)
\end{equation}
for every $v\in H^1_0(\Omega)$ where 
$a(u,v)=\int_\Omega \nabla u \cdot \nabla v ~d\xv$ and $F(v)=\int_\Omega fv ~d\xv$. 

We now briefly recall (see \cite{volley} for more details) the lowest order virtual element approximation to \eqref{pb:weak}. Let $\{\mathcal{K}_h\}_h$ be a sequence of decompositions of $\Omega$ where each mesh $\mathcal{K}_h$ is a collection of nonoverlapping
 polygonal elements $K$ with boundary $\partial K$, and  let $\mathcal{E}_h$ be the set
of edges $E$ of $\mathcal{K}_h$.
Each mesh is labeled by $h$, the diameter of the mesh, defined as
usual by $h=\max_{E\in \mathcal{K}_h} h_{K}$, where
$h_{K}=\sup_{\mathbf{x},\mathbf{y}\in K}\vert\mathbf{x}-\mathbf{y}\vert$. 
We denote the set of vertices ${\sf v}$ in $\mathcal{K}_h$ by $\mathcal{V}_h$. The global lowest order  virtual element space is defined as 
\begin{equation}\label{VEM:global}
 {\color{black} V_{h,0}}=\{v_h\in H^1_0(\Omega):~v_h\vert_K\in V_h^K~\text{and~} v_h({\sf v})=0 ~\forall {\sf v}\in \partial\Omega\}\subset H^1_0(\Omega),
 \end{equation}  
 where 
\begin{equation} \label{VEM:local}
 {\color{black}V_h(K)}=\{v_h\in H^1(K):~\Delta v_h=0 ~\text{in}~ K,~ v_h\vert _E \in \mathbb{P}^1(E)~\forall E\subset \partial K\},
\end{equation}
 is the local virtual element space.
We denote by $u_h\in {\color{black} V_{h,0}} $ the virtual element approximation to the solution $u$ of \eqref{pb:weak}, defined as  the unique solution to 
\begin{equation}
a_h(u_h,v_h)=(f_h,v_h) 
\end{equation}
for every $v_h\in {\color{black} V_{h,0}}$, where 
$f_h$
 is the piecewise constant approximation of $f$ {\color{black} on $\mathcal{K}_h$} and $a_h(u_h,v_h)=\sum_{K\in \mathcal{K}_h} a^K_h(u_h,v_h)$ being
$$a^K_h(u_h,v_h)= \int_K {\color{black}\Pi_0^{0,K}} \nabla(u_h) \cdot {\color{black}\Pi_0^{0,K}}  \nabla(v_h) ~d\xv + S^K((I-{\color{black}\Pi_1^{0,K}} )(u_h),(I-{\color{black}\Pi_1^{0,K}} )(v_h)),$$ 
the local discrete bilinear form that satisfies the usual stability and consistency properties (see \cite{volley} for precise definitions). For {\color{black} $w_h\in \text{Ker}(\Pi_1^{0,K})$} the stabilization form $S^K(\cdot,\cdot)$ is defined as 
$$ S^K(w_h,w_h)=\sum_{i=1}^{n_K} w_h^2({\sf v}_{i,K}),$$
being ${\sf v}_{i,K}$, $i=1,\ldots,n_K$ the vertices of $K$.  For more details about different choices for the stabilization form, see \cite{Beirao-Lovadina-Russo}.

\section{Polygonal Anisotropy and mesh regularity }\label{S:preliminaries}
In this section, following \cite{Weisser:2019}, we first make precise the notions of {\it isotropic} and {\it anisotropic} polygonal element. This will be obtained analysing the spectral decomposition of a suitable matrix ({\color{black} in the sequel named {\it covariance} matrix) } associated to the element. 
More precisely, let $K$ be a polygonal element of the partition  $\mathcal{K}_h$. We denote by  $|K|$ the area of $K$, we define the barycenter of $K$  as
\[
  \bar \xv_K = \frac{1}{|K|}\int_K \xv\,d\xv,
\]
and we introduce the covariance matrix of $K$ as
\begin{equation}\label{eq:covariance}
  \MCov(K) = \frac{1}{|K|}\int_K (\xv-\bar \xv_K)(\xv-\bar \xv_K)^\top\,d\xv \in\mathbb R^{2\times 2}.
\end{equation}
Obviously, $\MCov$ is real valued, symmetric and positive definite, once we assume  that $K$ is not degenerating (i.e. $|K| >0$). Therefore,  $\MCov$ admits an eigenvalue decomposition
\[
  \MCov(K) = U_K\Lambda_KU_K^\top,
\]
with
\begin{equation}\label{eq:eigenvalues}
  U^\top=U^{-1}
  \quad\mbox{and}\quad
  \Lambda_K=\diag(\lambda_{K,1},\lambda_{K,2}),
\end{equation}
where $\lambda_{K,1}\geq\lambda_{K,2}>0$.

The eigenvectors of $\MCov(K)$ give the characteristic directions of $K$.  Consequently, if 
\[
  \MCov(K) = cI
\]
for $c>0$, there are no dominant directions in the element $K$. Thus, we can characterise the anisotropy with the help of the quotient $\lambda_{K,1}/\lambda_{K,2}\geq 1$ and say that an element $K$ is
\begin{align*}
  \text{isotropic, if}\quad & \frac{\lambda_{K,1}}{\lambda_{K,2}} \approx 1, \\
  \text{and anisotropic, if}\quad & \frac{\lambda_{K,1}}{\lambda_{K,2}} \gg 1.
\end{align*}

{\color{black} Hinging upon the above spectral informations on the polygonal elements}, we introduce a linear transformation of an anisotropic element $K$ onto a kind of reference element $\widehat{K}$. {\color{black} For each $\xv\in K$, we define the mapping by
\begin{equation}\label{trafo:1}
  \xv \mapsto \widehat{\xv} = F_K(\xv)=A_K \xv
  \quad\mbox{ with }\quad
  A_K=\alpha_K \Lambda_K^{-1/2}U_K^\top
\end{equation}
where $\alpha_K>0$ will be chosen later. From now on, $\widehat{K}=F_K(K)$ will be called the reference element associated to $K$. 

It is possible to prove (see \cite{Weisser:2019}) the following result.

\begin{lemma}\label{lem:PropTransform}
  There  holds
  \begin{enumerate}
    \item $|\widehat K| = \alpha_K^2|K|/\sqrt{\det(\MCov(K))}$,
    \item $\overline{{\xv}}_{\widehat{K}} = F_K(\overline{{\xv}}_K)$,
    \item $\MCov(\widehat K) = \alpha_K^2I$.
  \end{enumerate}
\end{lemma}

According to the previous lemma, the reference element $\widehat K$ is isotropic, since $\lambda_{\widehat K,1}/\lambda_{\widehat K,2}=1$, and thus, it has no dominant direction. For what concerns the choice of the parameter $\alpha_K$ we set  
\begin{equation}\label{eq:Alpha}
  \alpha_K = \left(\frac{\sqrt{\det(\MCov(K))}}{|K|}\right)^{1/d}
  = \left(\frac{\sqrt{\lambda_{K,1}\lambda_{K,2}}}{|K|}\right)^{1/2},
\end{equation}
which obviously ensures, in view of  Lemma~\ref{lem:PropTransform}, $|\widehat K|=1$.}
As usual, we mark the operators and functions defined over the reference configuration by a hat, as, for instance,
\[
 \widehat v = v\circ F_K^{-1}:\widehat K \to K.
\]
Obviously, it is
\begin{equation}\label{eq:IdentityGradient}
 \nabla v = \alpha_KU_K\Lambda_K^{-1/2}\widehat\nabla \widehat v,
\end{equation}
and, after some algebra,
\begin{equation}\label{eq:IdentityHessian}
 \widehat H(\widehat v) = \alpha_K^{-2}\Lambda_K^{1/2}U_K^\top H(v)U_K\Lambda_K^{1/2},
\end{equation}
where $H(v)$ denotes the Hessian matrix of~$v\in H^2(\Omega)$ and $\widehat H(\widehat v)$ the corresponding Hessian on the reference configuration.

%\section{Mesh regularity and preliminary results}
%\label{sec:Regularity}
Following \cite{Weisser:2019}, we are now ready to state the mesh requirements which will be needed in the sequel for deriving the properties  of the quasi-interpolation operator (Section \ref{S:quasi-ineterp}) and the anisotropic a posteriori error analysis (Section \ref{S:apos}). We first recall the notion of isotropic regular polygonal meshes (Definition \ref{def:reg_isotropic_mesh}) which is instrumental for the definition of {\it anisotropic} polygonal meshes (Definition \ref{def:reg_anisotropic_mesh}).

\begin{definition}[{\bf regular isotropic mesh}]\label{def:reg_isotropic_mesh}
A polygonal mesh  $\mathcal K_h$ is called {\it regular} or a {\it regular isotropic mesh}, if all elements $K\in \mathcal K_h$ 
are such that:
\begin{enumerate}
 \item[(a)] $K$ is a star-shaped polygon with respect to a circle of radius~$\rho_K$ and center~$z_K\in K$. 
 \item[(b)] The aspect ratio is uniformly bounded from above by $\sigma_\mathcal{K}$, i.e. $h_K/\rho_K<\sigma_\mathcal{K}$, being $h_K$ the diameter of $K$.
 \item[(c)] For every edge $E\subset\partial K$ it holds $h_K\leq c_\mathcal{K}h_E$, being $h_E$ the length of $E$.
\end{enumerate}
Here, the constants $\sigma_{\mathcal K}$ and $c_{\mathcal K}$ have to be uniform for all considered regular elements.
\end{definition}
%In~\cite{Weisser2014}, it has been shown that under these assumptions, the triangulation obtained by connecting the vertices of a regular polygonal element~$K$ with its point~$z_K$ is regular in the sense of Ciarlet \cite{BrennerScott2002}.

\begin{definition}[{\bf regular anisotropic mesh}]\label{def:reg_anisotropic_mesh}
 Let $\mathcal K_h$ be a polygonal mesh with anisotropic elements. $\mathcal K_h$ is called \emph{regular} or a \emph{regular anisotropic mesh}, if
 \begin{enumerate}
  \item[(a')] The reference configuration $\widehat K$ for all $K\in\mathcal K_h$ obtained by~\eqref{trafo:1} is a regular polygonal element according to Definition~\ref{def:reg_isotropic_mesh}.
  \item[(b')] Neighbouring elements behave similarly in their anisotropy. More precisely, for two neighbouring elements $\Kplus$ and $\Kminus$, i.e. $\overline \Kplus\cap\overline \Kminus\neq\varnothing$, with covariance matrices 
   \[
     M_\mathrm{Cov}(\Kplus) = U_{\Kplus}\Lambda_{\Kplus}U_{\Kplus}^\top
     \quad\mbox{and}\quad
     M_\mathrm{Cov}(\Kminus) = U_{\Kminus}\Lambda_{\Kminus}U_{\Kminus}^\top
   \] 
   as defined above, we can write
   \[
     \Lambda_{\Kminus} = (I+\Delta^{\Kplus,\Kminus})\Lambda_{\Kplus}
     \quad\mbox{with}\quad
     \Delta^{\Kplus,\Kminus} = \diag\left(\delta^{\Kplus,\Kminus}_1,\delta^{\Kplus,\Kminus}_2\right),
   \]
   and
   \[
     U_{\Kminus} = R^{\Kplus,\Kminus}U_{\Kplus}
     \quad\mbox{with}\quad
     R^{\Kplus,\Kminus} ~\text{rotation~matrix}
     %= \begin{pmatrix}
      %               \cos\phi^{\Kplus,\Kminus} & -\sin\phi^{\Kplus,\Kminus}\\
        %             \sin\phi^{\Kplus,\Kminus} &  \cos\phi^{\Kplus,\Kminus}
        %           \end{pmatrix},
   \]
   where for $i=1,2$
   \[
     0 \leq |\delta^{\Kplus,\Kminus}_i| < c_\delta < 1
     \quad\mbox{and}\quad
     \|I- R^{\Kplus,\Kminus}\|_0\left(\frac{\lambda_{\Kplus,1}}{\lambda_{\Kplus,2}}\right)^{1/2} < c_\phi.
    % |\phi^{\Kplus,\Kminus}|\left(\frac{\lambda_{\Kplus,1}}
    %{\lambda_{\Kplus,2}}\right)^{1/2} < c_\phi.
   \]
   uniformly for all neighbouring elements, being $\|\cdot\|_0$ the spectral norm.
 \end{enumerate}
\end{definition}

Thus a regular anisotropic element can be mapped according to~\eqref{trafo:1} onto a regular polygonal element in the usual sense. In the definition of quasi-interpolation operators (see Section \ref{S:quasi-ineterp} ), we deal, however, with patches of elements instead of single elements. Thus, we study the mapping of such patches. 
Let $\omega=\omega_{\sf v}$ be the neighbourhood of the vertex ${\sf v}$ which is defined by
\[
  \overline\omega_{\sf v} = \bigcup\left\{\overline{K'}: {\sf v}\in\overline{K'},\quad K'\in\mathcal K_h\right\}.
\]
Furthermore, for $K\in\mathcal K_h$, recall that the map $F_K$ defined in ~\eqref{trafo:1} is given by
\begin{equation*}
  \xv \mapsto F_K(\xv) = A_K \xv = \alpha_K \Lambda_K^{-1/2}U_K^\top \xv.
\end{equation*}
Consequently, we may write $\widehat K = F_K(K)$ and we know, that $\widehat K$ is regular for all $K\in\mathcal K_h$ with some regularity parameters $\sigma_\mathcal K$ and $c_\mathcal K$. However, let now $\Kplus,\Kminus\in\mathcal K_h$ with $\Kplus,\Kminus\subset\omega$. We are interested in the regularity of $F_{\Kplus}(\omega)$ and $F_{\Kplus}(\Kminus)$.
For the proofs of the following results we refer to \cite{Weisser:2019}.
\begin{lemma}\label{lem:regPerturbedMapping}
Let~$\mathcal K_h$ be a regular anisotropic mesh, $\omega=\omega_{\sf v}$ be a patch as described above, and $\Kplus,\Kminus\in\mathcal K_h$ with $\Kplus,\Kminus\subset\omega$. The mapped element $F_{\Kplus}(\Kminus)$ is regular in the sense of Definition \ref{def:reg_isotropic_mesh} with slightly perturbed regularity parameters $\widetilde\sigma_\mathcal K$ and $\widetilde c_\mathcal K$. Consequently, the mapped patch $F_K(\omega)$ consists of regular polygonal elements for all $K\in\mathcal K_h$ with $K\subset\omega$.
\end{lemma}

\begin{proposition}\label{prop:BoundedNumOfNodesVSElems}
Let~$\mathcal K_h$ be a regular anisotropic mesh. Each vertex~${\sf v}$ of the mesh~$\mathcal K_h$ belongs to a uniformly bounded number of elements. Viceversa, each element $K\in\mathcal K_h$ has a uniformly bounded number of vertices on its boundary. 
\end{proposition}

{\color{black} In the rest of the paper we will work under the following mesh assumption.
\begin{assumption}\label{assumption}
Let $\{\mathcal{K}_h \}_h$ be a sequence of regular anisotropic meshes with regularity parameters uniformly bounded with respect to $h$.
\end{assumption}

Finally, we recall some instrumental results (see \cite{Weisser:2019} for the proofs) that will employed in the next section.
}

\begin{lemma}\label{lem:MapH1norm}
 Let $K\in\mathcal K_h$ be a polygonal element of a regular anisotropic mesh~$\mathcal K_h$. Then, for $v\in H^1(K)$ and corresponding $\widehat v\in H^1(\widehat K)$ there holds
\begin{eqnarray}
&& \|\widehat{v}\|_{L^2(\widehat{K})}=\vert K \vert^{-1/2} 
\|v\|_{L^2(K)}\label{norm1}\\
&&\|\widehat{\nabla}\widehat{v}\|_{L^2(\widehat{K})}=\vert K \vert^{-1/2} 
\|A_K^{-T} \nabla v\|_{L^2(K)}\label{norm2}\\
&&\sqrt{\frac{\lambda_{K,2}}{\lambda_{K,1}}}\;|\widehat v|_{H^1(\widehat K)}^2 
  \leq |v|_{H^1(K)}^2 
  \leq \sqrt{\frac{\lambda_{K,1}}{\lambda_{K,2}}}\;|\widehat v|_{H^1(\widehat K)}^2\label{norm3} .
\end{eqnarray}
\end{lemma}

\begin{lemma}[\bf anisotropic trace inequality]\label{lem:AnisotropicTraceInequality}
Let $K\in\mathcal K_h$ be a polygonal element of a regular anisotropic mesh~$\mathcal K_h$. For an edge $E \subset \partial K$ it holds
 \begin{equation}\label{eq:AnisotropicTraceInequality}
  \|v\|_{L_2(E)}^2 \lesssim \;\frac{|E|}{|K|}\left(\|v\|_{L_2(K)}^2+\|\alpha_K^{-1}\Lambda_K^{1/2}U_K^\top\nabla v\|_{L_2(K)}^2\right).
 \end{equation}
 \end{lemma}
 \begin{lemma}[\bf best-approximation by a constant]\label{lem:BestApproxByConst}
Let $K\in\mathcal K_h$ be a polygonal element of a regular anisotropic mesh~$\mathcal K_h$. For $v\in H^1(K)$, there exists a constant $p\in\mathcal P^0(K)$ such that
 \[
   \|v-p\|_{L_2(K)} \lesssim\;\|\alpha_K^{-1}\Lambda_K^{1/2}U_K^\top\nabla v\|_{L_2(K)}.
 \]

\end{lemma}

\section{Quasi-Interpolation of Functions in $H^1(\Omega)$} \label{S:quasi-ineterp}
{\color{black} Let $\omega$ be a patch of physical elements belonging to a regular anisotropic polygonal mesh. Let $\widehat{\omega}$ be the patch of reference elements $\widehat{K}$ 
such that $\hat\omega=F_{K^*}(\omega)$, where the mapping is dictated by an element $K^*$ of the patch $\omega$.} On the reference patch $\widehat{\omega}$ we introduce the space
\begin{equation}\label{mapped-space}
\widehat{\Theta}(\widehat{\omega})=\{\widehat{\theta}\in C^0(\widehat{\omega}):~\forall
\widehat{K} \in \widehat{\omega} \, \,\widehat{\theta}_{\vert \widehat{K}} =\theta\circ F_{K^*}^{-1}, ~\theta\in V_h( F_{K^*}^{-1}(\widehat{K}))
\}
\end{equation}
where $V_h( F_{K^*}^{-1}(\widehat{K})) $ is the lowest order local virtual element space defined on the  polygon $F_{K^*}^{-1}(\widehat{K})$.\footnote{Note that the polygon $F_{K^*}^{-1}(\widehat{K})$ is not necessarily equal to $K^*$, unless we consider exactly the reference polygon $\widehat{K^*}$ associated to $K^*$.} 
{\color{black} We remark that in view of Lemma \ref{lem:regPerturbedMapping} the specific choice of the element $K^*$ in the definition of the space $\widehat{\Theta}(\widehat{\omega})$ is not restrictive.  Moreover, we observe that, in view of Assumption~\ref{assumption}, the dimension of $\widehat{\Theta}(\widehat{\omega})$ is uniformly bounded with respect to $h$.}
Finally, it is worth noticing that functions in $\widehat{\Theta}(\widehat{\omega})$ are not necessarily virtual element functions. However, constant functions are contained in $\widehat{\Theta}(\widehat{\omega})$ and this will be sufficient for our scopes.

{\color{black} Now, following \cite{Bernardi-Girault:1998}, we introduce a projection operator $\widehat{r}_\omega(\widehat{v})$ on the reference patch $\widehat{\omega}$.}
\begin{definition}\label{def:r}
For any function $\widehat{v} \in L^1(\widehat{\omega})$ we define 
$\widehat{r}_\omega(\widehat{v})\in \widehat{\Theta}(\widehat{\omega})$ as 
\begin{equation}
\int_{\widehat{\omega}}(\widehat{r}_{\widehat{\omega}}( \widehat{v})- \widehat{v}) \widehat{\theta}=0
\qquad \forall \widehat{\theta}\in \widehat{\Theta}(\widehat{\omega}).
\end{equation}
\end{definition}
 It is important to remark that
$\widehat{r}_{\widehat{\omega}}$ is a projection operator  on $\widehat{\omega}$. 

On the physical patch $\omega$ we can define 
$r_\omega(v)$ so that $r_\omega(v)\circ F_{K^*}^{-1}=\widehat{r}_{\widehat{\omega}}(v\circ F_{K^*}^{-1})$, i.e.
$\widehat{r_\omega(v)}=\widehat{r}_{\widehat{\omega}}(\widehat{v})$. 
Let  $\omega_i$ be the patch of elements sharing the vertex $\sf{v}_i$ and set $r_i=r_{\omega_i}$. {\color{black} The operators $r_i$ will be employed to build the quasi-interpolant $\mathfrak I_C$ (see \eqref{def:quasi-interpolant} below). In the sequel, we collect some approximation results for $r_i$ that will be instrumental for proving the approximation properties of $\mathfrak I_C$}.
\begin{lemma}\label{lemma1}
Let $\mathcal K_h$ be a \emph{regular anisotropic mesh}.
For any $K\subset \omega_i$ there hold
\begin{equation}\label{eq1:lemma1}
\| u - r_i(u)\|^2_{L^2(K)} \lesssim \sum_{\widetilde{K}\subset \omega_i} \frac{\vert K\vert}{\vert \widetilde{K}\vert} \|A_{\widetilde{K}}^{-T}\nabla u\|^2_{L^2(\widetilde{K})}.
\end{equation}
which can also be written in the following way
\begin{equation}\label{eq2:lemma1}
\| u - r_i(u)\|^2_{L^2(K)} \lesssim \vert K\vert \sum_{\widetilde{K}\subset \omega_i} \sqrt{\frac{\lambda_{\widetilde{K},1}}{\lambda_{\widetilde{K},2}}}\vert u \vert^2_{H^1(\widetilde{K})}.
\end{equation}

\end{lemma}
\begin{proof}
Let $K\in\omega_i$ then we have  
\begin{eqnarray}
\|u-r_i(u)\|_{L^2(K)}&=&\vert K\vert^{1/2}\|\widehat{u}-\widehat{r}_{\color{black}\widehat{\omega}_i}(\widehat{u})\|_{L^2(\widehat{K})}.\label{aux:lemma1}
%&\leq& 2 \vert K_0\vert^{1/2}\|\widehat{u}\|_{L^2(\widehat{\omega})}\nonumber\\
%&=&2 \vert K_0\vert^{1/2} 
%\left(\sum_{\widehat{K}\subset\widehat{\omega}} \|\widehat{u}\|^2_{L^2(\widehat{K})}\right)^{1/2}\nonumber\\
%&=&2 
%\left(\sum_{K\subset\omega_i} 
%\frac{\vert K_0\vert}{\vert K\vert}\|u\|^2_{L^2(K)}\right)^{1/2}\nonumber\\
%&\lesssim& \|u\|^2_{L^2(\omega_i)}
\end{eqnarray}
Now, employing the fact that $\widehat{r}_{\color{black}\widehat{\omega}_i}$ is a projector on $\widehat{\omega}_i$ we have 
$$ \widehat{u}-\widehat{r}_{\color{black}\widehat{\omega}_i}(\widehat{u})=\widehat{u}-\widehat{\theta}- \widehat{r}_{\color{black}\widehat{\omega}_i}(\widehat{u}-\widehat{\theta}),$$
for $\widehat{\theta}\in \widehat{\Theta}(\widehat{\omega}_i)$
which implies
$$ \| \widehat{u}-\widehat{r}_{\color{black}\widehat{\omega}_i}(\widehat{u})\|_{L^2(\widehat{\omega}_i)}\leq 2 
\| \widehat{u}-\widehat{\theta}\|_{L^2(\widehat{\omega}_i)}.
$$
Assume $\widehat{\theta}$ is constant on $\widehat{\omega}_i$ and $\widehat{u}\in H^1(\widehat{\omega}_i)$, hence employing standard interpolation error estimate together with \eqref{norm2} we have
\begin{eqnarray}
 \| \widehat{u}-\widehat{r}_{\color{black}\widehat{\omega}_i}(\widehat{u})\|_{L^2(\widehat{\omega}_i)} &\lesssim& \vert \widehat{u}\vert_{H^1(\widehat{\omega}_i)}\nonumber\\
 &\lesssim& \left( \sum_{\widehat{K}\subset \widehat{\omega}_i}
 \vert \widehat{u}\vert^2_{H^1(\widehat{K})}
  \right)^{1/2}\nonumber\\
  &\lesssim & \left( 
\sum_{\widetilde{K}\subset \omega_i} \vert \widetilde{K}\vert^{-1} \|A_{\widetilde{K}}^{-T} \nabla u \|^2_{L^2(\widetilde{K})}  
  \right)^{1/2}.\nonumber
 \end{eqnarray}
 Combining \eqref{aux:lemma1} with the above inequality yields \eqref{eq1:lemma1}. On the other hand, using \eqref{norm2}-\eqref{norm3} we get \eqref{eq2:lemma1}.
\end{proof}
\begin{lemma}\label{lemma2}
For any $K\subset \omega_i$ there holds
\begin{equation}
\vert  u - r_i(u)\vert^2_{H^1(K)}\lesssim \sum_{\widetilde{K}\subset \omega_i} 
\sqrt{\frac{\lambda_{K,1} \lambda_{\widetilde{K},1}}{\lambda_{K,2}\lambda_{\widetilde{K},2}}}\vert u \vert^2_{H^1(\widetilde{K})}.
\end{equation}
\end{lemma}
\begin{proof}
By using \eqref{norm3} and taking $\widehat{\theta}$ constant on $\widehat{\omega}_i$, employing the equivalence of all norms on the finite dimensional space $\widehat{\Theta}(\widehat{\omega}_i)$, the fact that  $\widehat{r}_{\color{black}\widehat{\omega}_i}$ is a projection on $\widehat{\omega}_i$ and standard interpolation error estimate  we obtain
\begin{eqnarray}
\vert  u - r_i(u)\vert_{H^1(K)} &\leq& \left( \frac{\lambda_{K,1}}{\lambda_{K,2}}\right)^{1/4} \vert \widehat{u}-\widehat{r}_{\color{black}\widehat{\omega}_i}(\widehat{u})\vert_{H^1(\widehat{K})}\leq \left( \frac{\lambda_{K,1}}{\lambda_{K,2}}\right)^{1/4} \vert \widehat{u}-\widehat{r}_{\color{black}\widehat{\omega}_i}(\widehat{u})\vert_{H^1(\widehat{\omega}_i)}\nonumber\\
&\leq&  
\left( \frac{\lambda_{K,1}}{\lambda_{K,2}}\right)^{1/4}
\left( \vert \widehat{u}-\widehat{\theta} \vert_{H^1(\widehat{\omega}_i)} +   
\vert \widehat{\theta}-\widehat{r}_{\color{black}\widehat{\omega}_i}(\widehat{u}) \vert_{H^1(\widehat{\omega}_i)} 
\right)\nonumber\\
&\lesssim& 
 \left( \frac{\lambda_{K,1}}{\lambda_{K,2}}\right)^{1/4}
\left( \vert \widehat{u}-\widehat{\theta} \vert_{H^1(\widehat{\omega}_i)} +   
\| \widehat{\theta}-\widehat{r}_{\color{black}\widehat{\omega}_i}(\widehat{u}) \|_{L^2(\widehat{\omega}_i)} 
\right)\nonumber\\
&\lesssim& 
 \left( \frac{\lambda_{K,1}}{\lambda_{K,2}}\right)^{1/4}
\left( \vert \widehat{u}-\widehat{\theta} \vert_{H^1(\widehat{\omega}_i)} +   
\| \widehat{r}_{\color{black}\widehat{\omega}_i}(\widehat{\theta}- \widehat{u}) \|_{L^2(\widehat{\omega}_i)} 
\right)\nonumber\\
&\lesssim& 
 \left( \frac{\lambda_{K,1}}{\lambda_{K,2}}\right)^{1/4}
\left( \vert \widehat{u}-\widehat{\theta} \vert_{H^1(\widehat{\omega}_i)} +   
\| \widehat{\theta}- \widehat{u} \|_{L^2(\widehat{\omega}_i)} 
\right)\nonumber\\
&\lesssim& 
 \left( \frac{\lambda_{K,1}}{\lambda_{K,2}}\right)^{1/4}
\vert \widehat{u}\vert_{H^1(\widehat{\omega}_i)}\nonumber.
\end{eqnarray}
By using \eqref{norm3} on each $\widehat{K}\subset\widehat{\omega}_i$ we get the thesis.
\end{proof}
{\color{black} We are now ready to introduce the quasi-interpolation operator. For simplicity of exposition, we first consider the case where no boundary conditions are imposed on the boundary of $\Omega$. To this aim, we introduce the  global lowest order virtual element space $V_h\subset H^1(\Omega)$, which is defined as $V_{h,0}$ except for the conditions imposed on the boundary vertexes (cf. \eqref{VEM:global}).}
The quasi-interpolation of lowest order $\mathfrak I_C: H^1(\Omega)\to V_h$ is defined as
\begin{equation}\label{def:quasi-interpolant}
 (\mathfrak I_C v)(x) = \sum_{i=1}^N [r_i(v)]({\sf v}_i)\;\varphi_i(x)
\end{equation}
where $r_i=r_{\omega_i}$  and $\varphi_i\in V_h$ is the global virtual element basis function with $\varphi_i({\sf v}_j)=\delta_{i,j}$, $i,j=1,\ldots, N.$

We first observe that the following inverse inequality holds.
\begin{lemma} \label{Lemma:inverse}
For any $v\in V_h(K)$ there holds
\begin{equation}
\vert v\vert_{H^1(K)}\lesssim \frac{\sqrt{\lambda_{K,1}}}{\lambda_{K,2}} \| v\|_{L^2(K)}.
\end{equation}
\end{lemma}
\begin{proof}
We follow the steps of the proofs of  Lemma 3.2, Lemma 3.4 and Theorem 3.6 in 
\cite{Chen:Calcolo}. We report here the details only for completeness. 
We introduce a sub-triangulation $\Tau_K$ of $K$ made of (possibly anisotropic) triangles having at least one edge coinciding with one of the edges of  the polygon $K$. We denote by $S_1(\Tau_K)$ the space of piecewise continuous linear finite element functions on $\Tau_K$ and set 
$S_1^0(\Tau_K)=S_1(\Tau_K)\cap H^1_0(K)$. We introduce the projector $Q_K:V_h(K)\to S_1(\Tau_K)$
defined as
$$Q_k v\vert_{\partial K}= v\vert_{\partial K}\qquad 
(Q_k v, \phi)_{L^2(K)}=(v,\phi)_{L^2(K)}\quad\forall \phi \in S_1^0(\Tau_K).$$
The following splitting will be employed in the sequel
$$ Q_k v =v_{\partial,h}+v_{0,h}$$
where $v_{\partial,h}\in S_1(\Tau_K)$ with $v_{\partial,h}\vert_{\partial K}=v\vert_{\partial K}$ (we recall that $v\vert_{\partial K}$ is piecewise linear as $v\in V_h(K)$) and $v_{0,h}\in S_1^0(\Tau_K)$ defined as $v_{0,h}=Q_k v - v_{\partial,h}$.
We observe that it holds
$$(Q_K v, Q_K v)_{L^2(K)}=  (Q_K v,v_{\partial,h}) + (Q_K v,v_{0,h})=I_1+I_2. $$
Obviously, we have 
$$ 
I_1\leq \|Q_K v\|_{L^2(K)}\| v_{\partial,h}\|_{L^2(K)}
$$
and 
$$ 
 I_2=(v,v_{0,h})_{L^2(K)}\leq  \|v\|_{L^2(K)} \|v_{0,h}\|_{L^2(K)}
\leq \|v\|_{L^2(K)}\left (
\|Q_K v\|_{L^2(K)}
+ 
\|v_{\partial,h}\|_{L^2(K)}\right)
$$
from which it follows
$$
\|Q_K v\|_{L^2(K)}^2\leq \|Q_K v\|_{L^2(K)}\|v_{\partial,h}\|_{L^2(K)}
+ 
\|Q_K v\|_{L^2(K)}\|v\|_{L^2(K)}
+ 
\|v\|_{L^2(K)}\|v_{\partial,h}\|_{L^2(K)},$$
which implies
\begin{equation}\label{aux:chen:1}
\|Q_K v\|_{L^2(K)}\lesssim \|v\|_{L^2(K)}+\|v_{\partial,h}\|_{L^2(K)}.
\end{equation}
Let us now estimate the term $\|v_{\partial,h} \|_{L^2(K)}$. We observe that it holds
\begin{eqnarray}
\|v_{\partial,h} \|_{L^2(K)}^2&\simeq& \sum_{E\subset \partial K} \|v_{\partial,h} \|^2_{L^2(T_E)}\nonumber\\
&\lesssim & \sum_{E\subset \partial K} h_E \|v_{\partial,h}\|^2_{L^2(E)}\lesssim 
\sum_{E\subset \partial K} h_E \|v\|^2_{L^2(E)}.\nonumber
\end{eqnarray}
where $T_E\in \Tau(K)$ is the triangle having $E$ as an edge and $h_E$ is the diameter of $E$.
Employing the weighted trace estimate 
$$h_E \|v\|_{L^2(E)}^2 \lesssim \varepsilon^{-2} \| v\|^2_{L^2(T_E)} + \varepsilon^2 h_E^2 \|\nabla v\|^2_{L^2(T_E)},$$
we obtain, with $h_{E,\max}=\max_{E\subset \partial K} h_E$, the following
\begin{equation}
\|v_{\partial,h} \|_{L^2(K)}\lesssim \varepsilon^{-1}\|v\|_{L^2(K)} + \varepsilon h_{E,\max} \|\nabla v\|_{L^2(K)},
\end{equation}
which yields
\begin{equation}\label{aux:chen:2}
\|Q_K v\|_{L^2(K)}\lesssim (1+ \varepsilon^{-1})\|v\|_{L^2(K)} + \varepsilon h_{E,\max} \|\nabla v\|_{L^2(K)}.
\end{equation}
Let $v\in V_h(K)$ then it clearly holds
$$ \|\nabla v \|_{L^2(K)}=\inf_{w\in H^1(K):~w\vert_{\partial K}=v\vert_{\partial K}}  \|\nabla w \|_{L^2(K)},$$
which implies, by employing standard inverse inequality on (anisotropic) triangles the following 
\begin{align*}
\|\nabla v\|_{L^2(K)}^2 &\leq \|\nabla Q_K v\|_{L^2(K)}^2
=\sum_{T_E\in \Tau_K}  \|\nabla Q_K v\|_{L^2(T_E)}^2
\lesssim  \sum_{T_E\in \Tau_K}  \frac{1}{\lambda_{2,T_E}}\|Q_K v\|_{L^2(T_E)}^2\\
&\lesssim   \frac{1}{\lambda_{2,\Tau_K}}\|Q_K v\|_{L^2(K)}^2,
\end{align*}
where $\lambda_{2,\Tau_K}=\min_{T_E\in \Tau_K} \lambda_{2,T_E}$ and $\lambda_{2,T_E}$ is defined analogously to $\lambda_{2,K}$ (here $K=T_E$), cf \eqref{eq:eigenvalues}.
The above inequality combined with \eqref{aux:chen:2} yields
\begin{equation}
\|\nabla v\|_{L^2(K)}\lesssim \frac{1}{\sqrt{\lambda_{2,\Tau_K}}}\left( \left(1+\frac{1}{\varepsilon}\right)\|v\|_{L^2(K)}+ \varepsilon h_{E,\max} \|\nabla v\|_{L^2(K)}\right).
\end{equation}
By choosing {\color{black} $\varepsilon={\sqrt{\lambda_{2,\Tau_K}}}/{(2h_{E,\max})}$} we obtain
\begin{equation}
\|\nabla v\|_{L^2(K)}\lesssim \frac{\sqrt{\lambda_{2,\Tau_K}}+ h_{E,\max}}{\lambda_{2,\Tau_K}} \|v\|_{L^2(K)}.
\end{equation}
As $\sqrt{\lambda_{K,1}}\simeq h_{E,\max}$ and 
$\lambda_{2,\Tau_K}\simeq \lambda_{2,K}$ we get the thesis.
\end{proof}

\begin{theorem}\label{Thm:L2}
For any $K\subset \mathcal{K}_h$ there hold
\begin{equation}
\| u - \mathfrak I_C u\|^2_{L^2(K)} \lesssim \sum_{i=1}^{n_K}\sum_{\widetilde{K}\subset \omega_i} \frac{\vert K\vert}{\vert \widetilde{K}\vert} \|A_{\widetilde{K}}^{-T}\nabla u\|^2_{L^2(\widetilde{K})},
\end{equation}
or, written in an alternative way, 
\begin{equation}
\| u - \mathfrak I_C u\|^2_{L^2(K)} \lesssim \vert K\vert 
 \sum_{i=1}^{n_K}
 \sum_{\widetilde{K}\subset \omega_i} \sqrt{\frac{\lambda_{\widetilde{K},1}}{\lambda_{\widetilde{K},2}}}\vert u \vert^2_{H^1(\widetilde{K})},
\end{equation}
where $n_K$ denotes the number of vertices of $K$. 
\end{theorem}
\begin{proof}
Denoting by $n_K$ the number of vertices of $K$ {\color{black} and by $\omega_i$ the patch of elements sharing the $i$-th vertex of $K$} we have
\begin{eqnarray}
(u - \mathfrak I_C u)\vert_K &=& u\vert_K - \sum_{i=1}^{n_K} [r_1(u)]({\sf v}_i)  \varphi_i\vert_K - \sum_{i=2}^{n_K} [r_i(u)-r_1(u)]({\sf v}_i)\varphi_i\vert_K\nonumber\\
&=& 
(u - r_1(u))\vert_K  - \sum_{i=2}^{n_K} [r_i(u)-r_1(u)]({\sf v}_i)\varphi_i\vert_K,\nonumber
\end{eqnarray}
where in the last step we employed the fact that $r_1(u)$ is a virtual element function defined on the patch $F_K^{-1}(\widehat{\omega})$ and $K\subset F_K^{-1}(\widehat{\omega})$. It follows
\begin{eqnarray}
\| u - \mathfrak I_C u\|_{L^2(K)} &\leq& \| u - r_1(u) \|_{L^2(K)} 
+ 
\sum_{i=2}^{n_K} \vert [r_i(u)-r_1(u)]({\sf v}_i)\vert \|\varphi_i\|_{L^2(K)}
\nonumber\\
&\leq& \| u - r_1(u) \|_{L^2(K)} 
+ 
\vert K \vert^{1/2}
\sum_{i=2}^{n_K} \| r_i(u)-r_1(u)\|_{L^\infty(K)} .\nonumber
\end{eqnarray}
To conclude, it is enough to employ Lemma \ref{lemma1} in combination with the following bound
\begin{eqnarray}
\| r_i(u)-r_1(u)\|_{L^\infty(K)}&=&\| \widehat{r}_i(\widehat{u})-\widehat{r}_1(\widehat{u})\|_{L^\infty(\widehat{K})}
\nonumber\\
&\lesssim& \| \widehat{r}_i(\widehat{u})-\widehat{r}_1(\widehat{u})\|_{L^2(\widehat{K})}\lesssim 
\| \widehat{u}-\widehat{r}_i(\widehat{u})\|_{L^2(\widehat{K})}+
\| \widehat{u}-\widehat{r}_1(\widehat{u})\|_{L^2(\widehat{K})}\nonumber\\
&\lesssim& 
\vert K\vert^{-1/2}\left(
\| {u}-{r}_i({u})\|_{L^2({K})}+
\| {u}-{r}_1({u})\|_{L^2({K})}\right),
\end{eqnarray}
where in the first inequality we employed the fact that all norms are equivalent on the finite dimensional space $\widehat{\Theta}(\widehat{\omega})$. 
\end{proof}
\begin{corollary}
There holds
\begin{equation}\label{proj:L2-estimate}
\| u - \mathfrak I_C u\|^2_{L^2(K)} \lesssim \sum_{i=1}^{n_K}\sum_{\widetilde{K}\subset \omega_i} \|A_{\widetilde{K}}^{-T}\nabla u\|^2_{L^2(\widetilde{K})},
\end{equation}
where $n_K$ denotes the number of vertices of $K$. 
\end{corollary}
\begin{proof}
Thanks to the mesh regularity assumption it is possible to prove that  $\frac{\vert K\vert}{\vert \widetilde{K}\vert} $ is bounded.
\end{proof}
{\color{black}The following result will be obtained  under an assumption on the behaviour 
of the constant $\alpha_K$ appearing in the map \eqref{trafo:1}.
\begin{assumption}\label{assumption:2}
We assume that it holds $\alpha_K\simeq 1$ for every $K\in \mathcal{K}_h$, uniformly in $h$.
\end{assumption}
For a numerical exploration on the validity of Assumption~\ref{assumption:2} see \cite[Section 6.2]{Weisser:2019}. For future use, we note that the above assumption implies $\vert K\vert\simeq\sqrt{\lambda_{K,1}\lambda_{K,2}}$}.
\begin{theorem}\label{Thm:H1}
Under Assumption \ref{assumption:2}, for any $K\subset \omega_i$ there holds
{
\begin{equation}\label{interp:H1-estimate}
\vert u - \mathfrak I_C u\vert_{H^1(K)} \lesssim 
%\left(
\frac{\lambda_{K,1}}{\lambda_{K,2}}
%\right)
\, \vert u\vert_{H^1(\omega_K)},
\end{equation}
}
being $\omega_K$ the patch of polygons $K'$ such that $\overline{K'} \cap \overline{K}\not=\emptyset$.
\end{theorem}
\begin{proof}
Following the proof of Theorem \ref{Thm:L2} and employing Lemma \ref{lemma2} we have
\begin{eqnarray}
\vert u-  \mathfrak I_C u\vert_{H^1(K)}& \leq&
\vert u - r_1(u) \vert_{H^1(K)}
+ 
\sum_{i=2}^{n_K} \vert [r_i(u)-r_1(u)]({\sf v}_i)\vert 
\vert \varphi_i\vert_{H^1(K)} \nonumber\\
&\leq&  \left(\sum_{\widetilde{K}\subset \omega_1} 
\sqrt{\frac{\lambda_{K,1} \lambda_{\widetilde{K},1}}{\lambda_{K,2}\lambda_{\widetilde{K},2}}}\vert u \vert^2_{H^1(\widetilde{K})}\right)^{1/2}+ \sum_{i=2}^{n_K} \vert [r_i(u)-r_1(u)]({\sf v}_i)\vert 
\vert \varphi_i\vert_{H^1(K)} \nonumber.
\end{eqnarray}
{ Now, we observe that, similarly to the proof of Theorem \ref{Thm:L2}, employing \eqref{norm2} together with Lemma \ref{lemma1} the following holds
\begin{eqnarray}
\| r_i(u)-r_1(u)\|_{L^\infty(K)}&=&\| \widehat{r}_i(\widehat{u})-\widehat{r}_1(\widehat{u})\|_{L^\infty(\widehat{K})}
\nonumber\\
&\lesssim& \vert \widehat{r}_i(\widehat{u})-\widehat{r}_1(\widehat{u})\vert_{L^2(\widehat{K})}\lesssim 
\vert \widehat{u}-\widehat{r}_i(\widehat{u})\vert_{L^2(\widehat{K})}+
\vert \widehat{u}-\widehat{r}_1(\widehat{u})\vert_{L^2(\widehat{K})}\nonumber\\
&\lesssim& 
\vert K \vert^{1/2}
\left(
\vert {u}-{r}_i({u})\vert_{L^2({K})}+
\vert {u}-{r}_1({u})\vert_{L^2({K})}\right)\nonumber\\
&\lesssim& 
\left(\sum_{\widetilde{K}\subset \omega_1\cup\omega_i} \sqrt{\frac{\lambda_{\widetilde{K},1}}{\lambda_{\widetilde{K},2}}}\vert u \vert^2_{H^1(\widetilde{K})}\right)^{1/2}.\nonumber
\end{eqnarray}
Employing Lemma \ref{Lemma:inverse} and the fact that $\|\varphi\|_{L^2(K)}\leq \vert K\vert^{1/2}$ we have, {after invoking Assumption~\ref{assumption:2}}, $\vert \varphi_i \vert_{H^1(K)}\lesssim 
\left(\frac{\lambda_{K,1}}{\lambda_{K,2}}\right)^{3/4}$. Thus, we get 

$$
 \sum_{i=2}^{n_K} \vert [r_i(u)-r_1(u)]({\sf v}_i)\vert 
\vert \varphi_i\vert_{H^1(K)} 
\lesssim
 \left(\frac{\lambda_{K,1}}{\lambda_{K,2}}\right)^{3/4}\sum_{i=1}^{n_K} \left(\sum_{\widetilde{K}\subset \omega_i} 
\sqrt{\frac{\lambda_{\widetilde{K},1}}{\lambda_{\widetilde{K},2}}}\vert u \vert^2_{H^1(\widetilde{K})}\right)^{1/2}.
$$

As a consequence of the anisotropic mesh regularity assumption we have 
$$ \lambda_{K,i}\simeq \lambda_{\widetilde{K},i}, ~i=1,2$$
which yields

$$
 \sum_{i=2}^{n_K} \vert [r_i(u)-r_1(u)]({\sf v}_i)\vert 
\vert \varphi_i\vert_{H^1(K)} 
\leq 
 %\left(
 \frac{\lambda_{K,1}}{\lambda_{K,2}}
% \right)
 \, \vert u \vert_{H^1(\omega_{K})}.
$$

Combining the above inequalities we obtain the thesis.}

\end{proof}
\begin{theorem}\label{Thm:trace}
Let $E\subset \partial K$ be an edge of $K \in \mathcal{K}_h$.
Then it holds 
\begin{equation}\label{proj:L^2-estimate-edge}
\| u - \mathfrak I_C u\|_{L^2(E)}\lesssim \frac{\vert E\vert^{1/2}}{\vert K \vert^{1/2}}
 \|A_{\widetilde{K}}^{-T} \nabla u\|_{L^2(\omega_E)}
\end{equation}
being $\omega_E$ the patch of elements $\widetilde{K}$ having non-empty intersect with $\overline{E}$.
\end{theorem}
\begin{proof}
{\color{black} Let $E\subset \partial K$ be an edge of $K \in \mathcal{K}_h$ with endpoints $\sf{v}_1$ and  $\sf{v}_2$.  Observing that $r_1(u)$ is a virtual element function, we have }
\begin{eqnarray}
\| u -  \mathfrak I_C u\|_{L^2(E)}& = &
\| u - \sum_{i=1}^2 [r_i(u)]({\sf v}_i) \varphi_i\|_{L^2(E)}
\nonumber\\
&=& \| u - r_1(u) -[ r_2(u) - r_1(u)]({\sf v}_2) \varphi_2\|_{L^2(E)}\nonumber\\
&\leq & \| u - r_1(u)\|_{L^2(E)} + \| r_2(u) - r_1(u)\|_{L^\infty(E)}\|\varphi_2\|_{L^2(E)}.\nonumber
\end{eqnarray}
Noting that $\| \varphi_2\|_{L^2(E)}\leq \vert E\vert^{1/2}$ and observing that it holds
$\| r_2(u) - r_1(u)\|_{L^\infty(E)}= \| \widehat{r}_2(\widehat{u}) - \widehat{r}_1(\widehat{u})\|_{L^\infty(\widehat{E})}\lesssim 
\| \widehat{r}_2(\widehat{u}) - \widehat{r}_1(\widehat{u})\|_{L^2(\widehat{E})}
 \lesssim \vert E\vert ^{-1/2} \| {r}_2({u}) - {r}_1({u})\|_{L^2({E})} $ we have 
\begin{eqnarray}
\| u -  \mathfrak I_C u\|_{L^2(E)}& \lesssim &
 \| u - {r}_1({u})\|_{L^2({E})} +
\| {r}_2({u}) - {r}_1({u})\|_{L^2({E})} \nonumber\\
&\lesssim& 
\left(\frac{\vert E\vert}{\vert K\vert}\right)^{1/2}
\left( \| u -r_1(u)\|_{L^2(K)}  + \| r_2(u)-r_1(u)\|_{L^2(K)}
\right. \nonumber\\
&& \left. +\| A_K^{-T} \nabla(u-r_1(u))\|_{L^2(K)}  + \| A_K^{-T} \nabla(r_2(u)-r_1(u)) \|_{L^2(K)} \right),\nonumber
\end {eqnarray} 
 where we employed trace inequality \eqref{eq:AnisotropicTraceInequality}. Now, using 
  \eqref{eq1:lemma1} and
 \eqref{norm2}  we obtain
 \begin{eqnarray}
 \| u -  \mathfrak I_C u\|_{L^2(E)}&\lesssim&
 \left(\frac{\vert E\vert}{\vert K\vert}\right)^{1/2}
 \left( 
 \left(\sum_{\widetilde{K}\subset \omega_1\cup \omega_2} \frac{\vert K\vert}{\vert \widetilde{K}\vert} \|A_{\widetilde{K}}^{-T}\nabla u\|^2_{L^2(\widetilde{K})}\right)^{1/2}\right.\nonumber\\
 &&\left.  + \vert K\vert^{1/2} (\|\widehat{\nabla}(\widehat{u}-\widehat{r}_1(\widehat{u}))\|_{L^2(\widehat{K})}+ 
 \|\widehat{\nabla}(\widehat{u}-\widehat{r}_2(\widehat{u}))\|_{L^2(\widehat{K})})
 \right).\nonumber
 \end{eqnarray}
 Proceeding as in the proof of Lemma \ref{lemma2} and employing \eqref{norm2}
 we get for $i=1,2$
 $$  \|\widehat{\nabla}(\widehat{u}-\widehat{r}_i(\widehat{u}))\|_{L^2(\widehat{K})}\lesssim \vert \widehat{u}\vert_{H^1(\widehat{\omega})} \lesssim
\left( \sum_{\widetilde{K}\subset \omega_i} \vert \widetilde{K}\vert^{-1}
 \|A_{\widetilde{K}}^{-T} \nabla u\|^2_{L^2(\widetilde{K})}\right)^{1/2}.$$
Combining the above estimates and employing the anisotropic mesh assumptions guaranteeing $\vert K\vert \simeq \vert \widetilde{K}\vert$ for $K,\widetilde{K}$ belonging to the same patch, we obtain the thesis.
 \end{proof}

%%%%%%%%%%%%%%%%%%%%%%%%%%%%%%
%%%%%%%%%%%%%%%%%%%%%%%%%%%%%%%
In the following, we rewrite \eqref{proj:L2-estimate} and \eqref{proj:L^2-estimate-edge} in an equivalent way, suitable for future use when deriving in the next section polygonal anisotropic error estimates. 
More precisely, let~$\vec r_{K,i}$ be the normalized eigenvectors to the eigenvalues~$\lambda_{K,i}$ for $i=1,2$, which have been used already several times in the matrix~$U_K$. Namely, it is $U_K = (\vec r_{K,1},\vec r_{K,2})$. Thus, we observe
\[
 \Lambda_K^{1/2}U_K^\top\nabla v
 = \begin{pmatrix}
    \lambda_{K,1}^{1/2}\,\vec r_{K,1}\cdot\nabla v\\
    \lambda_{K,2}^{1/2}\,\vec r_{K,2}\cdot\nabla v
   \end{pmatrix},
\]
and consequently
\[
 \|\alpha_K^{-1}\Lambda_K^{1/2}U_K^\top\nabla v\|_{L_2(\omega_K)}^2
 = \alpha_K^{-2}\left(
    \lambda_{K,1}\,\|\vec r_{K,1}\cdot\nabla v\|_{L_2(\omega_K)}^2+
    \lambda_{K,2}\,\|\vec r_{K,2}\cdot\nabla v\|_{L_2(\omega_K)}^2
   \right).
\]
Furthermore, since $\vec r_{K,i}\cdot\nabla v\in\mathbb R$, we get
\begin{eqnarray*}
  \|\vec r_{K,i}\cdot\nabla v\|_{L_2(\omega_K)}^2
  & = & \sum_{K^\prime\subset\omega_K}\int_{K^\prime} (\vec r_{K,i}^\top\nabla v)^2\,d\xv\\
  & = & \sum_{K^\prime\subset\omega_K}\int_{K^\prime} \vec r_{K,i}^\top\nabla v(\nabla v)^\top\vec r_{K,i}\,d\xv\\
  & = & \vec r_{K,i}^\top\, G_K(v)\, \vec r_{K,i},
\end{eqnarray*}
with
\[
  G_K(v)
  = \sum_{K^\prime\subset\omega_K}
    \begin{pmatrix}
      \|v_{x_1}\|_{L_2(K^\prime)}^2 & \displaystyle\int_{K^\prime} v_{x_1}v_{x_2}\,d\xv \\
      \displaystyle\int_{K^\prime} v_{x_1}v_{x_2}\,d\xv & \|v_{x_2}\|_{L_2(K^\prime)}^2 
    \end{pmatrix}.
\]
Thus~\eqref{proj:L2-estimate} and \eqref{proj:L^2-estimate-edge} can be rewritten as 
\begin{equation}\label{C:1}
  \|v-\mathfrak I_Cv\|_{L_2(K)}
   \leq c\, \alpha_K^{-1}\left(
      \lambda_{K,1}\,\vec r_{K,1}^\top\, G_K(v)\, \vec r_{K,1}+
      \lambda_{K,2}\,\vec r_{K,2}^\top\, G_K(v)\, \vec r_{K,2}
     \right)^{1/2}
\end{equation}
and
\begin{equation}\label{C:2}
  \|v-\mathfrak I_Cv\|_{L_2(E)}
   \leq c\, \alpha_K^{-1}\,\,\frac{|E|^{1/2}}{|K|^{1/2}}\left(
      \lambda_{K,1}\,\vec r_{K,1}^\top\, G_K(v)\, \vec r_{K,1}+
      \lambda_{K,2}\,\vec r_{K,2}^\top\, G_K(v)\, \vec r_{K,2}
     \right)^{1/2},
\end{equation}
respectively (cf.~\cite[(2.12) and (2.15)]{FormaggiaPerotto2003}). 

{\color{black} We now introduce a variant of $\mathfrak I_C$ preserving the homogeneous boundary conditions. To this aim we number the $N$ vertices of the partition so that the first $N^\partial$ vertices are the boundary ones, while the remaining ones (i.e. from $N^\partial+1$ to $N$) are the internal vertices. 
The quasi-interpolant $\mathfrak I_{C,0}: H^1_0(\Omega) \to V_{h,0}$ is thus defined as
\begin{equation}
\mathfrak I_{C,0} u = \sum_{i=N^\partial+1}^N [r_i(u)]({\sf v}_i)\;\varphi_i(x).
\end{equation}
The results contained in Theorem \ref{Thm:L2}, \ref{Thm:H1} and \ref{Thm:trace} are still true, and the analogous estimates to \eqref{C:1} and \eqref{C:2} hold as well. For instance, in order to extend Theorem \ref{Thm:L2} it is sufficient to observe that for $u\in H^1_0(\Omega)$ the following holds true 
\begin{equation}
\|u-\mathfrak I_{C,0} u \|_{L^2(K)}\leq \|u-\mathfrak I_{C} u \|_{L^2(K)} +  \sum_{i=1}^{N^\partial} \vert [r_i(u)]({\sf v}_i)\vert \;\|\varphi_i(x)\|_{L^2(K)}.
\end{equation}
Finally, denoting by $E\subset \partial \Omega$ one of the two  boundary edges containing the vertex ${\sf v}_i$ and employing the norm equivalence on finite dimensional spaces together with $u\vert_{\partial \Omega}=0$ we have 
\begin{eqnarray}
\vert [r_i(u)]({\sf v}_i)\vert&\leq& \| r_i(u)\|_{L^\infty(E)}
=  \| \widehat{r}_i(\widehat{u})\|_{L^\infty(\widehat{E})}\lesssim 
\| \widehat{r}_i(\widehat{u})\|_{L^2(\widehat{E})} \nonumber\\
&=& \| \widehat{r}_i(\widehat{u}) - \widehat{u}\|_{L^2(\widehat{E})}
\leq  \| \widehat{r}_i(\widehat{u}) - \widehat{u}\|_{L^2(\widehat{K})} + \vert \widehat{r}_i(\widehat{u}) - \widehat{u}\vert_{H^1(\widehat{K})},\nonumber
\end{eqnarray}
where in the last step we used standard trace inequality. Finally, recalling that it holds 
$$ 
 \| \widehat{r}_i(\widehat{u}) - \widehat{u}\|_{L^2(\widehat{K})} + \vert \widehat{r}_i(\widehat{u}) - \widehat{u}\vert_{H^1(\widehat{K})} \lesssim 
 \vert \widehat{u} \vert_{H^1(\widehat{K})},
$$
the analogous to Theorem \ref{Thm:L2} follows after employing 
\eqref{norm2} and summing over the boundary vertices.
}

\section{Anisotropic a posteriori error estimate}\label{S:apos}
In this section we derive an anisotropic polygonal a posteriori error
estimate for the virtual element approximation of \eqref{pb}.
{We preliminary observe that in view of Lemma
  \ref{lem:MapH1norm} the stabilization form satisfies
  \begin{equation}\label{S1}
    \sqrt{\frac{\lambda_{K,2}}{\lambda_{K,1}}}\;  S^K(w_h,w_h)  \lesssim |w_h|_{H^1(K)}^2
    \lesssim \sqrt{\frac{\lambda_{K,1}}{\lambda_{K,2}}}\;S^K(w_h,w_h),
  \end{equation}
  for $w_h=(I-\Pi_1^{0,K})v_h$, $v_h\in V_h(K)$.  Indeed, it is
  sufficient to employ \eqref{norm3} in combination with the following
$$ 
S^K(w_h,w_h) \simeq \|w_h\|^2_{L^\infty(K)}=
\|\widehat{w}_h\|^2_{L^\infty(\widehat{K})}\simeq
\|\widehat{w}_h\|^2_{H^1(\widehat{K})},
$$
where we used the definition of $S^K$, the mesh assumption (in
particular the uniform boundedness of the number $n_K$ of element
vertices) and the fact that on finite dimensional spaces all norms are
equivalent.

\begin{remark}\label{remark:S}
  Let us discuss the sharpness of the bounds in \eqref{S1}. On the rectangle $K^\ast = (0,a)\times(0,b)$, with $a>b$, it can be
  easily seen that the virtual element basis functions in $V_h$ are
  \begin{align*}
    \varphi_0(x,y)  &= \frac{x\,y}{a\,b}-\frac{y}{b}-\frac{x}{a}+1 \,,
    &
      \varphi_1(x,y) &= \frac{x}{a}-\frac{x\,y}{a\,b} \,,
    \\
    \varphi_2(x,y) &= \frac{x\,y}{a\,b} \,,
    &
      \varphi_3(x,y) &= \frac{y}{b}-\frac{x\,y}{a\,b} \,,
  \end{align*}
  and, for $w_h=(I-\Pi_1^{0,K})   \varphi_i$, $i=0,\ldots,3$, the following holds:
  \begin{gather*}
    \lambda_{1,E} = \frac{a^2}{12}\,, \quad\lambda_{2,E} =
    \frac{b^2}{12} \,,
    \\ \frac{|w_h|_{H^1(K^\ast)}^2}{ S^{K^\ast}(w_h,w_h)}= \frac 1 3
    \left(\sqrt{\frac{\lambda_{K,1}}{\lambda_{K,2}}}+
    \sqrt{\frac{\lambda_{K,2}}{\lambda_{K,1}}}\right)\,,
      \end{gather*}
      and hence
    \begin{gather*}
    \frac23\sqrt{\frac{\lambda_{K,2}}{\lambda_{K,1}}}\;
    S^{K^\ast}(w_h,w_h) \leq |w_h|_{H^1(K^\ast)}^2 \leq
    \frac23\sqrt{\frac{\lambda_{K,1}}{\lambda_{K,2}}}\;S^{K^\ast}(w_h,w_h)
    \,.
  \end{gather*}
\end{remark}
}

We now state the main result of the paper.
\begin{proposition}\label{main:th}
  Let $u_h\in V_{h,0}$ be the VEM approximation to the solution $u$ of
  \eqref{pb}. {\color{black} Under Assumptions \ref{assumption} and
    \ref{assumption:2}}, for $e=u-u_h$ it holds
  \begin{eqnarray}
    \| \nabla e \|^2_{L^2(\Omega)} &\lesssim& 
                                              \sum_{K\in\mathcal{T}_h}
                                              \| R_K\|_{L^2(K)} \alpha_K^{-1}\left(
                                              \lambda_{K,1}\,\vec r_{K,1}^\top\, G_K(e)\, \vec r_{K,1}+
                                              \lambda_{K,2}\,\vec r_{K,2}^\top\, G_K(e)\, \vec r_{K,2}
                                              \right)^{1/2}\nonumber\\
                                   &+&\sum_{E\in\mathcal{S}_h} \| J_E\|_{L^2(E)}
                                       \alpha_K^{-1}\left( \frac{\vert E \vert}{\vert K \vert}\right)^{1/2}\left(
                                       \lambda_{K,1}\,\vec r_{K,1}^\top\, G_K(e)\, \vec r_{K,1}+
                                       \lambda_{K,2}\,\vec r_{K,2}^\top\, G_K(e)\, \vec r_{K,2}
                                       \right)^{1/2}\nonumber\\
                                   &+& \sum_{K\in\mathcal{T}_h} 
                                       M_K^2
                                       S^K((I-\Pi_1^{0,K})u_h, (I-\Pi_1^{0,K})u_h) \nonumber\\
                                   &+& \sum_{K\in\mathcal{T}_h} \|f-f_h\|_{L^2(K)} 
                                       \alpha_K^{-1}\left(
                                       \lambda_{K,1}\,\vec r_{K,1}^\top\, G_K(e)\, \vec r_{K,1}+
                                       \lambda_{K,2}\,\vec r_{K,2}^\top\, G_K(e)\, \vec r_{K,2}
                                       \right)^{1/2},
                                       \nonumber
  \end{eqnarray}
  where
  \begin{eqnarray}
    R_K&=&(f_h + \nabla\cdot\Pi_0^{0,K} \nabla u_h)_K,\nonumber\\
    J_E&=&\vert\![ \Pi_0^{0,E} \nabla u_h ]\!{\vert_E},\nonumber\\
    M_K&=&\left(\frac{\lambda_{K,1}}{\lambda_{K,2}}\right)^{{\frac 5 4}}.\nonumber
  \end{eqnarray}
\end{proposition}
\begin{proof}
  The proof closely follows \cite{Cangiani_et_al-apost:2017}.  Let us
  set $e=u-u_h\in H^1_0(\Omega)$ and we preliminary observe that for
  every $v\in H^1_0(\Omega)$ the following holds
  \begin{eqnarray}
    a(e,v)&=& (f,v)-a(u_h,\chi)-a(u_h,v-\chi)\nonumber\\
          &=&(f-f_h,\chi) + (f,v-\chi) + a_h(u_h,\chi)-a(u_h,\chi)-a(u_h,v-\chi)\label{aux:1}
  \end{eqnarray} 
  for all $\chi\in V_h$. Moreover, we observe that integration by
  parts yields
  \begin{equation}\label{aux:2}
    a(u_h,w)=-(\nabla\cdot \Pi_0^{0,h}\nabla u_h,w) + \sum_{E\in{S}_h} \int_E \vert\![ \Pi_0^{0,E} \nabla u_h ]\!\vert w ds +((I -  \Pi_0^{0,h}) \nabla u_h,\nabla w)
  \end{equation}
  for all $w\in H^1_0(\Omega)$, where
  $\Pi_0^{0,h}(\cdot)\vert_K=\Pi_0^{0,K}(\cdot)$ for every
  $K\in \mathcal{K}_h$.  Employing \eqref{aux:1}-\eqref{aux:2} we get
  \begin{eqnarray}
    a(e,v)&=& \sum_{K\in \mathcal{T}_h} \left( (R_k,v-\chi)_K + (\theta_K,v-\chi)_K + B_K(u_h,v-\chi)\right) - \sum_{E\in S_h} (J_E,v-\chi)_E \nonumber\\
          && + (f-f_h,\chi) + a_h(u_h,\chi)-a(u_h,\chi)
  \end{eqnarray}
  where
  \begin{eqnarray}
    R_K&=&(f_h + \nabla\cdot\Pi_0^{L^2} \nabla u_h)_K,\nonumber\\
    \theta_K&=&(f-f_h)_K,\nonumber\\
    B_K(w_h,v)&=&((I-\Pi_0^{0,K}) \nabla w_h, \nabla v)_K,\nonumber\\
    J_E&=&\vert\![ \Pi_0^{0,E} \nabla u_h ]\!{\vert_E}\nonumber.
  \end{eqnarray}

  {\color{black}Let $e_I= \mathfrak I_{C,0} e \in V_{h,0} $ be the
    quasi-interpolant of $e$ satisfying the analogous version to the
    estimates \eqref{C:1}-\eqref{C:2}}. Then we have
  \begin{eqnarray}
    \| \nabla e\|_{L^2(\Omega)}^2 &=& \sum_{K\in \mathcal{T}_h}\left\{
                                      (R_K,e-e_I)_{L^2(K)}+
                                      (\theta_K,e-e_I)_{L^2(K)}+
                                      (f-f_h,e_I)_{L^2(K)}\right. \nonumber\\
                                  && \left. +
                                     B_K(u_h,e-e_I)+
                                     (a_h^K(u_h,e_I)-a^K(u_h,e_I))
                                     \right\}
                                     -\sum_{E\in\mathcal{S}_h} (J_E,e-e_I)_{L^2(E)}\nonumber\\
                                  &=:& \sum_{K\in \mathcal{T}_h} \left( {\tt I+II + III + IV + V} \right)
                                       - \sum_{E\in\mathcal{S}_h} {\tt VI}.
  \end{eqnarray}
  Let us now estimate the above terms. Employing Cauchy-Schwarz
  inequality together with {\color{black} the analogous estimate to
    \eqref{C:1}} we have
  \begin{equation}\label{I} {\tt I}\lesssim \| R_K\|_{L^2(K)}
    \alpha_K^{-1}\left(
      \lambda_{K,1}\,\vec r_{K,1}^\top\, G_K(e)\, \vec r_{K,1}+
      \lambda_{K,2}\,\vec r_{K,2}^\top\, G_K(e)\, \vec r_{K,2}
    \right)^{1/2}.
  \end{equation}
  Similarly, employing {\color{black} the analogous estimate to
    \eqref{C:2}}, we have
  \begin{equation}\label{VI} {\tt VI}\lesssim \| J_E\|_{L^2(E)}
    \alpha_K^{-1}\left(\frac{\vert E\vert }{\vert K\vert }\right)^{1/2}\left(
      \lambda_{K,1}\,\vec r_{K,1}^\top\, G_K(e)\, \vec r_{K,1}+
      \lambda_{K,2}\,\vec r_{K,2}^\top\, G_K(e)\, \vec r_{K,2}
    \right)^{1/2}.
  \end{equation}
  % Similarly, we have
%$${\tt II} \leq \| f-f_h \|_{L^2(K)} \alpha_K^{-1}\left(
%  \lambda_{K,1}\,\vec r_{K,1}^\top\, G_K(e)\, \vec r_{K,1}+
%  \lambda_{K,2}\,\vec r_{K,2}^\top\, G_K(e)\, \vec r_{K,2}
% \right)^{1/2}.$$
  Combining ${\tt II}$ and ${\tt III}$ yields
  \begin{eqnarray*}
    {\tt II} + {\tt III} 
    &  = & (f-f_h,e)_K 
           =   (f-f_h,e-\Pi_0^{0,K}e)_K \\
    &\leq& \|f-f_h\|_{L^2(K)} \|e-\Pi_0^{0,K}e\|_{L^2(K)}\\
    &\lesssim& \|f-f_h\|_{L^2(K)} \|\alpha_K^{-1}\Lambda_K^{1/2}U_K^\top\nabla e\|_{L^2(K)}\\
    &  = & \|f-f_h\|_{L^2(K)} \alpha_K^{-1}\left(
           \lambda_{K,1}\,\vec r_{K,1}^\top\, G_K(e)\, \vec r_{K,1}+
           \lambda_{K,2}\,\vec r_{K,2}^\top\, G_K(e)\, \vec r_{K,2}
           \right)^{1/2}
  \end{eqnarray*}
  according to Lemma~\ref{lem:BestApproxByConst}.  We now focus on the
  term ${\tt IV}$. Noting that it holds
  \begin{equation}\label{aux:3}
    \|(I-\Pi_0^{0,K})\nabla u_h \|_{L^2(K)} = 
    \|(I-\Pi_0^{0,K})\nabla (I - \Pi_1^{0,K})u_h \|_{L^2(K)}
    \leq \| \nabla (I - \Pi_1^{0,K})u_h \|_{L^2(K)}  
  \end{equation}
  and remembering \eqref{S1} we have
  \begin{eqnarray} {\tt IV} &\leq& \| \nabla(I-\Pi_1^{0,K}) u_h
    \|_{L^2(K)}
                                   \| \nabla(e-e_I)\|_{L^2(K)}\nonumber\\
                            &\leq&
                                   \left({\frac{\lambda_{K,1}}{\lambda_{K,2}}}\right)^{\frac
                                   1 4} \left(S^K( (I-\Pi_1^{0,K}) u_h
                                   , (I-\Pi_1^{0,K}) u_h
                                   )\right)^{\frac 1 2}
                                   \| \nabla(e-e_I)\|_{L^2(K)} \nonumber\\
                            &\leq&   \left({\frac{\lambda_{K,1}}{\lambda_{K,2}}}\right)^{\frac 5 4} \left(S^K( (I-\Pi_1^{0,K}) u_h , (I-\Pi_1^{0,K}) u_h )\right)^{\frac 1 2} 
                                   \| \nabla e\|_{L^2(\omega_K)}
  \end{eqnarray}
  where in the last step we {\color{black} employed the analogous
    version to \eqref{interp:H1-estimate} for $\mathfrak{I}_{C,0}$}.
  Finally, we consider the term ${\tt V}$. It is immediate to verify
  that it holds
  \begin{equation}\label{V} {\tt V}=
    \underbrace{((I-\Pi_0^{0,K})\nabla u_h,\nabla e_I )_K}_{\tt{V}_1}   %$$}
    - 
    \underbrace{S^K((I-\Pi_1^{0,K})u_h, (I-\Pi_1^{0,K})e_I}_{\tt{V}_2}.  %% ${\tt V}_2$
  \end{equation} 
  Employing the Cauchy-Schwarz inequality together with \eqref{aux:1}
  and {\color{black} the analogous estimate to
    \eqref{interp:H1-estimate}} we have
  \begin{eqnarray}
    |{\tt V}_1|&\leq& \| (I-\Pi_0^{0,K})\nabla u_h \|_{L^2(K)} \| \nabla e_I\|_{L^2(K)}\nonumber\\
               &=&\| (I-\Pi_0^{0,K})\nabla (I-\Pi_1^{0,K})u_h \|_{L^2(K)} \| \nabla e_I\|_{L^2(K)}
                   \nonumber\\
               &\leq&  \| \nabla (I-\Pi_1^{0,K})u_h \|_{L^2(K)} \| \nabla e_I\|_{L^2(K)}\nonumber\\
               &\leq& 
                      \left({\frac{\lambda_{K,1}}{\lambda_{K,2}}}\right)^{\frac 1 4} \left(S^K( (I-\Pi_1^{0,K}) u_h , (I-\Pi_1^{0,K}) u_h )\right)^{\frac 1 2} \| \nabla e_I\|_{L^2(K)}\nonumber\\
               &\leq& 
                      2\left({\frac{\lambda_{K,1}}{\lambda_{K,2}}}\right)^{\frac 5 4} \left(S^K( (I-\Pi_1^{0,K}) u_h , (I-\Pi_1^{0,K}) u_h )\right)^{\frac 1 2} \| \nabla e\|_{L^2(\omega_K).}\nonumber
  \end{eqnarray}

  Employing \eqref{S1} and \eqref{interp:H1-estimate} we have
  \begin{eqnarray}
    |{\tt V}_2| &\leq& 
                       \left({\frac{\lambda_{K,1}}{\lambda_{K,2}}}\right)^{\frac 1 4}
                       S^K((I-\Pi_1^{0,K})u_h, (I-\Pi_1^{0,K})u_h)^{\frac 1 2 }
                       \|\nabla(I-\Pi_1^{0,K})e_I\|_{L^2(K)} 
                       \nonumber\\
                &\leq& \left({\frac{\lambda_{K,1}}{\lambda_{K,2}}}\right)^{\frac 1 4}
                       S^K((I-\Pi_1^{0,K})u_h, (I-\Pi_1^{0,K})u_h)^{\frac 1 2 } 
                       \| \nabla e_I\|_{L^2(K)}\nonumber\\
                &\leq& 2\left({\frac{\lambda_{K,1}}{\lambda_{K,2}}}\right)^{\frac 5 4}
                       S^K((I-\Pi_1^{0,K})u_h, (I-\Pi_1^{0,K})u_h)^{\frac 1 2 } 
                       \| \nabla e\|_{L^2(\omega_K)}.\nonumber
  \end{eqnarray}
  Hence,
  \begin{eqnarray} {\tt V}&\leq& 4
    \left(\frac{\lambda_{K,1}}{\lambda_{K,2}}\right)^{\frac 5 4}
    S^K((I-\Pi_1^{0,K})u_h, (I-\Pi_1^{0,K})u_h)^{\frac 1 2 } \| \nabla
    e\|_{L^2(\omega_K)}.\nonumber
  \end{eqnarray}
  % \Green{%
  % Maybe we can also combine ${\tt IV}$ and ${\tt V}$. This yields
  % \begin{equation*}
  %   {\tt IV}+{\tt V}=((I-\Pi_0^{L^2})\nabla u_h,\nabla e ) -
  %   S^K((I-\Pi_1^{L^2})u_h, (I-\Pi_1^{L^2})e_I)=:\widetilde{\tt
  %   V}_1+{\tt V}_2,
  % \end{equation*}
  % compare~\eqref{V}. We may proceed as above and employ the
  % stability $\|\nabla e_I\|_{L^2(K)}\lesssim \|\nabla
  % e\|_{L^2(K)}$ (to be proved).  } {\color{black} [M: Yes, good
  % idea. But the final result is still the second term on the right
  % in the bound of Proposition 4. Correct?]}  \Blue{[S: Yes and no! I
  % think the result is the same, however, in this case we don't need
  % the estimate $\| \nabla(e-e_I)\|_{L^2(K)}\leq \| \nabla
  % e\|_{L^2(K)}$ used to bound ${\tt
  % IV}$. Therefore, we only have to prove the first item bolow!?]}
  % {\color{black}[M: right!]}
%
%
  % {\color{black} [M: Summarizing, so far we need to prove:
  % \begin{itemize}
  % \item $\|\nabla e_I\|_{L^2(K)}\lesssim \|\nabla e\|_{L^2(K)}$
  % \item $\| \nabla(e-e_I)\|_{L^2(K)}\leq \| \nabla
  %   e\|_{L^2(K)}$. \Blue{\quad\mbox{[S: Maybe not needed, see
  %   above.]}} {\color{black}[M: right!]}
  % \end{itemize}
  % ] }of}
  Using that the cardinality of
  $\omega_K$ is uniformly bounded (i.e. the number of edges of each
  polygon is uniformly bounded) yield the thesis.
\end{proof}
\begin{remark}\label{rem:exponent}
A close inspection of the proof of Proposition \ref{main:th}  reveals that the presence of the factor $5/4$ in  the term $M_K$ is related to the combined use of \eqref{S1} and \eqref{interp:H1-estimate}. While the bounds in \eqref{S1} are optimal (see Remark \ref{remark:S}), we conjecture that the estimate  \eqref{interp:H1-estimate} is suboptimal in view of the presence of the factor $\lambda_{K,1}/\lambda_{K,2}$. The presence of the factor $5/4$ has a relevant effect on the performance of the adaptive refinement procedure (see Section \ref{sec:numres} for further comments).
\end{remark}
In the following, we focus on the computation of
$G_K(e)$. In order to deal with this term we employ Zienkiewicz-Zhu
(ZZ) error estimator which yields
\begin{equation}\label{G_K:approx}
  G_K(e) \simeq \sum_{K^\prime\subset\omega_K}
  \begin{pmatrix}
    \displaystyle\int_{K^\prime} (\eta_1^{ZZ}(u_h))^2  \,d\xv  & \displaystyle\int_{K^\prime}  \eta_1^{ZZ}(u_h) \eta_2^{ZZ}(u_h) \,d\xv \\
    \displaystyle\int_{K^\prime} \eta_1^{ZZ}(u_h) \eta_2^{ZZ}(u_h)
    \,d\xv & \displaystyle\int_{K^\prime} (\eta_2^{ZZ}(u_h))^2 \,d\xv
  \end{pmatrix},
\end{equation}
with
$\eta_i^{ZZ}(u_h)$ to be properly defined.  Here, we
take $$\eta_i^{ZZ}(u_h)=(I-\Pi_h^{ZZ}) (\partial_{x_i} (\Pi_1^{0,h}
u_h)),$$ $i=1,2$ where for every vertex ${\sf v}$ of $K^\prime$ we set
$$ 
\Pi_h^{ZZ} (\partial_{x_i} (\Pi_1^{0,h} u_h))({\sf v})=
\frac{1}{\sum_{K^{''}: {\sf v}\in K^{''}} \vert K^{''} \vert }
\sum_{K^{''}: {\sf v}\in K^{''}} \vert K^{''} \vert \partial_{x_i}
(\Pi_1^{0,K^{''}} u_h)_{\vert K^{''}}.
$$
We employ the above vertex values $\Pi_h^{ZZ} (\partial_{x_i}
(\Pi_1^{0,h} u_h))({\sf
  v})$ to construct, via e.g. least square fitting, a linear
polynomial on $K^\prime$ that we denote by $\Pi_h^{ZZ} (\partial_{x_i}
(\Pi_1^{0,h} u_h))_{\vert
  K^\prime}$. This latter enters in the construction of
$\eta_i^{ZZ}(u_h)$ and thus it is employed to approximate $G_K(e)$.

% mainfile: paper_30102019
\newcommand{\refmaxmom}{RefMaxMomentum}
\newcommand{\refanis}{AnisotropicRef}

\section{Numerical Results}
\label{sec:numres}
\newcommand{\stt}{\eta_h}
\newcommand{\sth}{\eta_h^{\mathrm{heur}}}
In this section we assess the behavior of the error estimate on three
test cases.  We recall that, according to the previous section (cf.\ Proposition \ref{main:th}), the estimator is
defined as:
\begin{equation}\label{eq:def_est_theory}
  \stt = 
  \left(
    \sum_{K\in\mathcal{T}_h} \eta_{K}^2 + \sum_{E\in\mathcal{S}_h}
    \xi_E^2 + \sum_{K\in\mathcal{T}_h} \sigma_{K}^2
  \right)^{1/2} \,,
\end{equation}
where

\begin{eqnarray*}
  \eta_K^2 &=& \| R_K\|_{L^2(K)} \alpha_K^{-1}\left(
             \lambda_{K,1}\,\vec r_{K,1}^\top\, G_K(e)\, \vec r_{K,1}+
             \lambda_{K,2}\,\vec r_{K,2}^\top\, G_K(e)\, \vec r_{K,2}
             \right)^{1/2} ,
  \\
  \xi_E^2 &=& \| J_E\|_{L^2(E)} \max_{\substack{K\colon E\subset \partial K}} 
  \alpha_K^{-1}\left( \frac{\vert E \vert}{\vert K
  \vert}\right)^{1/2}\left( \lambda_{K,1}\,\vec r_{K,1}^\top\,
  G_K(e)\, \vec r_{K,1}+ \lambda_{K,2}\,\vec r_{K,2}^\top\,
  G_K(e)\, \vec r_{K,2} \right)^{1/2} ,
\end{eqnarray*}
and $\sigma_K$ is given by
\begin{eqnarray}
  \tilde{\sigma}_K^2 &=& S^K((I-\Pi_1^{0,K})u_h, (I-\Pi_1^{0,K})u_h) \,,
                         \label{eq:def_est_theory_sigmatilde}
  \\
  \sigma_K^2 &=& M_K^2
%                 \left(\frac{\lambda_{K,1}}{\lambda_{K,2}}\right)^{\frac52}
                 \tilde{\sigma}_K^2 \,.
                 \label{eq:def_est_theory_sigma}
\end{eqnarray}

To highlight the advantage of using an anisotropic adaptive process,
we compare our results with the ones obtained by refining the mesh
with the following isotropic error estimate:
\begin{equation}\label{eq:def_est_iso}
  \eta^{\mathrm{iso}}_h = \left( \sum_{K\in\mathcal{T}_h}
    h^2_K \| R_K\|_{L^2(K)}^2 + \sum_{E\in\mathcal{S}_h}
    h_E \| J_E\|_{L^2(E)}^2
    + \sum_{K \in \mathcal{T}_h} \tilde{\sigma}_K^2
  \right)^{\frac12},
\end{equation}
see \cite{Berrone-Borio:2017,Cangiani_et_al-apost:2017}.\\
We also introduce the following \emph{heuristically scaled estimator}
\begin{equation}\label{eq:def_est_heur}
  \sth = \left(\sum_{K\in\mathcal{T}_h} \eta_{K}^2 + \sum_{E\in\mathcal{S}_h}
    \xi_E^2 + \sum_{K\in\mathcal{T}_h} \tilde{\sigma}_{K}^2\right)^{\frac12} \,,
\end{equation}
that differs from $\stt$ for the presence of the unscaled stabilization terms $\tilde{\sigma}_K^2$ (cf. Remark \ref{rem:exponent}).

In all the test cases, we consider \eqref{pb} with $\Omega =
(0,1)\times(0,1)$. All the three proposed tests have a boundary layer
and are solved using VEM of order $1$ and $2$. { We remark that the extension of the anisotropic a posteriori framework developed in the previous section to the case of VEM of order $2$ (see \cite{volley} for details on the definition of the approximation spaces) is straightforward.} 
In the first two test cases the solution is purely anisotropic while in the last test there
is both an isotropic structure and an anisotropic layer. Before presenting the results of the computations, we describe in detail the cell refinement strategy. {To simplify the implementation of the anisotropic mesh refinement process, that in presence of very general elements may become computationally demanding, we restrict ourselves to convex elements. The efficient implementation of the mesh refinement process in the general case is under investigation.}

%In the first test the boundary layer is in the upper right corner with a strong gradient in
%both directions. In the second test, instead, the boundary layer is on
%the right side of the domain, so the gradient has only a large
%component along the $x$ direction. The last one is the union of a
%boundary layer on the right side of the domain and a bubble function
%in the center to mix both an anisotropic and an isotropic solution.

\subsection{Cell refinement strategy}
{
The anisotropic adaptive VEM hinges upon the classical paradigm
$$
{\ldots} \to {\tt SOLVE} \to {\tt ESTIMATE} \to {\tt MARK} \to {\tt REFINE} \to
{\ldots}$$  
The module {\tt MARK} is based on 
the D\"orfler strategy, see, e.g., \cite{Nochetto-Veeser:2012} for more details. All numerical tests have been run with marking parameter equal to $1/2$.

In the sequel we focus on the description of the module {\tt REFINE}, see Algorithm \ref{alg:Refinement} below. We aim at designing a refinement strategy that reduces the size of the element along the direction of the
gradient of the error (thus cutting in the orthogonal direction), while preventing an unnecessary increase of the aspect ratio of the polygon. The approach here applied extends the strategy presented in \cite{Berrone-Borio-DAuria:2019} {(cf., e.g.,  \cite{Anisotropic_refinement} and \cite{Houston:anisDG:2007b} for refinement strategies in the case of triangular and quadrilateral elements, respectively).}

More precisely, for each marked polygon $K$,  we compute: (a) the eigenvalues $\lambda_{K,1}$ and
$\lambda_{K,2}$
($\lambda_{K,1}\ge\lambda_{K,2}$) of the covariance matrix
$M_{\rm Cov}=M_{\rm Cov}(K)$ together with the corresponding
eigenvectors $\vec{r}_{K,i}$, $i=1,2$;  (b) the eigenvalues  $\lambda_{G,1}$ and $\lambda_{G,2}$ ($\lambda_{G,1} \ge \lambda_{G,2}$) of the matrix $G_K=G_K(e)$ together with the corresponding eigenvectors
$\vec{r}_{G,i}$, $i=1,2$. The matrix $G_K$ is computed by resorting to the ZZ approximation, cf. \eqref{G_K:approx} {for VEM of order $1$ (the case of VEM of order $2$ simply requires the use of a quadratic least square fitting)}.
We notice that  large values of $\lambda_{G,1}/\lambda_{G,2}$ indicate a local anisotropic behaviour of the gradient of the error, whereas large values of $\lambda_{K,1}/\lambda_{K,2}$ are associated to anisotropic elements.  

If $\left(\lambda_{G,1}/\lambda_{G,2}\right) \ge
\left(\lambda_{K,1}/\lambda_{K,2}\right)$, then the refinement strategy cuts the polygon $K$ along $\vec{r}_{G,2}$, otherwise it cuts along $\vec{r}_{K,2}$.

Whenever $\lambda_{G,1}/\lambda_{G,2}\ge \lambda_{K,1}/\lambda_{K,2}$ the refinement strategy takes advantage of the pronounced anisotropic behaviour of the gradient of the error, whereas if $\lambda_{K,1}/\lambda_{K,2}$ dominates then the aspect ratio of the element is reduced. Heuristically speaking, when $\lambda_{K,1}/\lambda_{K,2}$ dominates on $\lambda_{G,1}/\lambda_{G,2}$ the module {\tt REFINE} aims at identifying a situation where the anisotropy of the element is too pronounced (and possibly unnecessary) with respect to the (anisotropic) behaviour of the gradient of the error. A typical situation could be the presence of an anisotropic element in a region where an isotropic refinement is needed (i.e. the gradient of the error does not exhibit any preferential direction).

%Le us consider the spectral decomposition of the matrix
%$
%G_K=\lambda_{G,1}\vec{r}_{G,1}\vec{r}_{G,1}^\top
%+\lambda_{G,2}\vec{r}_{G,2}\vec{r}_{G,2}^\top.
%$
%And let us notice that if the gradient of the error is constant on the element $\lambda_{G,1}=||\nabla e||^2$ and
%$\vec{r}_{G,1}$ is the unit vector in the direction of $\nabla e$, the cut in the direction of $\vec{r}_{G,2}$
%cuts the element in a direction orthogonal to the gradient of the error reducing the factor $H_{G,K}^2.$
%The factor $H_{G,K}=\left(
%             \lambda_{K,1}\,\vec r_{K,1}^\top\, G_K(e)\, \vec r_{K,1}+
%             \lambda_{K,2}\,\vec r_{K,2}^\top\, G_K(e)\, \vec r_{K,2}
%             \right)^{1/2}\,.$ can be written as
%
%$$H_{G,K}^2
%=
%\lambda_{K,2}\lambda_{G,2}
%\left[
%  \frac{\lambda_{K,1}}{\lambda_{K,2}}\frac{\lambda_{G,1}}{\lambda_{G,2}}
%  (\vec r_{K,1}^\top\vec r_{G,1})^2
%  +\frac{\lambda_{K,1}}{\lambda_{K,2}}
%  (\vec r_{K,1}^\top\vec r_{G,2})^2
%  +\frac{\lambda_{G,1}}{\lambda_{G,2}}
%  (\vec r_{K,2}^\top\vec r_{G,1})^2
%  +(\vec r_{K,2}^\top\vec r_{G,2})^2\right].
%$$

Finally, we remark that whenever the estimator
$\eta^{\mathrm{iso}}_h$ is used to drive the adaptive procedure, the marked polygon is always cut along the direction $\vec{r}_{K,2}$.

\begin{algorithm}
  \caption{The Module {\tt REFINE}}
  \begin{flushleft}
    Given a marked cell $K$
  \end{flushleft}
  \begin{algorithmic}[1]
    \STATE Compute the barycenter $\bar \xv_K$
    \STATE Compute the
    tensor $G_K$ 
    \STATE Compute the covariance matrix $M_{\rm Cov}$ 
    \STATE Compute the eigenvalues of the two tensors
    \IF{$\left(\lambda_{G,1}/\lambda_{G,2}\right)\ge\left(\lambda_{K,1}/\lambda_{K,2}\right)
      $} 
     	\STATE  Build a straight line passing  through $\bar \xv_K$ and parallel to 
    $\vec{r}_{G,2}$ \label{alg:isotropic}
    	\ELSE\STATE Build a straight
    line passing through $\bar \xv_K$ and parallel to
    $\vec{r}_{K,2}$
    	 \ENDIF
    \STATE Refine the cell
  \end{algorithmic}
  \label{alg:Refinement}
\end{algorithm}

In the following test cases we employ
$$\tilde{e} = \| \nabla(u - \Pi^0_k u_h) \|_{\Omega} \qquad k=1,2$$
 as a
measure of the exact error and we iterate the adaptive process until
$\tilde{e} \leq 10^{-3}$.  }
%%%%%%%%%%%%%%%%%%%%%%%%%%%%%%%%%%%%%
\subsection{Test case 1}

In the first test we set the forcing term in such a way that the exact
solution is given by
\begin{equation*}
  u(x,y) = 10^{-6}x(1-x)(1-y)(\mathrm{e}^{10x} - 1)(\mathrm{e}^{10y} -1) \,.
\end{equation*}
In Figure~\ref{fig:test1-sol} we plot the exact solution, that
displays a peak in the top-right corner of the domain, with boundary
layers in the $x$ and $y$ directions.
%%%%%%%%%
\begin{figure}
  \centering \includegraphics[scale=0.7]{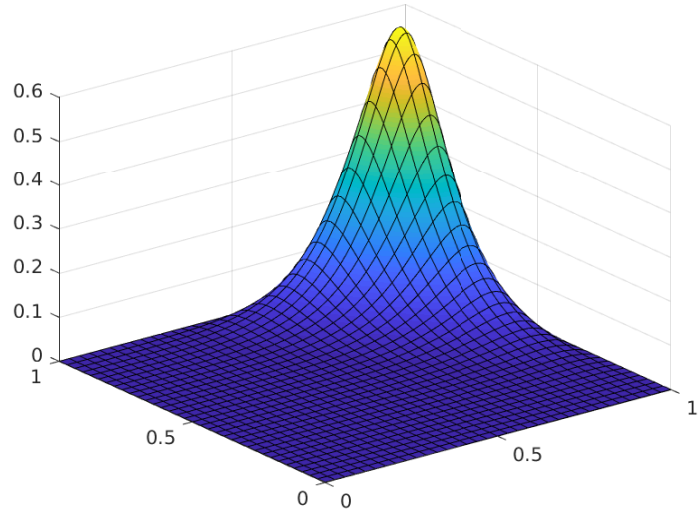}
  \caption{Test case 1. Plot of the exact solution}
  \label{fig:test1-sol}
\end{figure}
%%%%%%%%%
In Figure~\ref{fig:test1-estimator-th} we report the behavior
of the components of the error estimator $\stt$ defined as in \eqref{eq:def_est_theory} when
the adaptive process is run
using $\stt$ as estimator, based on employing both VEM of order 1 (cf. Figure~\ref{fig:test1-analysis:order1:estimcompeta2}) and of order 2 (cf. Figure~\ref{fig:test1-analysis:order2:estimcompeta2}).
%%%%%%%%%
\begin{figure}
  \centering
  \begin{subfigure}{.49\linewidth}
    \centering \includegraphics[width =
    \linewidth]{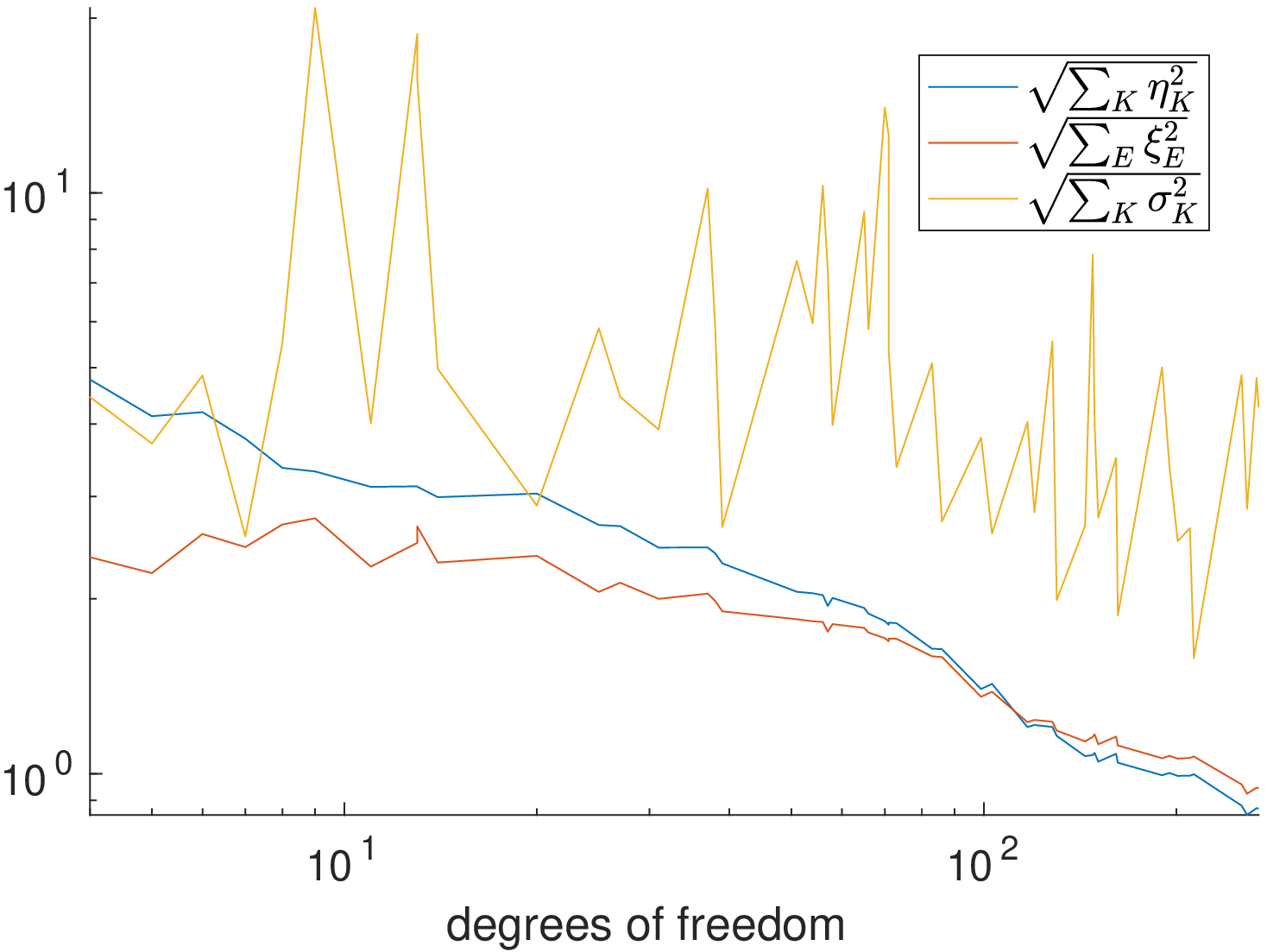}
    \caption{VEM of order 1.}
    \label{fig:test1-analysis:order1:estimcompeta2}
  \end{subfigure}
  \hfill
  \begin{subfigure}{.49\linewidth}
    \centering \includegraphics[width =
    \linewidth]{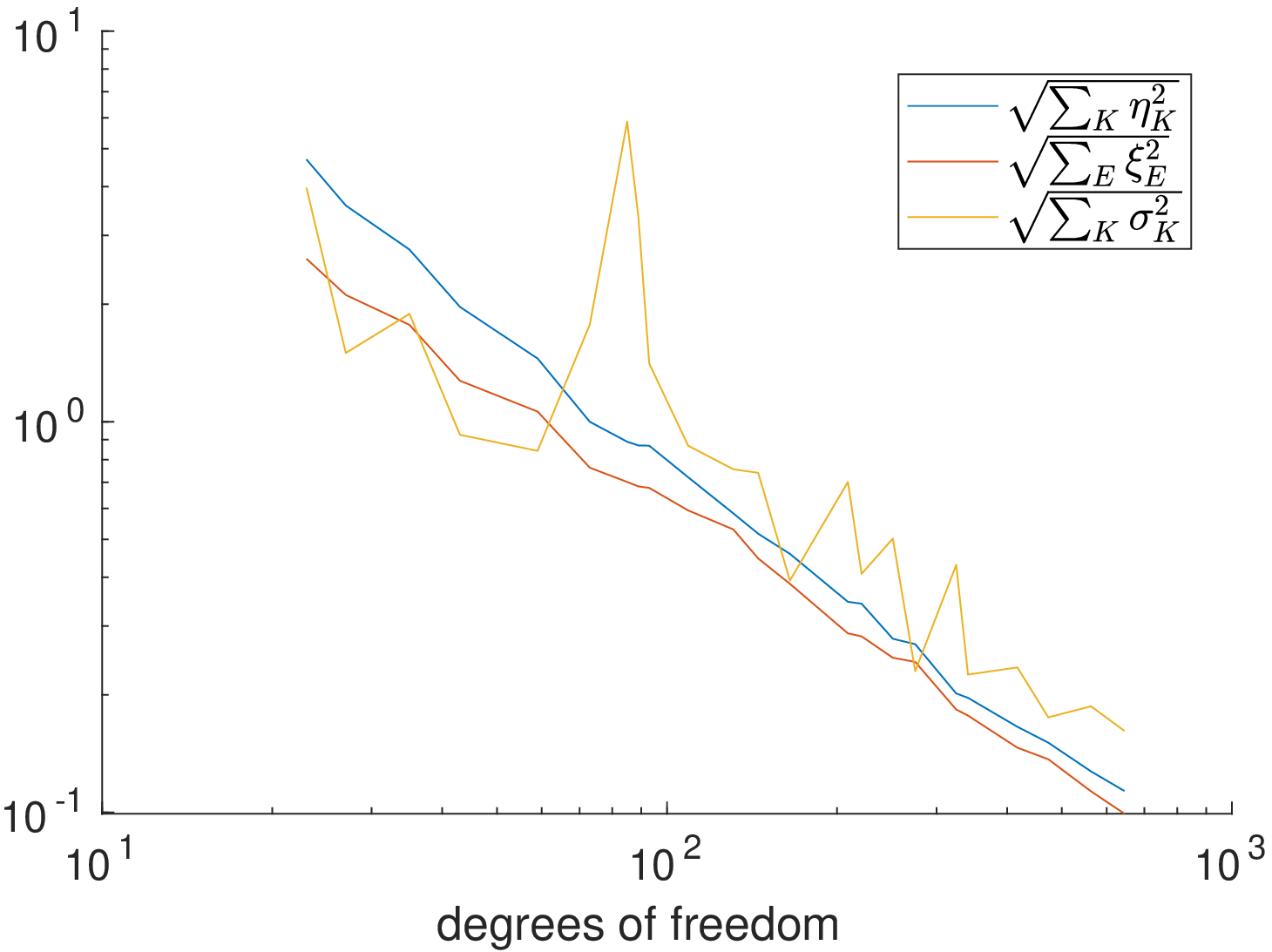}
    \caption{VEM of order 2.}
    \label{fig:test1-analysis:order2:estimcompeta2}
  \end{subfigure}
  \caption{Test case 1. Components of the error estimator $\stt$ as a function of the number of degrees of freedom, when the adaptive process is driven by the estimator $\stt$.}
  \label{fig:test1-estimator-th}
\end{figure}
%%%%%%%%%
From the numerical results reported in Figure~\ref{fig:test1-estimator-th}, we can infer that the three components of the estimator $\stt$ differ one from each other of an order of magnitude.  This
behavior is more evident for VEM of order 1, since the larger number
of refinements produces large anisotropic elements. 
A closer inspection, reveals that the factor $M_K$ could be much  larger than the other factors, and thus  the aspect ratio of polygons could largely increase during the anisotropic adaptive process and can drive the adaptive process itself.
The large aspect ratio of the elements produced is amplified by the exponent $\frac54$
that is present in $M_K$ that makes the VEM stabilization dominant
with respect to the other components of the estimator (cf. Remark \ref{rem:exponent}). 
The resulting behavior of the estimators might then be not optimal because the large values of
$\sigma_K$ cause the refinement process to concentrate on cells whose
error is already small.
Indeed, as we show in Figure
\ref{fig:test1-analysis-sigmaK}, the behavior of $\sigma_K$
when the refinement process is driven by $\sth$ tends to be parallel 
to the one of $\eta_K$ and $\xi_E$,
showing that the behavior of the estimate $\stt$ is
caused by the fact that the estimator $\sigma_K$ is selecting for
refinement elements that should not be selected.
%For the rest of the section we present the results obtained with $\sth$.
Notice that the above definition of the $\tilde{\sigma}_K$ estimator
{ used in $\sth$}
correspond to the one given in Proposition \ref{main:th},
setting $M_K = 1$.\\

%{\color{blue} non è una ripetizione?}
%To support further the validity of the \emph{heuristically scaled estimator}, we have repeated the same test as before, and now 
%in Figure~\ref{fig:test1-analysis-sigmaK} we report the computed values of $\eta_{K}$,  $\xi_E$, $\sigma_K$ as a function of the number of degrees of freedom based on employing the \emph{heuristically scaled estimator} $\sth$ defined in \eqref{eq:def_est_heur} to drive the adaptive algorithm.\\
%%%%%%%%%
\begin{figure}
  \centering
  \begin{subfigure}{.49\linewidth}
    \centering
    \includegraphics[width=\linewidth]{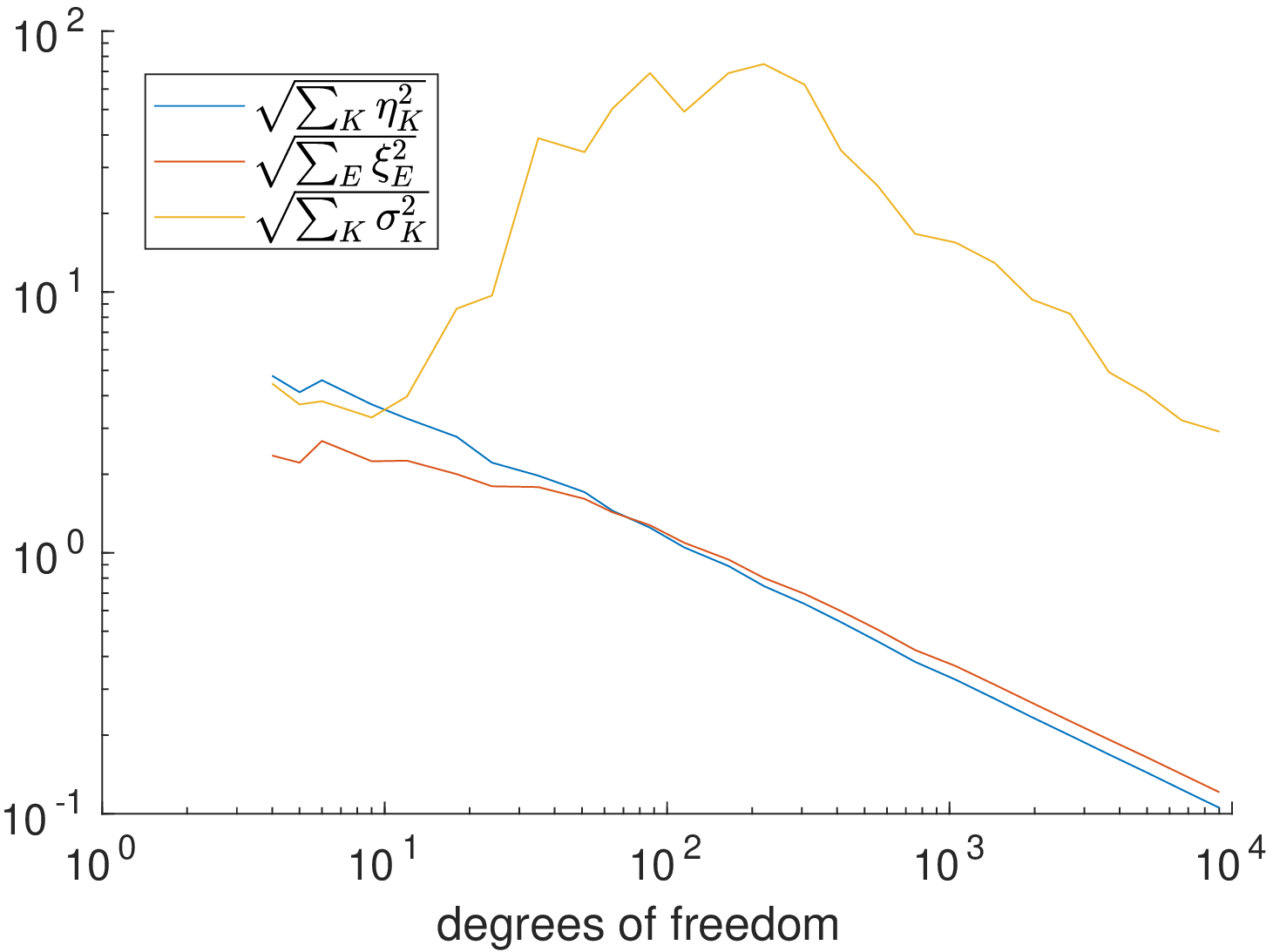}
    \caption{VEM of order 1.}
    \label{fig:test1-analysis:order1:aspectratiostabeta1}
  \end{subfigure}
  \hfill
  \begin{subfigure}{.49\linewidth}
    \centering
    \includegraphics[width=\linewidth]{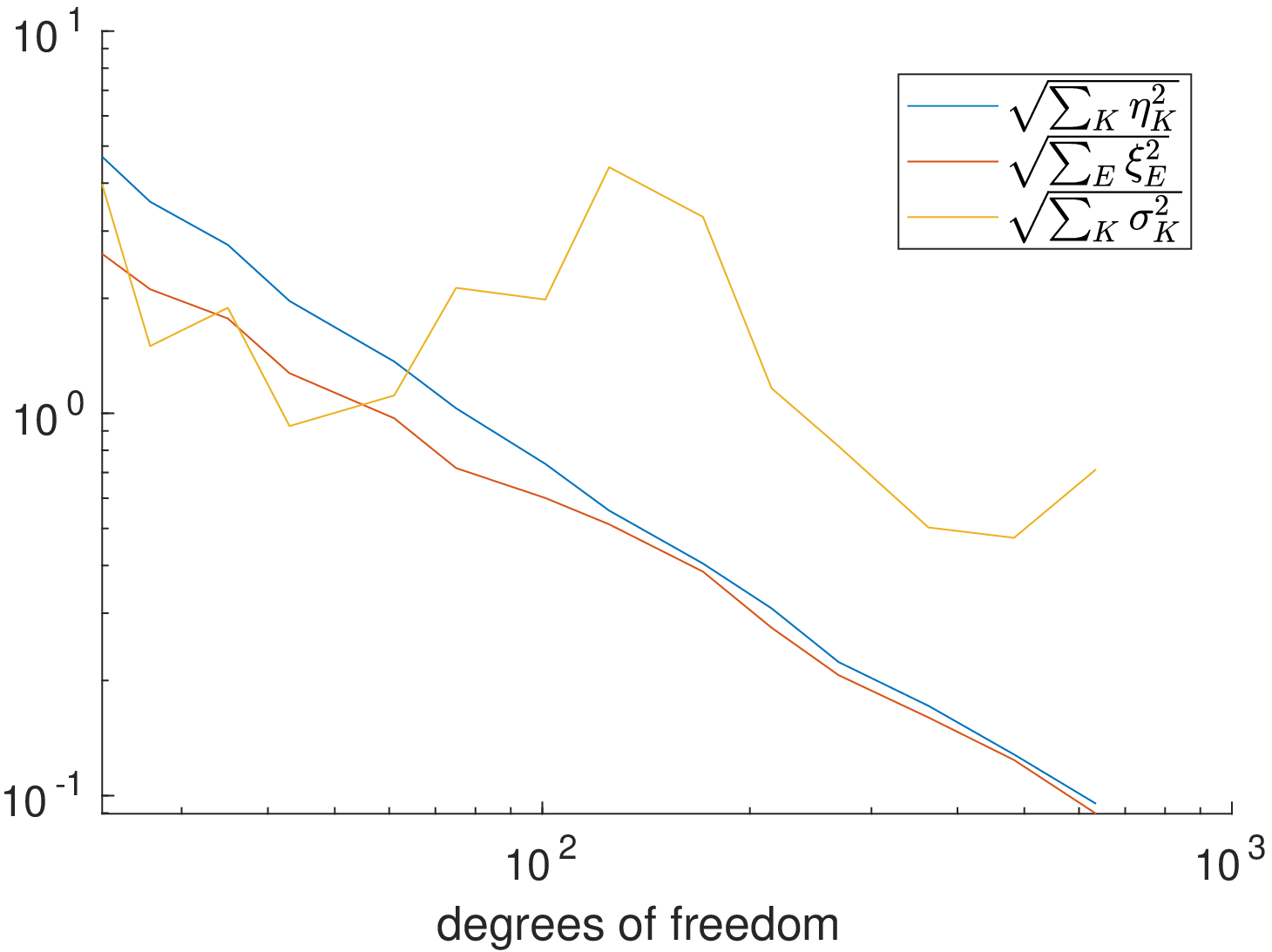}
    \caption{ VEM of order 2.}
    \label{fig:test1-analysis:order2:aspectratiostabeta1}
  \end{subfigure}
  \caption{ Test case 1.  Computed values of $\eta_{K}$,  $\xi_E$ and $\sigma_K$ as a function of the number of degrees of freedom based on employing the \emph{heuristically scaled estimator} $\sth$ defined in \eqref{eq:def_est_heur} to drive the adaptive algorithm.}
  \label{fig:test1-analysis-sigmaK}
\end{figure}
%%%%

We next analyze the behavior of the error based on employing the \emph{heuristically scaled estimator} defined in \eqref{eq:def_est_heur} to drive the adaptive algorithm.
In Figure~\ref{fig:test1-solutions} we report the color-plot of the solutions and
the meshes obtained at the first refinement iteration and at an
intermediate adaptive step, based on employing VEM of order 1 and 2, respectively.
%%%%
\begin{figure}
  \centering
  \begin{subfigure}[b]{.49\linewidth}
    \centering
    \includegraphics[width=\linewidth]{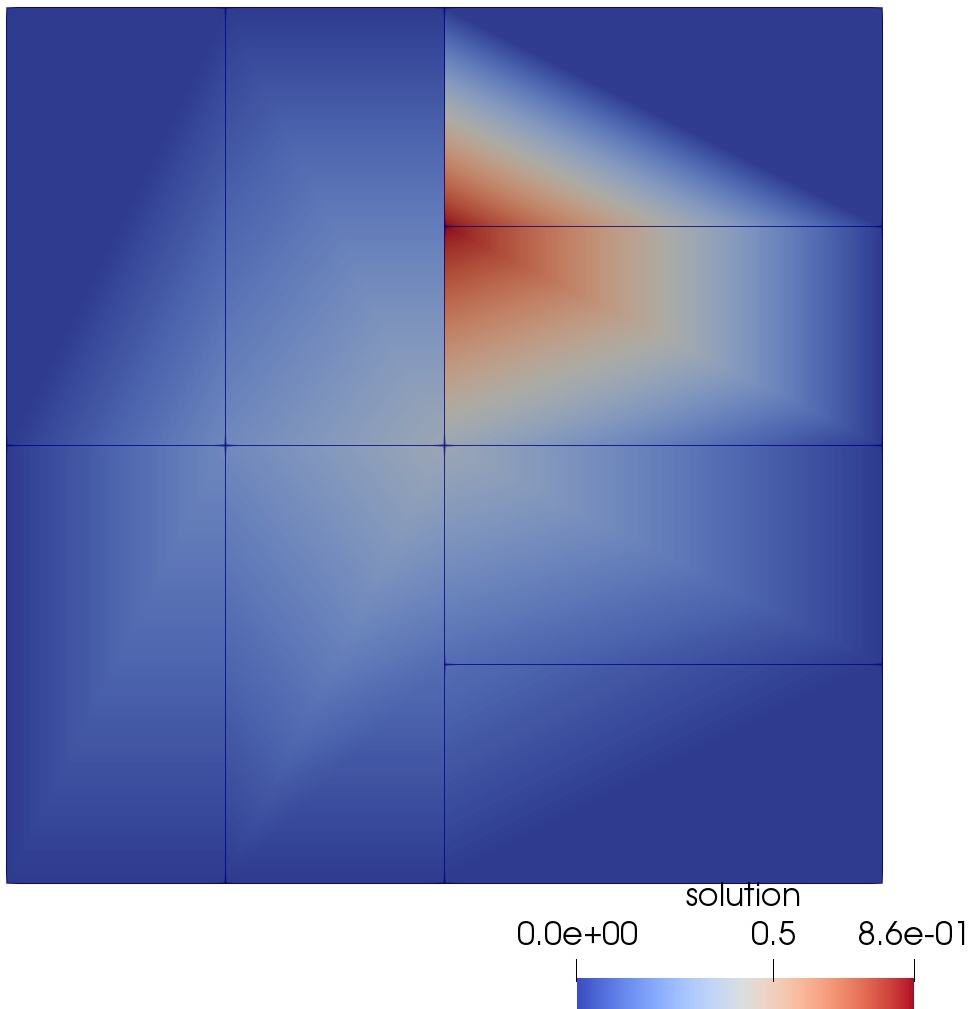}
    \caption{VEM of order 1. Adaptive step n. 1.}
  \end{subfigure}
  \begin{subfigure}[b]{.49\linewidth}
    \centering
    \includegraphics[width=\linewidth]{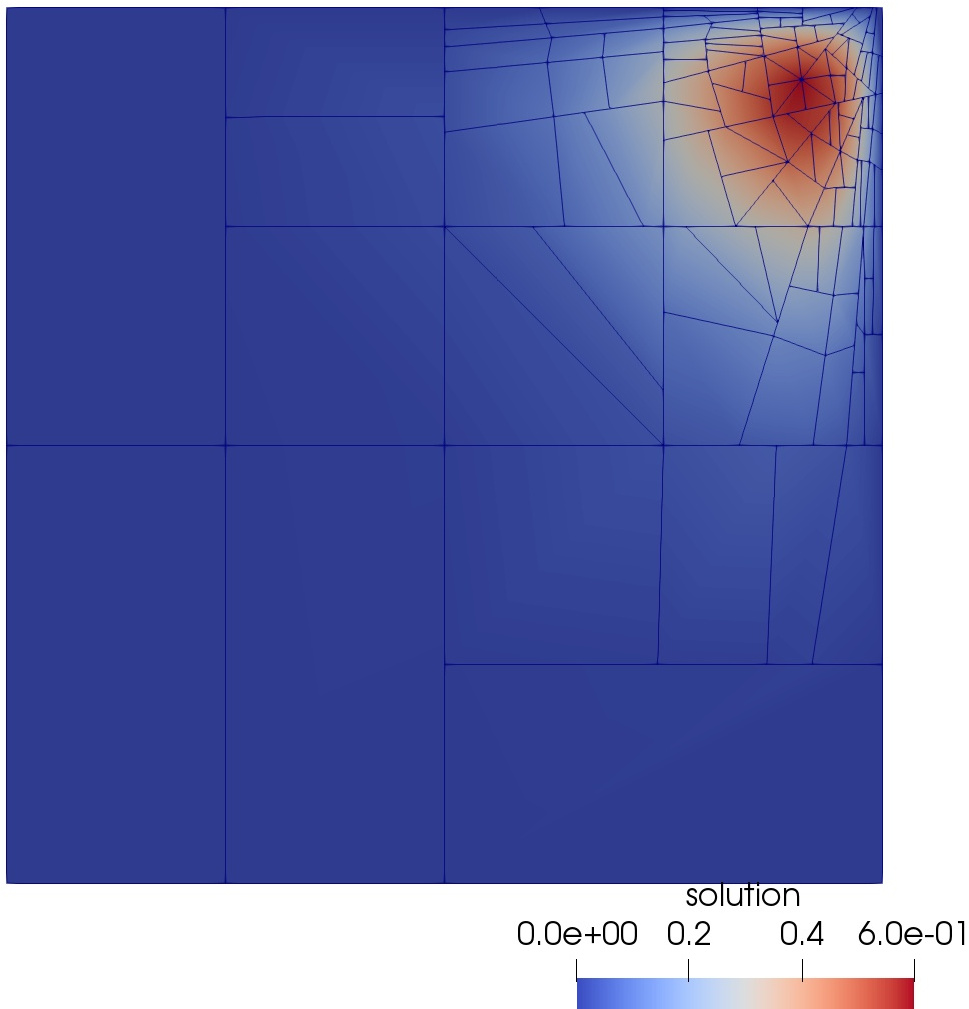}
    \caption{VEM of order 1. Adaptive step n. 12.}
  \end{subfigure}
  \begin{subfigure}[b]{.49\linewidth}
    \centering
    \includegraphics[width=\linewidth]{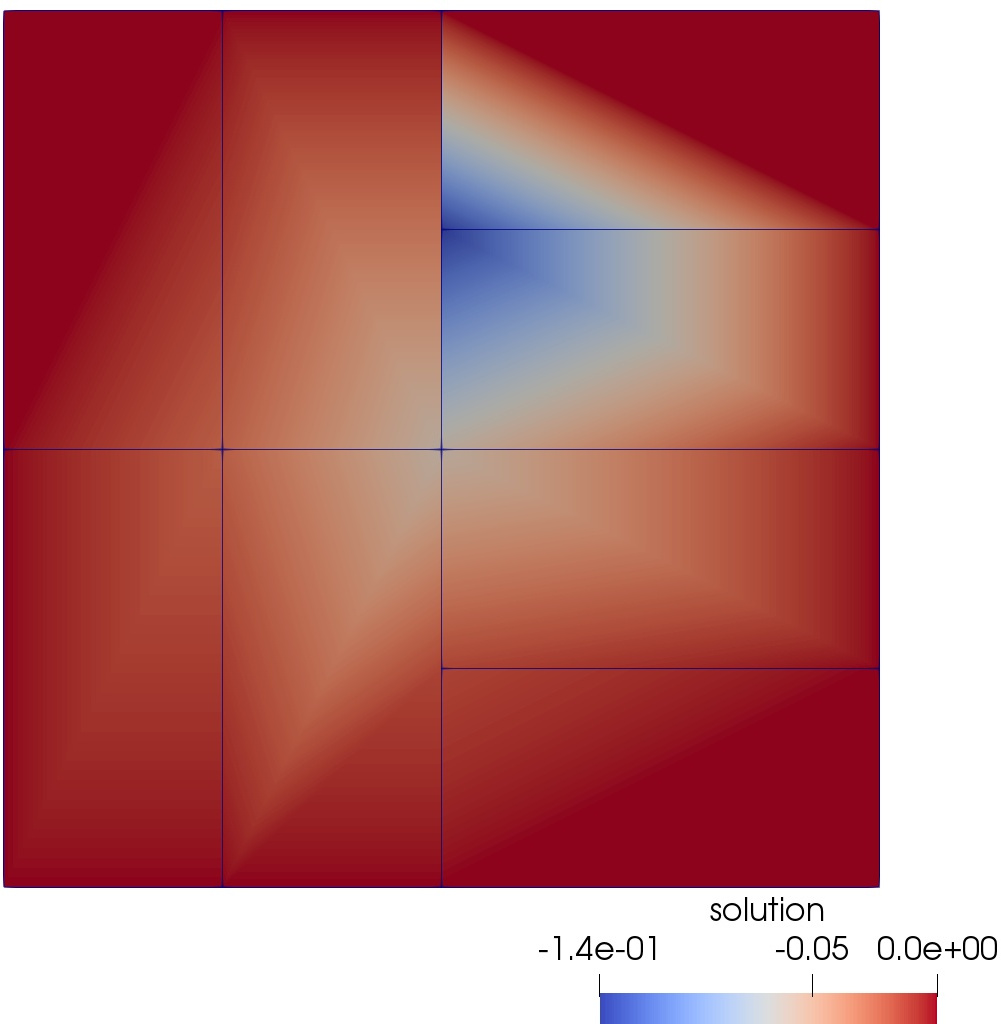}
    \caption{VEM order 2. Adaptive step n. 1.}
  \end{subfigure}
  \begin{subfigure}[b]{.49\linewidth}
    \centering
    \includegraphics[width=\linewidth]{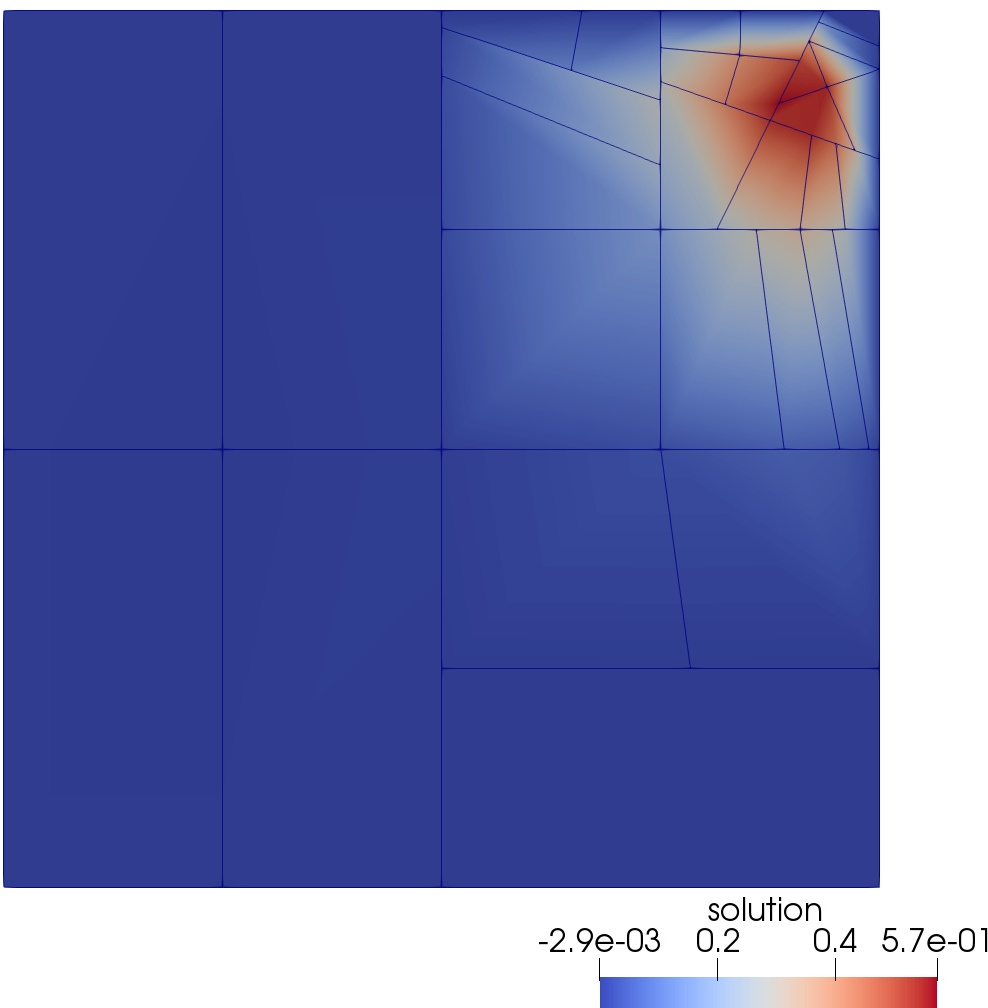}
    \caption{VEM order 2. Adaptive step n. 7.}
  \end{subfigure}
  \caption{Test case 1. Computed solutions and corresponding anisotropic grids at different steps of the adaptive algorithm based on employing the \emph{heuristically scaled estimator} $\sth$ defined in \eqref{eq:def_est_heur} to drive the adaptive algorithm.}
  \label{fig:test1-solutions}
\end{figure}
%%%%
A zoom of a detail of the computed anisotropic mesh as well as the computed solution at the final adaptive step,
are reported in Figure~\ref{fig:test1-solutions-zoom}, again employing VEM of order 1 and 2.
 %%%%
\begin{figure}
  \centering
  \begin{subfigure}[b]{.32\linewidth}
    \centering
    \includegraphics[width=\linewidth]{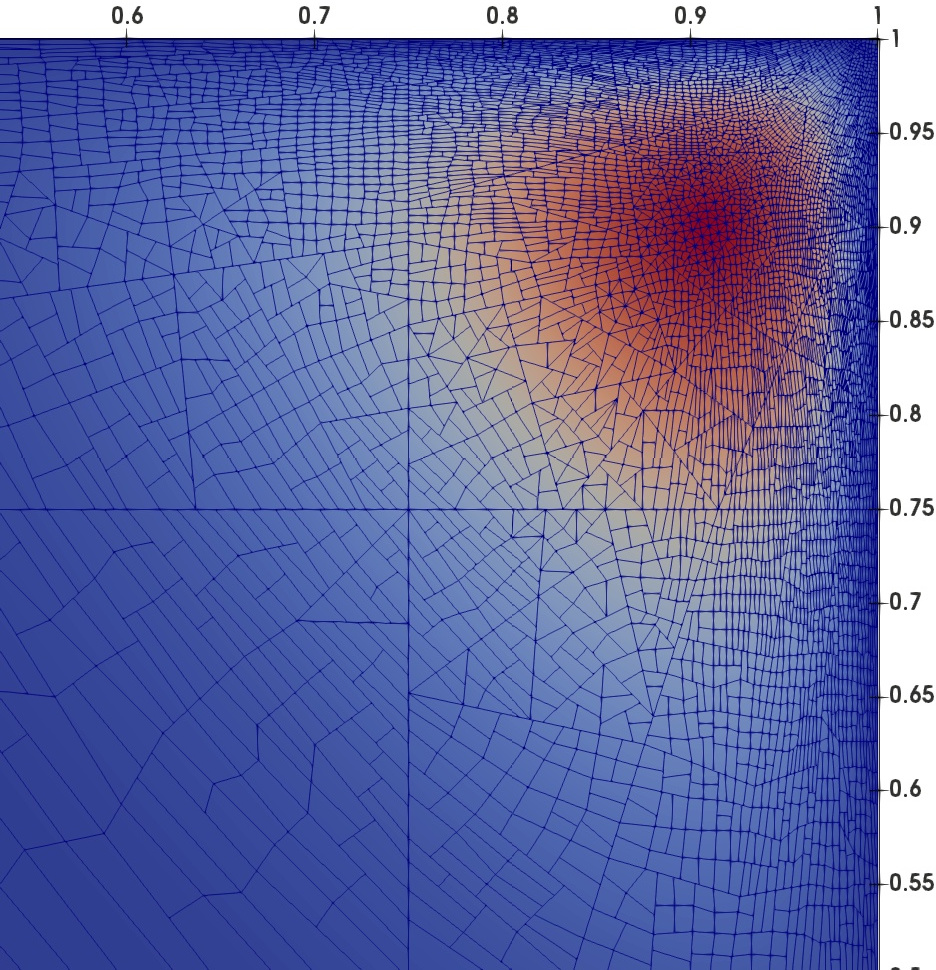}
    \caption{VEM order 1. Adaptive step n. 25.}
  \end{subfigure}
  \begin{subfigure}[b]{.32\linewidth}
    \centering
    \includegraphics[width=\linewidth]{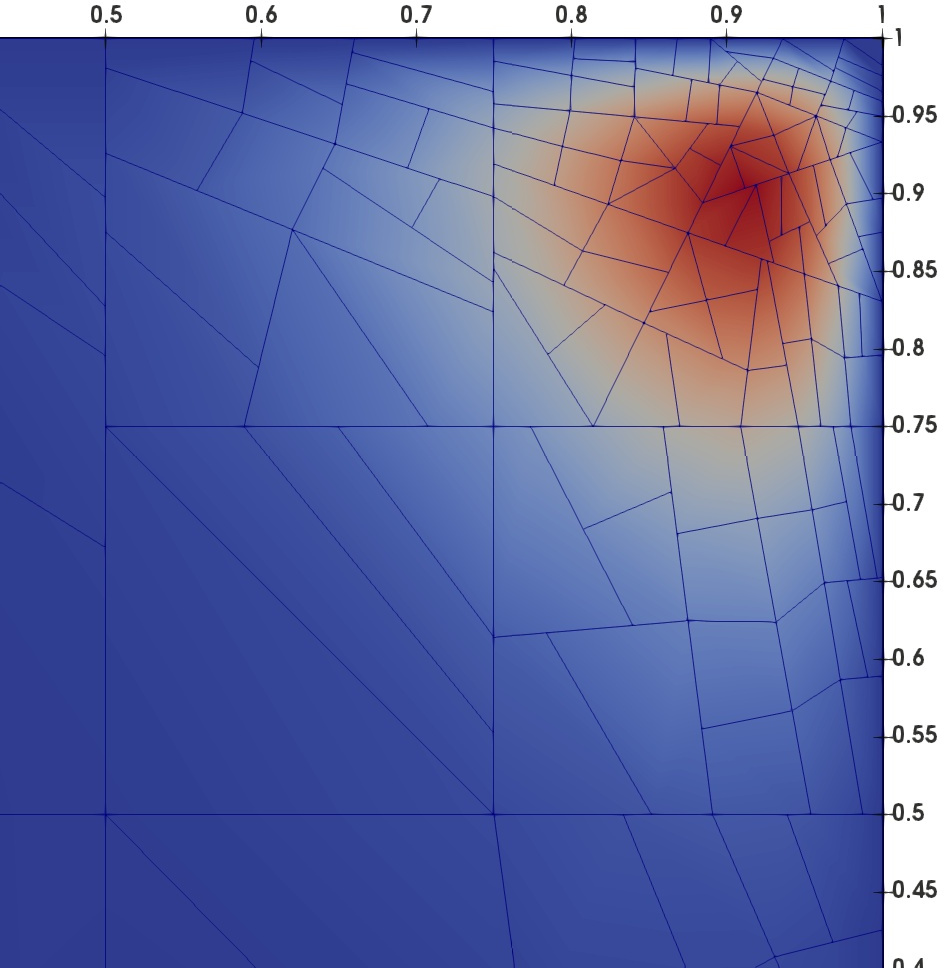}
    \caption{VEM order 2. Adaptive step n. 13}
  \end{subfigure}
  \caption{Test case 1. Zoom of the computed solutions and corresponding anisotropic grids at the final step of the adaptive algorithm based on employing the \emph{heuristically scaled estimator} $\sth$ defined in \eqref{eq:def_est_heur} to drive the adaptive algorithm.}
  \label{fig:test1-solutions-zoom}
\end{figure}
%%%%

We next compared the behavior of the computed estimator and of the  error as a function of the number of the degrees of freedom, again based on employing the \emph{heuristically scaled estimator} defined in \eqref{eq:def_est_heur} to drive the adaptive algorithm. 
In Figure~\ref{fig:test1-analysis:totalerrors} we report the computed
estimator $\sth$  and the computed error $\tilde{e}$, plotted
against the number of degrees of freedom as well as the  convergence
rates (computed based on employing a least square fitting). For the sake of comparison, we also report the analogous quantities obtained with the isotropic error estimator $\eta^{\mathrm{iso}}_h$ defined in \eqref{eq:def_est_iso}, and denoted by $\eta^{\mathrm{iso}}_h$ and $\tilde{e}^{\mathrm{iso}}$, respectively.
We observe, as expected,  that the isotropic adaptive process requires a larger
number of degrees of freedom to reduce the error below a given tolerance compared with the anisotropic error estimator.
%%%%%%%%%%%%
\begin{figure}
  \centering
  \begin{subfigure}{.49\linewidth}
    \centering \includegraphics[width =
    \linewidth]{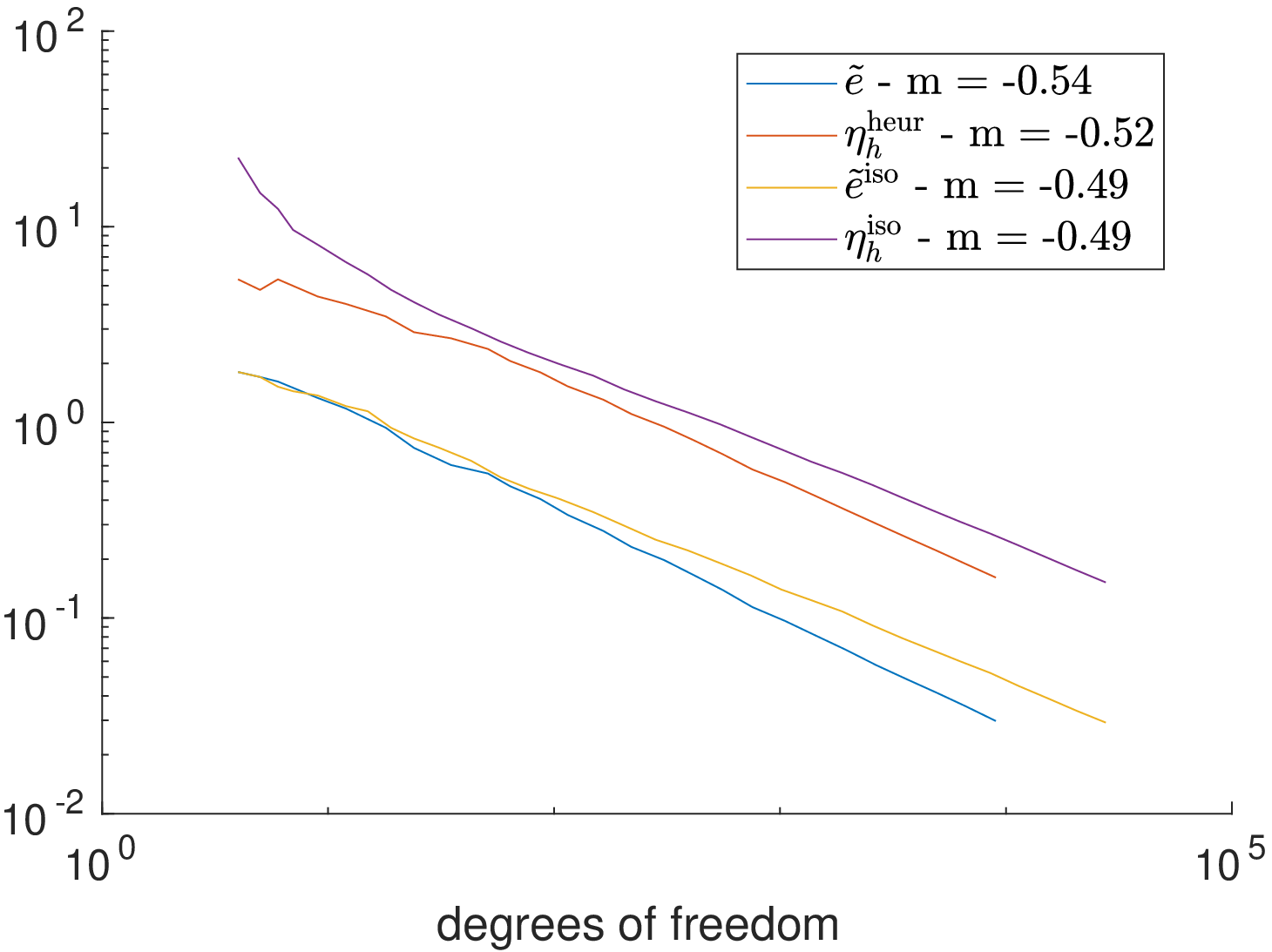}
    \caption{VEM of order 1.}
    \label{fig:test1-analysis:order1:totalerrors}
  \end{subfigure}
  \hfill
  \begin{subfigure}{.49\linewidth}
    \centering \includegraphics[width =
    \linewidth]{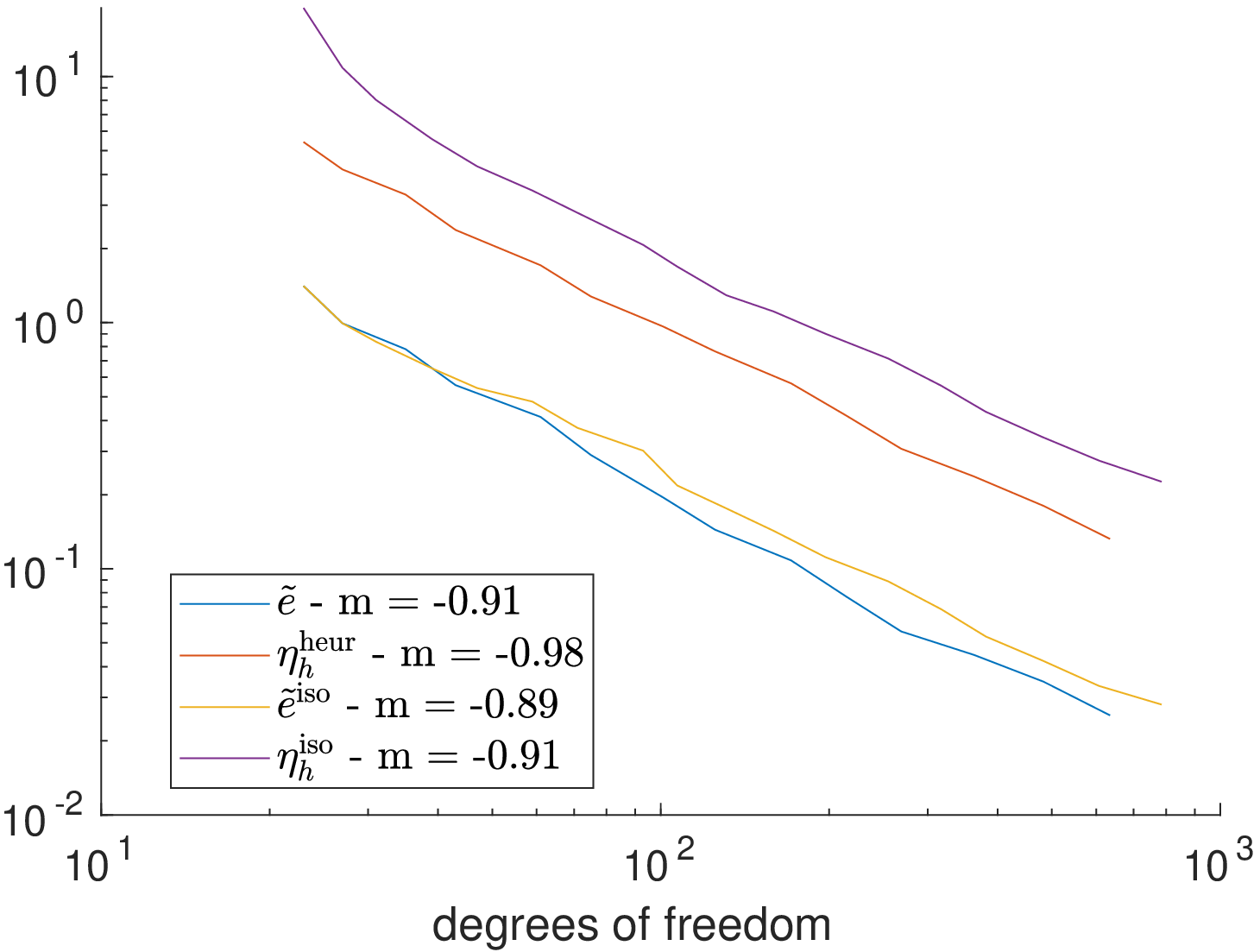}
    \caption{VEM of order  2}
    \label{fig:test1-analysis:order2:totalerrors}
  \end{subfigure}    
   \caption{Test case 1. Computed values of the estimator $\sth$, computed errors $\widetilde{e}$ based on employing the exact solution, and corresponding computed convergence rates \texttt{m} as a function of the number of degrees of freedom. The results are compared with the analogous quantities obtained based on employing the isotropic error estimator $\eta^{\mathrm{iso}}_h$ defined in \eqref{eq:def_est_iso}.}
    \label{fig:test1-analysis:totalerrors}
\end{figure}

\FloatBarrier

%%%%%%%%%%%%%%%%%%%%%%%%
%%%%%%%%%%%%%%%%%%%%%%%%
%%%%%%%%%%%%%%%%%%%%%%%%%
%%%%%%%%%%%%%%%%%%%%%%%%%
%%%%%%%%%%%%%%%%%%%%%%%%%
%%%%%%%%%%%%%%%%%%%%%%%%%
%%%%%%%%%%%%%%%%%%%%%%%.
%%%%%%%%%%%%%%%%%%%%%%%%%

\subsection{Test case 2}
We have repeated the same set of experiments of the previous section, now choosing the forcing term in such a way that the exact solution is given by
\begin{equation*}
 u(x,y) = 10^{-2}xy(1-x)(1-y)(\mathrm{e}^{10x} - 1).
\end{equation*}
We notice that the exact solution of the proposed test case  exhibits a steep boundary layer in the $x$-direction close to the
right boundary of the domain (see Figure~\ref{fig:test2-sol}). 
%%%%%%%%%
\begin{figure}
  \centering \includegraphics[scale=0.7]{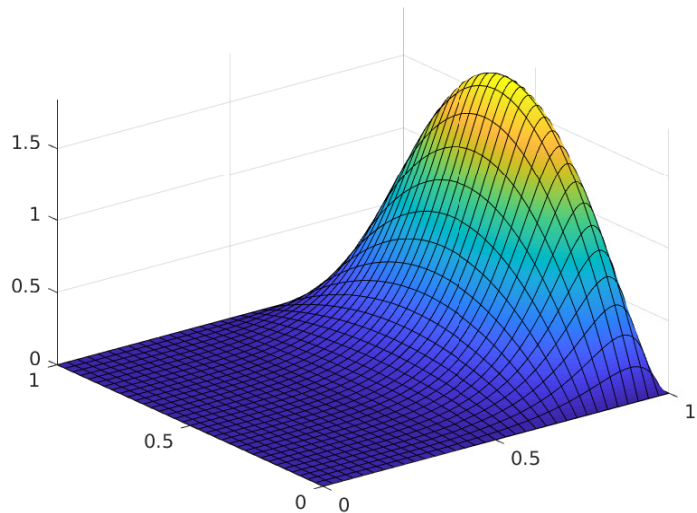}
  \caption{Test case 2. Plot of the exact solution}
  \label{fig:test2-sol}
\end{figure}
%%%%%%%%%
Next, we report the color-plot of the computed solutions and
the corresponding meshes at the initial step of the adaptive algorithm, and after 16 (resp. 9),  iterations based on employing VEM of order 1 (resp. 2), 
and using  the \emph{heuristically scaled estimator} $\sth$ defined in \eqref{eq:def_est_heur} to drive the adaptive process; cf.  Figure~\ref{fig:test2-solutions}.
%%%%
\begin{figure}
  \centering
  \begin{subfigure}[b]{.49\linewidth}
    \centering
    \includegraphics[width=\linewidth]{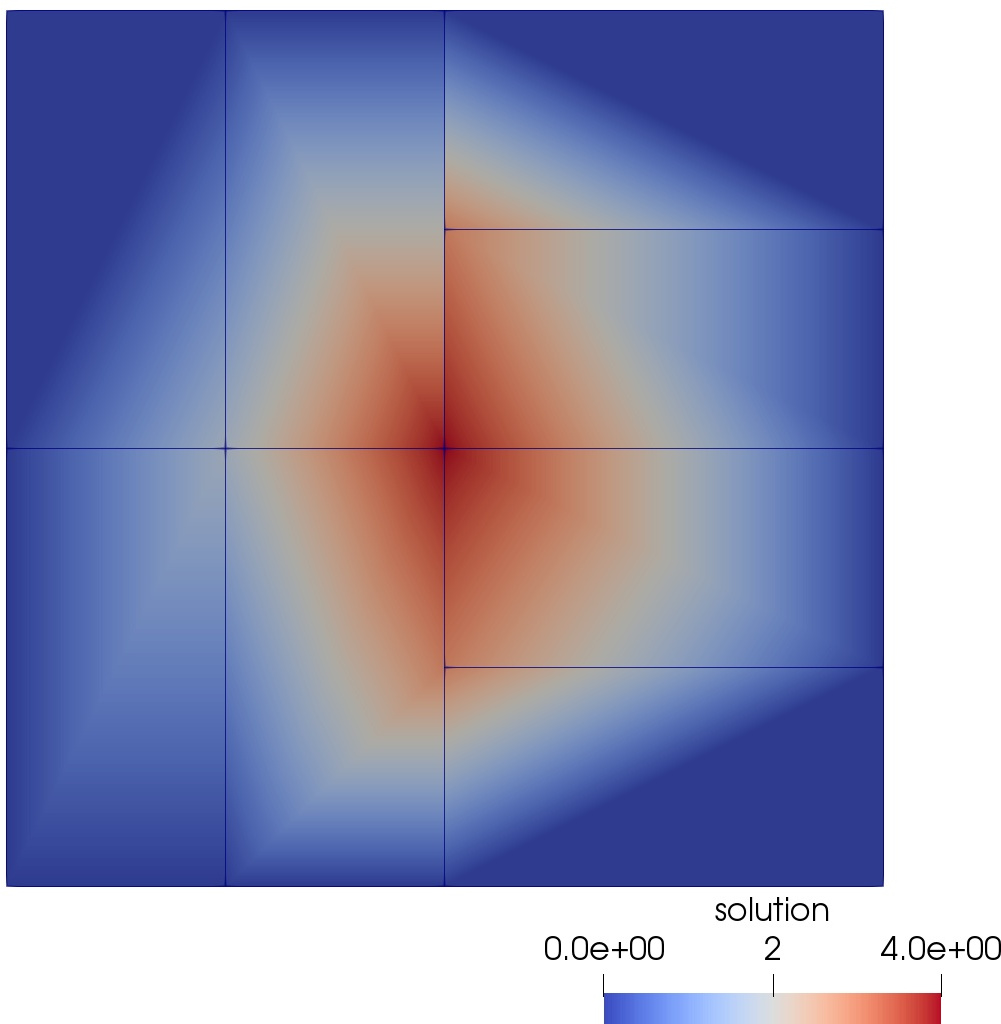}
    \caption{VEM of order 1. Adaptive step n. 1.}
  \end{subfigure}
  \begin{subfigure}[b]{.49\linewidth}
    \centering
    \includegraphics[width=\linewidth]{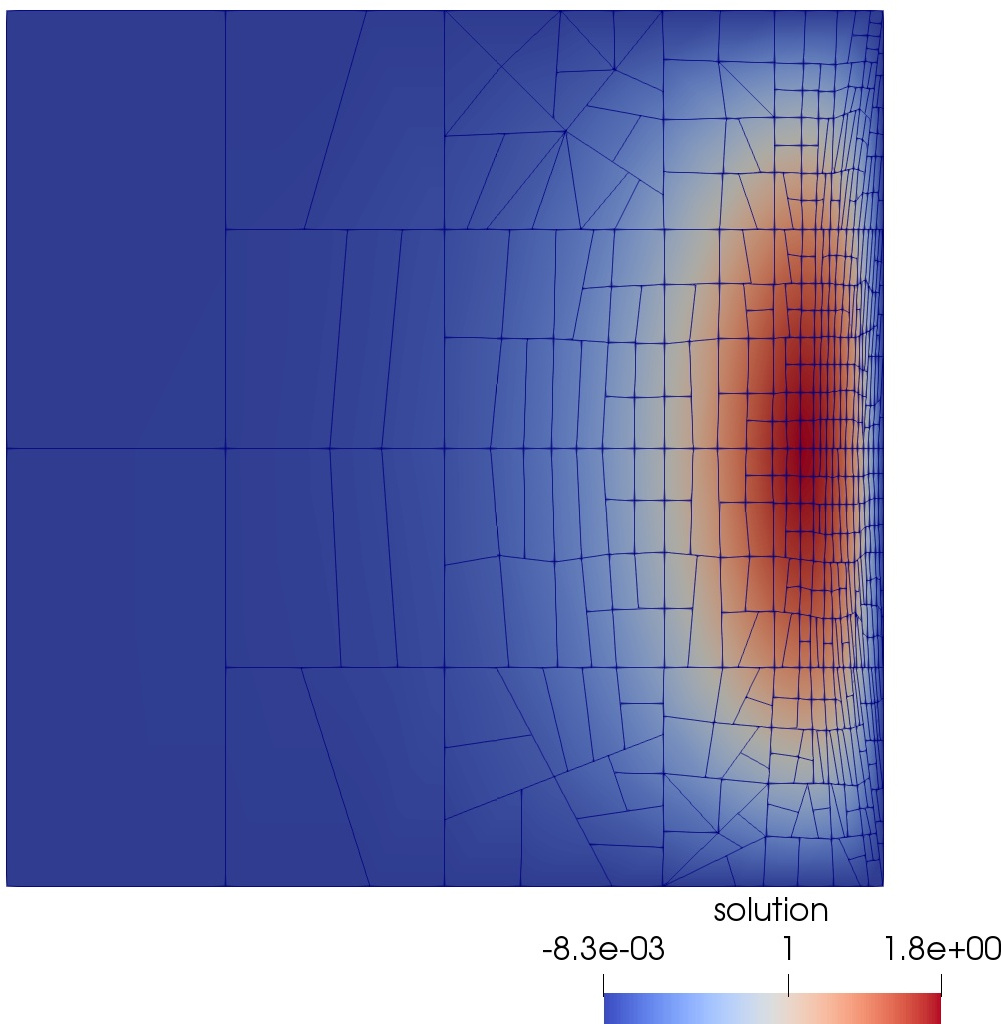}
    \caption{VEM of order 1. Adaptive step n. 16.}
  \end{subfigure}
  \begin{subfigure}[b]{.49\linewidth}
    \centering
    \includegraphics[width=\linewidth]{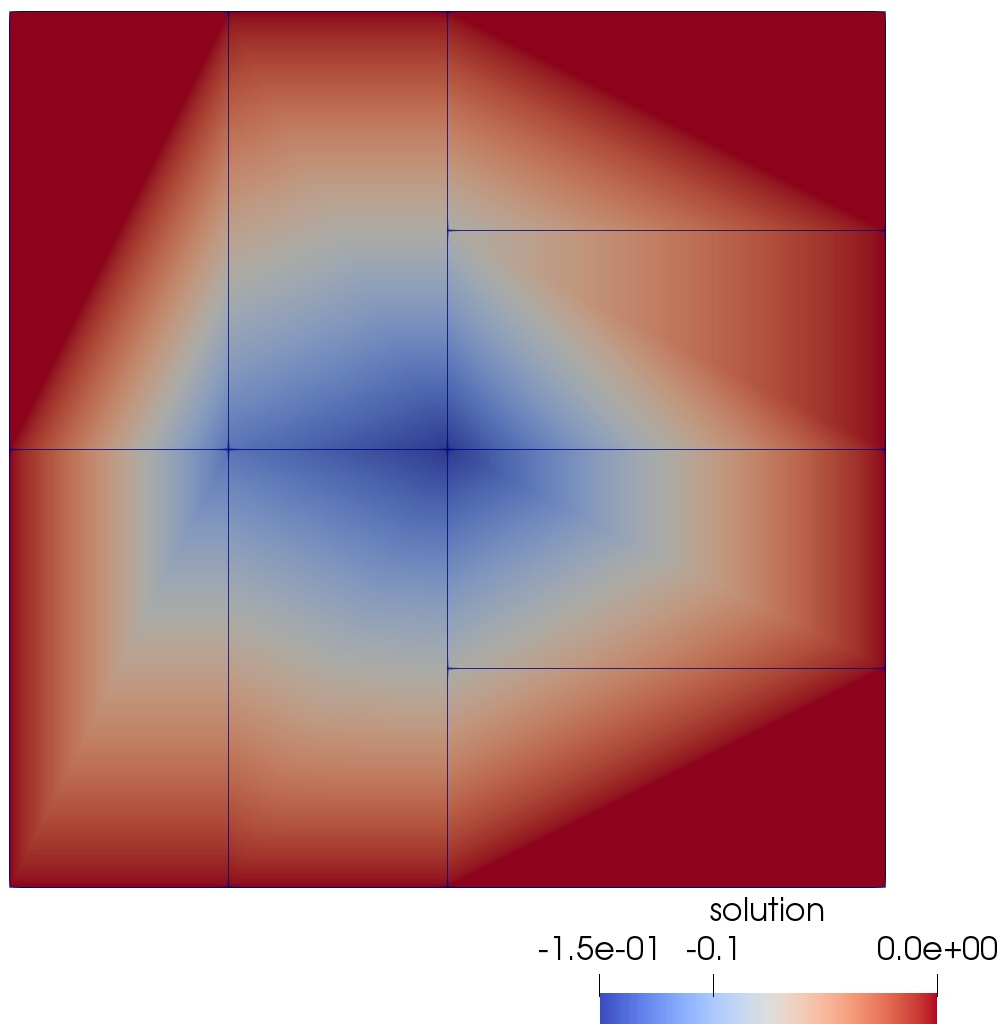}
    \caption{VEM order 2. Adaptive step n. 1.}
  \end{subfigure}
  \begin{subfigure}[b]{.49\linewidth}
    \centering
    \includegraphics[width=\linewidth]{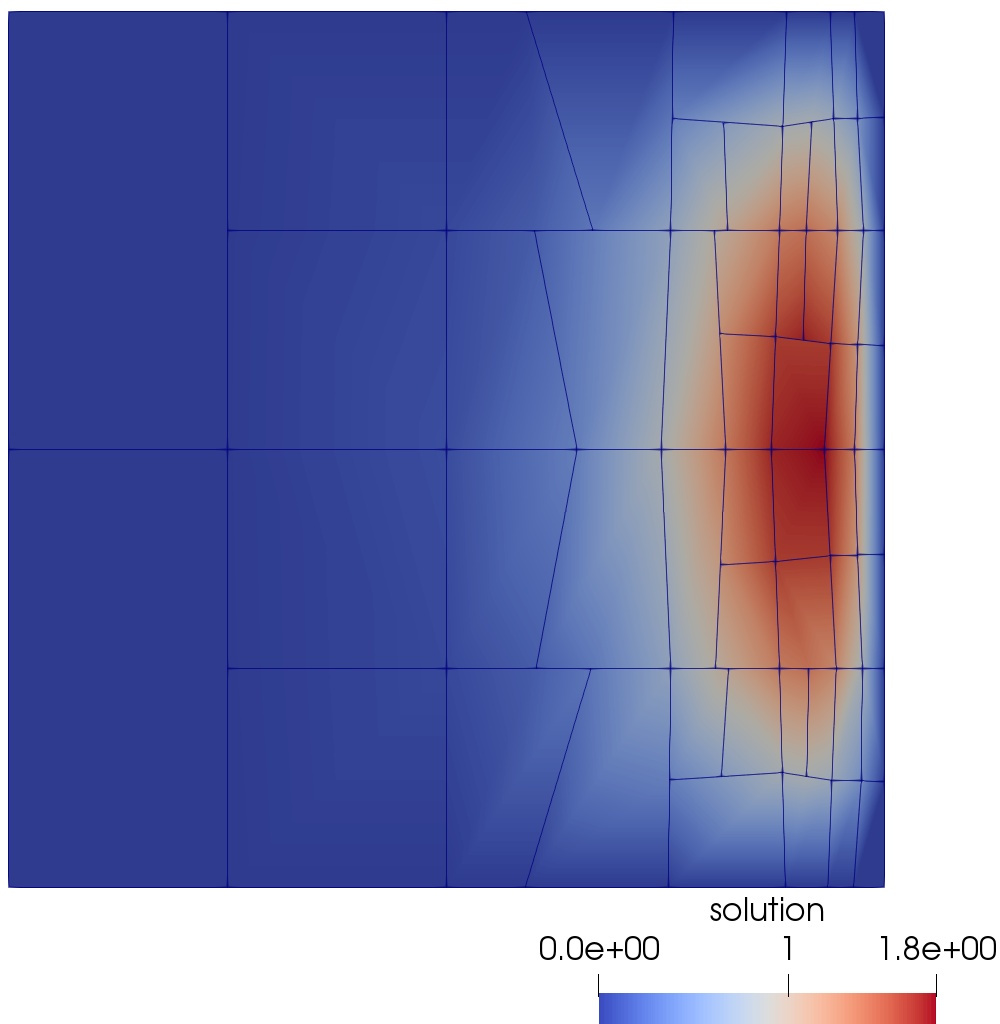}
    \caption{VEM order 2. Adaptive step n. 9.}
  \end{subfigure}
  \caption{Test case 2. Computed solutions and corresponding anisotropic grids at different steps of the adaptive algorithm based on employing the \emph{heuristically scaled estimator} $\sth$ defined in \eqref{eq:def_est_heur} to drive the adaptive algorithm.}
  \label{fig:test2-solutions}
\end{figure}
%%%%%%
A zoom of a detail of the computed anisotropic mesh as well as the corresponding computed solution at the final step of the adaptive algorithm
are reported in Figure~\ref{fig:test2-solutions-zoom}, again employing VEM of order 1 (Figure~\ref{fig:test2-solutions-zoom}, top) and VEM of order 2 (Figure~\ref{fig:test2-solutions-zoom}, bottom).
 %%%%
\begin{figure}
  \centering
  \begin{subfigure}[b]{.32\linewidth}
    \centering
    \includegraphics[width=\linewidth]{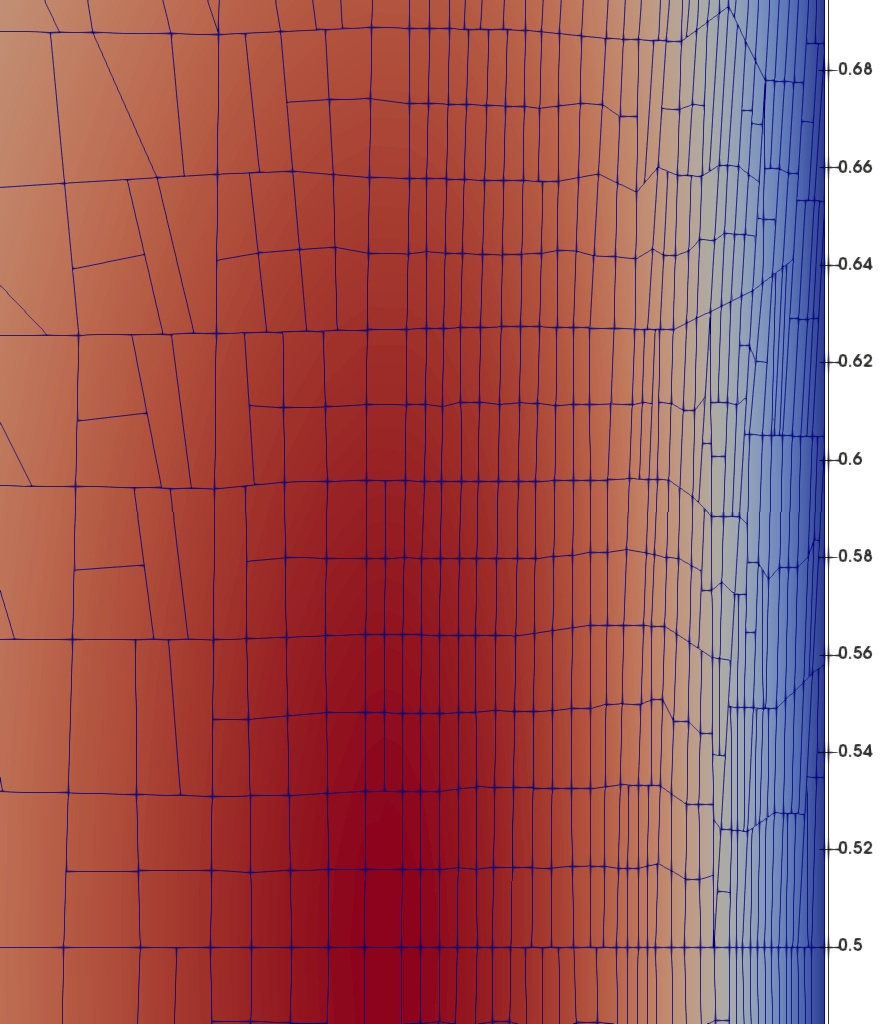}
    \caption{VEM order 1. Adaptive step n. 21.}
  \end{subfigure}
  \begin{subfigure}[b]{.32\linewidth}
    \centering
    \includegraphics[width=\linewidth]{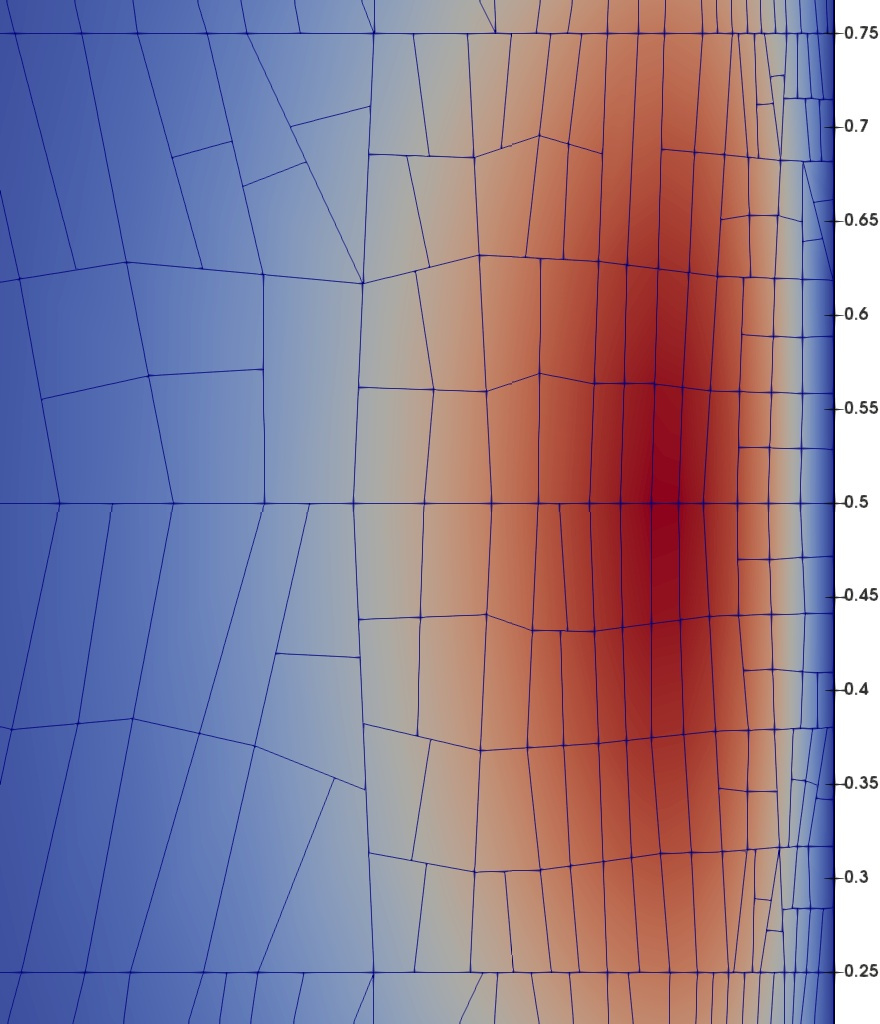}
    \caption{VEM order 2. Adaptive step n. 13}
  \end{subfigure}
  \caption{Test case 2. Zoom of the computed solutions and corresponding anisotropic grids at the final step of the adaptive algorithm based on employing the \emph{heuristically scaled estimator} $\sth$ defined in \eqref{eq:def_est_heur} to drive the adaptive algorithm.}
  \label{fig:test2-solutions-zoom}
\end{figure}
%%%%%%%%%%%%%%%%%%%
%%%%%%%%%%%%%%%%%%%
Finally, we compare the behavior of the computed estimator and of the  error as a function of the number of the degrees of freedom.  
In Figure~\ref{fig:test2-analysis:totalerrors} we report the 
estimator $\sth$  and the  error $\tilde{e}$, plotted
against the number of degrees of freedom as well as the computed convergence
rates. As before, we compare these results with the analogous quantities obtained with the isotropic error estimator $\eta^{\mathrm{iso}}_h$ defined in \eqref{eq:def_est_iso}. These results have been obtained with VEM of order 1, cf. Figure~\ref{fig:test2-analysis:order1:totalerrors} and with VEM of order 2, cf. Figure~\ref{fig:test2-analysis:order2:totalerrors}.
We observe, as expected,  that the isotropic adaptive process requires a larger
number of degrees of freedom to reduce the error below a given tolerance compared with the anisotropic error estimator.
%%%%%%%%%%%%
\begin{figure}
  \centering
  \begin{subfigure}{.49\linewidth}
    \centering \includegraphics[width =
    \linewidth]{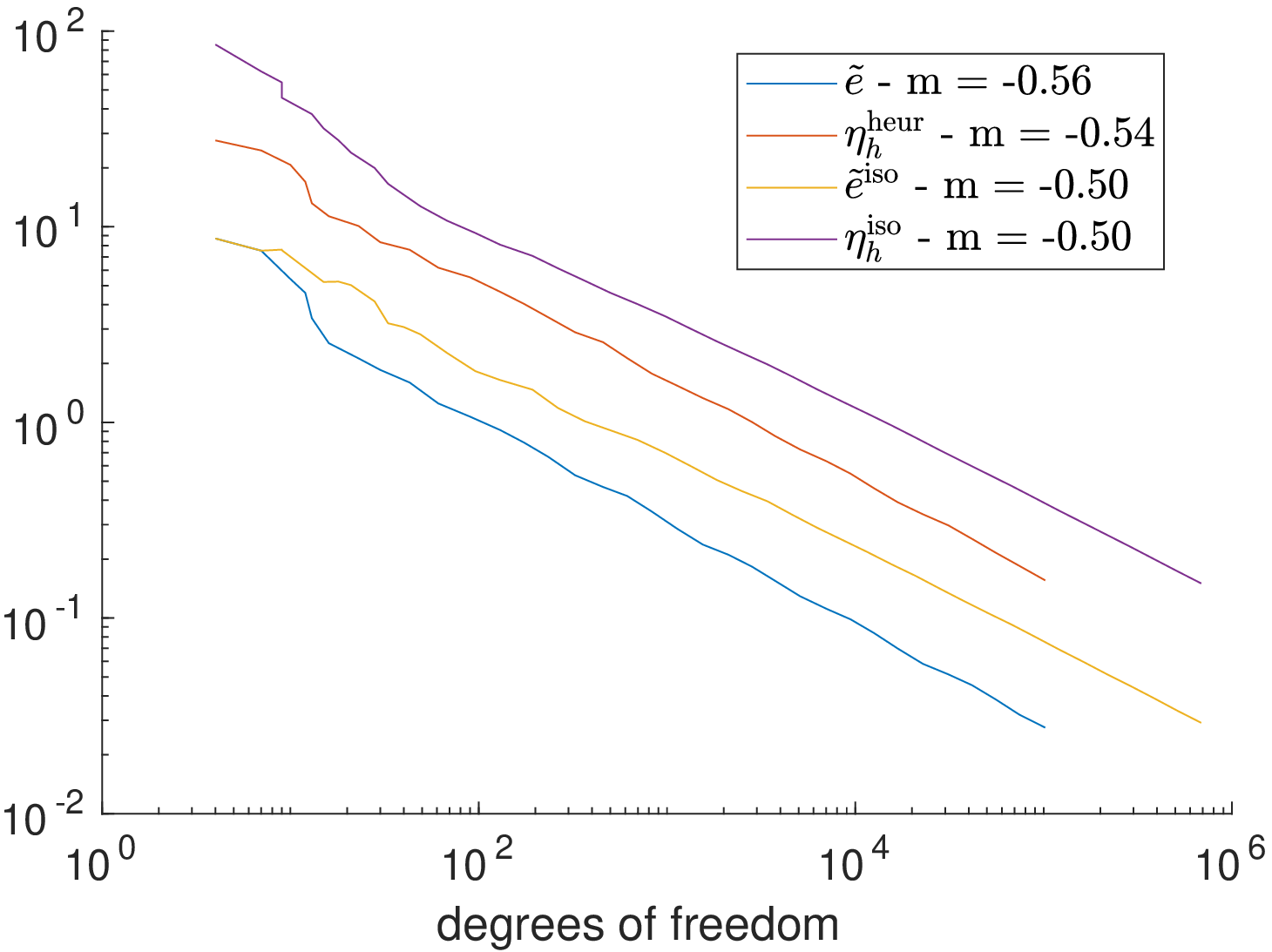}
    \caption{VEM of order 1.}
    \label{fig:test2-analysis:order1:totalerrors}
  \end{subfigure}
  \hfill
  \begin{subfigure}{.49\linewidth}
    \centering \includegraphics[width =
    \linewidth]{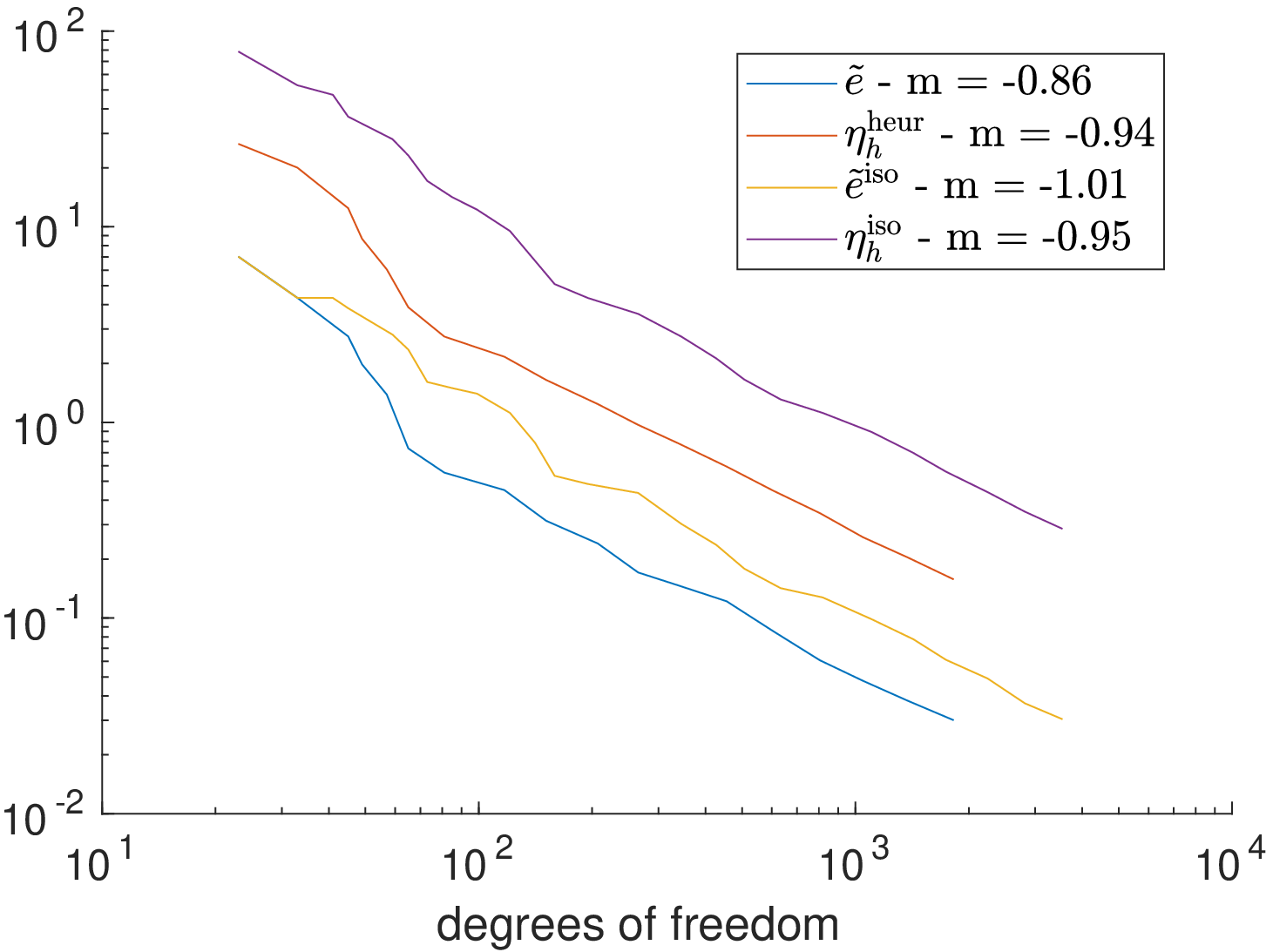}
    \caption{VEM of order  2} 
    \label{fig:test2-analysis:order2:totalerrors}    
  \end{subfigure}
   \caption{Test case 2. Computed values of the estimator $\sth$, computed errors $\widetilde{e}$ based on employing the exact solution, and corresponding computed convergence rates \texttt{m} as a function of the number of degrees of freedom. The results are compared with the analogous quantities obtained based on employing the isotropic error estimator $\eta^{\mathrm{iso}}_h$ defined in \eqref{eq:def_est_iso}.}
   \label{fig:test2-analysis:totalerrors}
\end{figure}
%%%%%%%%%%%%%5
\FloatBarrier

%%%%%%%%%%%%%
%%%%%%%%%%%%%
%%%%%%%%%%%%%
%%%%%%%%%%%%%
%%%%%%%%%%%%%
%%%%%%%%%%%%%
%%%%%%%%%%%%%
\subsection{Test case 3}
We have repeated the same set of experiments of the previous section, now choosing the forcing term in such a way that the exact solution is given by
\begin{equation*}
u(x,y) = 10^{-2}xy(x-1)(y-1)(\mathrm{e}^{10x} - 5000x + 4499) \,,
\end{equation*}
that is obtained summing an isotropic bubble of the form 
\begin{equation*}
  b(x,y) = 50x(1-y)y(0.9-x)(1-x) \,
\end{equation*}
to the solution of the second
test case, cf.  Figure~\ref{fig:test3-sol}. The manufactured exact solution  exhibits a steep boundary layer in the $x$-direction close to the right side of the domain, which requires {\em anisotropic} mesh refinement to be efficiently treated and a bubble function in the left part of the domain, which asks for {\em isotropic} mesh refinement. 
%%%%%%%%%
\begin{figure}
  \centering \includegraphics[scale=0.7]{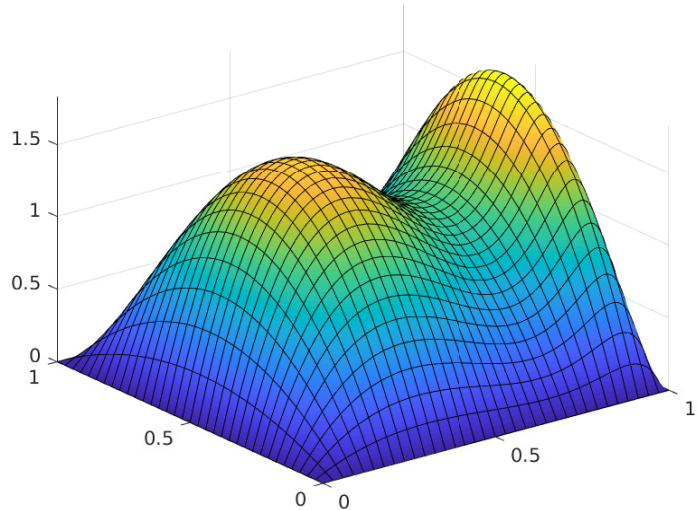}
  \caption{Test case 3. Plot of the exact solution}
  \label{fig:test3-sol}
\end{figure}
%%%%%%%%
Next, we report the color-plot of the computed solutions and
the meshes obtained at the initial step of the refinement process, and after 17 (resp. 9),  iterations based on employing VEM of order 1 (resp. 2), and using  the \emph{heuristically scaled estimator} $\sth$ defined in \eqref{eq:def_est_heur} to drive the adaptive algorithm; Figure~\ref{fig:test3-solutions} (top) shows the results obtained with VEM of order 1, whereas in Figure~\ref{fig:test3-solutions} (bottom) we show the analogous computations obtained with VEM of order 2.
%%%%
\begin{figure}
  \centering
  \begin{subfigure}[b]{.49\linewidth}
    \centering
    \includegraphics[width=\linewidth]{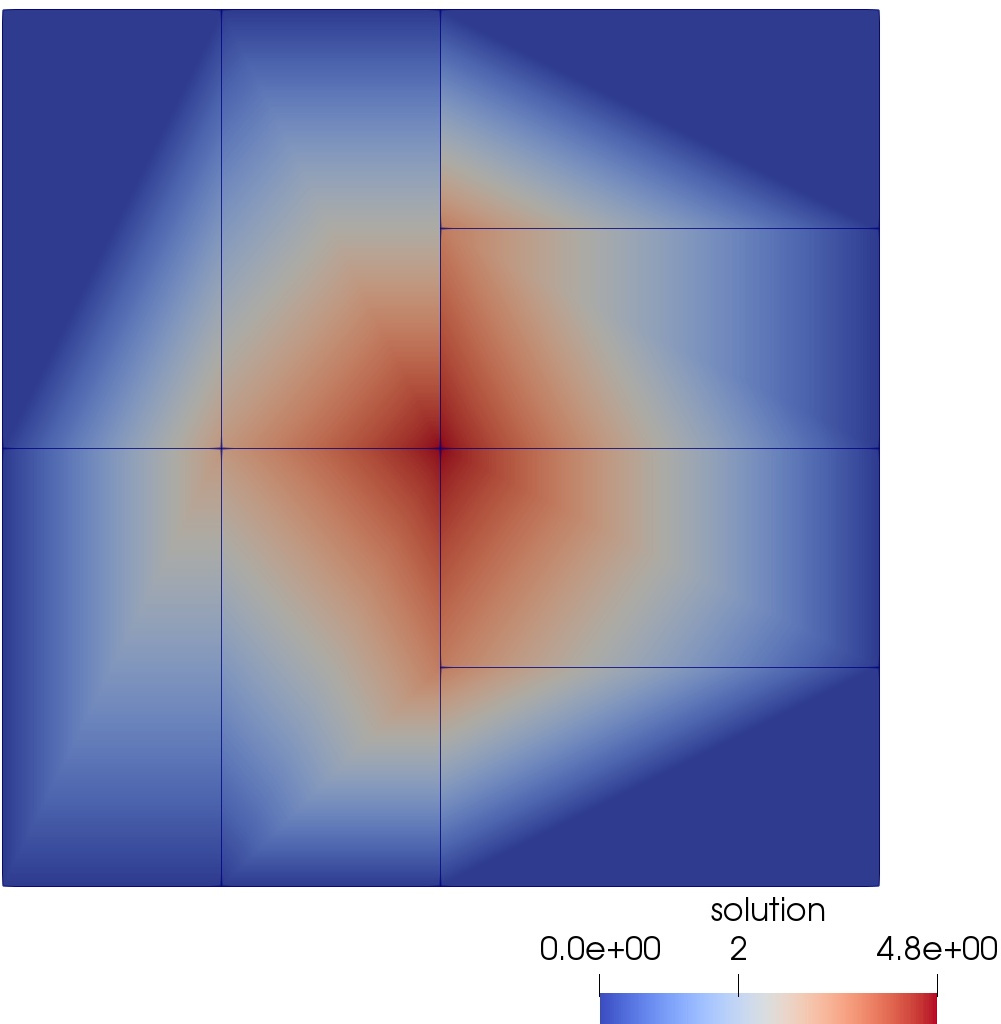}
    \caption{VEM of order 1. Adaptive step n. 1.}
  \end{subfigure}
  \begin{subfigure}[b]{.49\linewidth}
    \centering
    \includegraphics[width=\linewidth]{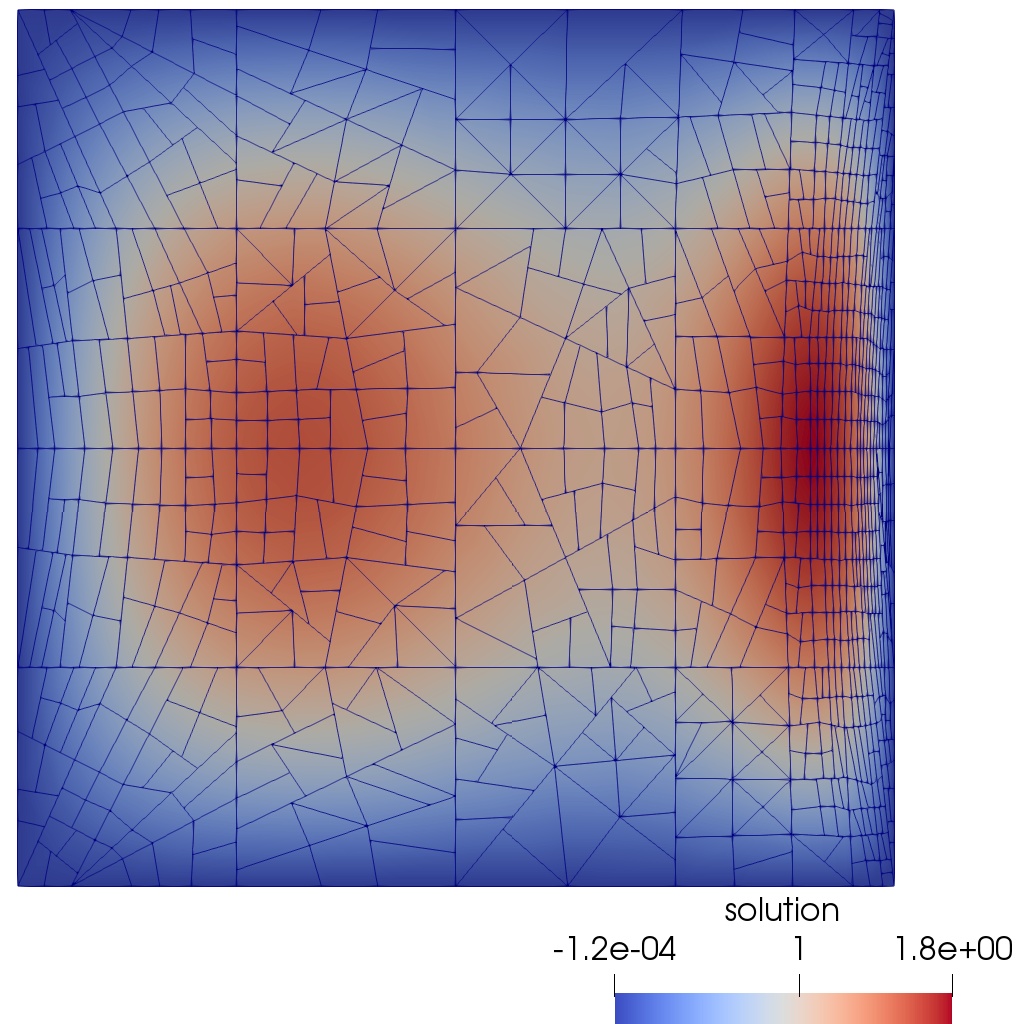}
    \caption{VEM of order 1. Adaptive step n. 17.}
  \end{subfigure}
  \begin{subfigure}[b]{.49\linewidth}
    \centering
    \includegraphics[width=\linewidth]{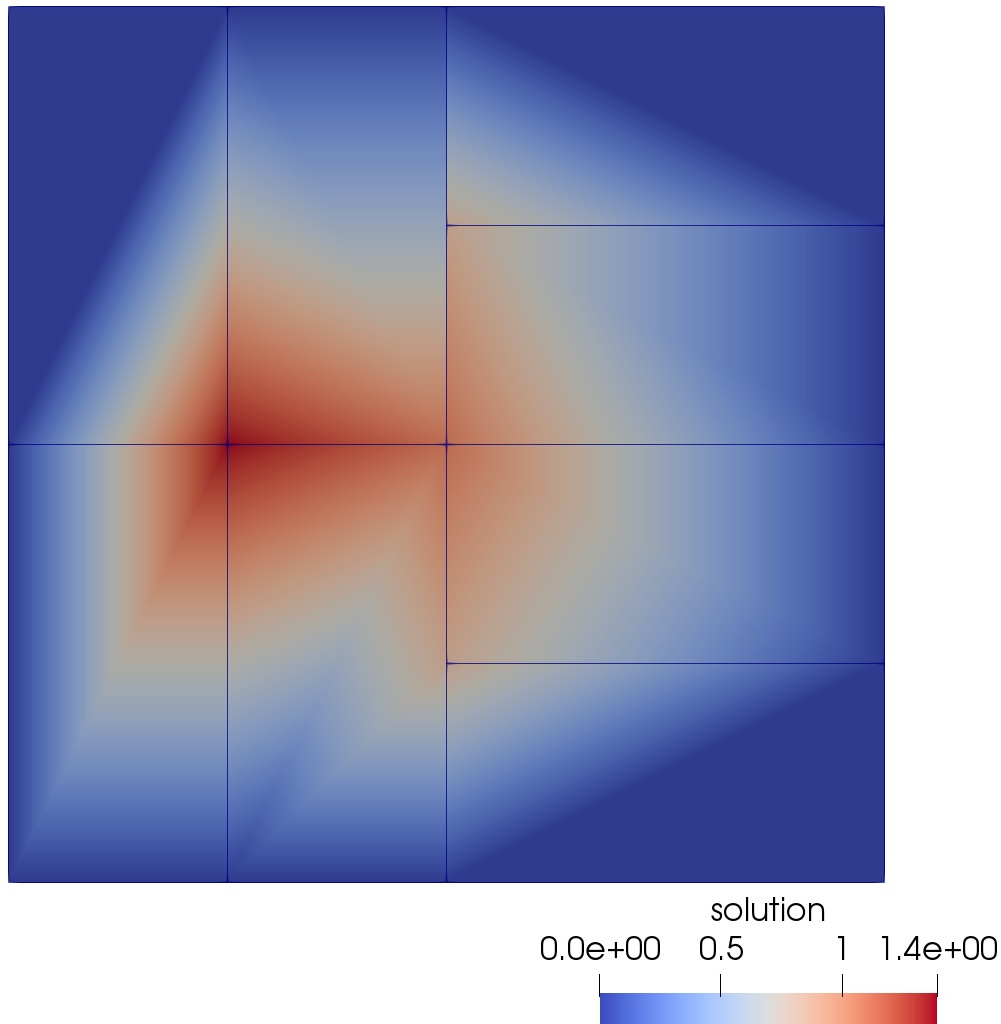}
    \caption{VEM order 2. Adaptive step n. 1.}
  \end{subfigure}
  \begin{subfigure}[b]{.49\linewidth}
    \centering
    \includegraphics[width=\linewidth]{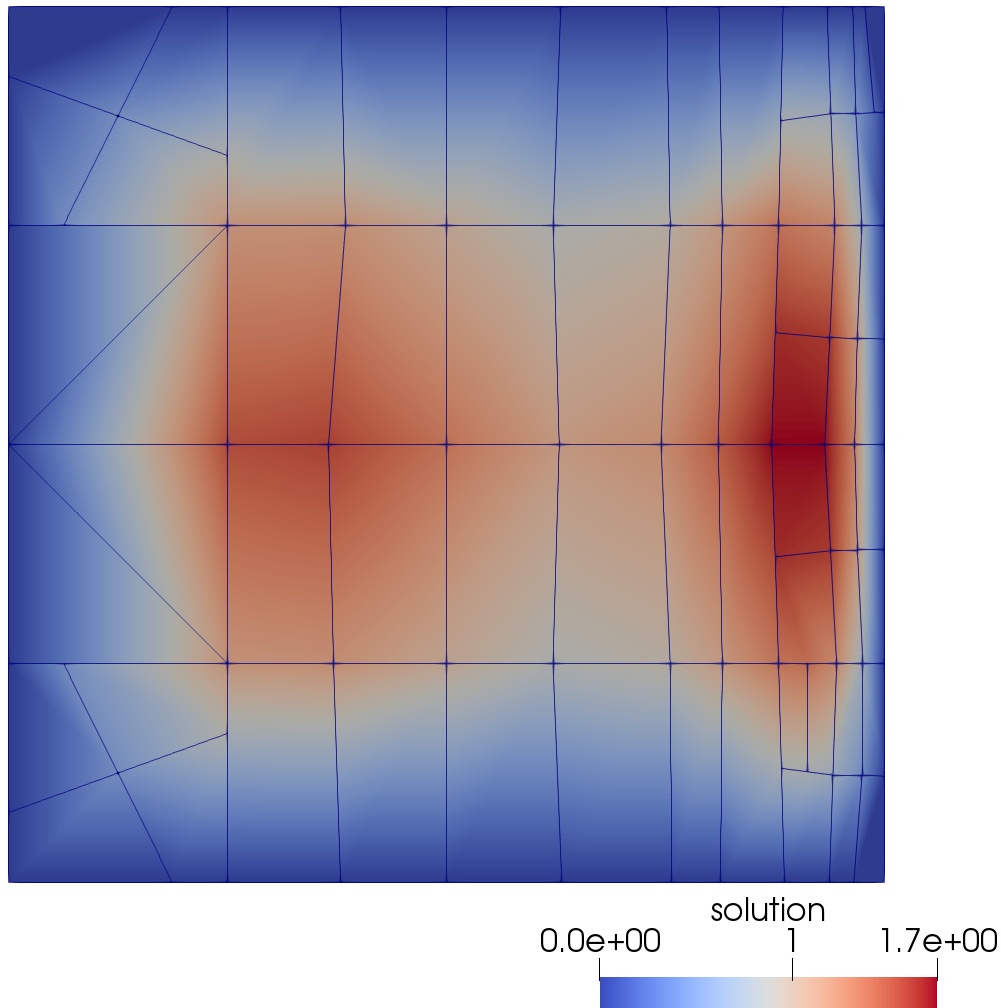}
    \caption{VEM order 2. Adaptive step n. 9.}
  \end{subfigure}
  \caption{Test case 3. Computed solutions and corresponding anisotropic grids at different steps of the adaptive algorithm based on employing the \emph{heuristically scaled estimator} $\sth$ defined in \eqref{eq:def_est_heur} to drive the adaptive algorithm.}
  \label{fig:test3-solutions}
\end{figure}
%%%%%%
A zoom of a detail of the computed solutions together with the corresponding computed anisotropic meshes at the final step  of the adaptive algorithm are reported in Figure~\ref{fig:test3-solutions-zoom}, again employing VEM of order 1 (left) and VEM of order 2 (right). The reported results show that  the combination of isotropic and anisotropic mesh refinement is correctly captured by the adaptive algorithm. 
 %%%%
\begin{figure}
  \centering
  \begin{subfigure}[b]{.32\linewidth}
    \centering
    \includegraphics[width=0.86\linewidth]{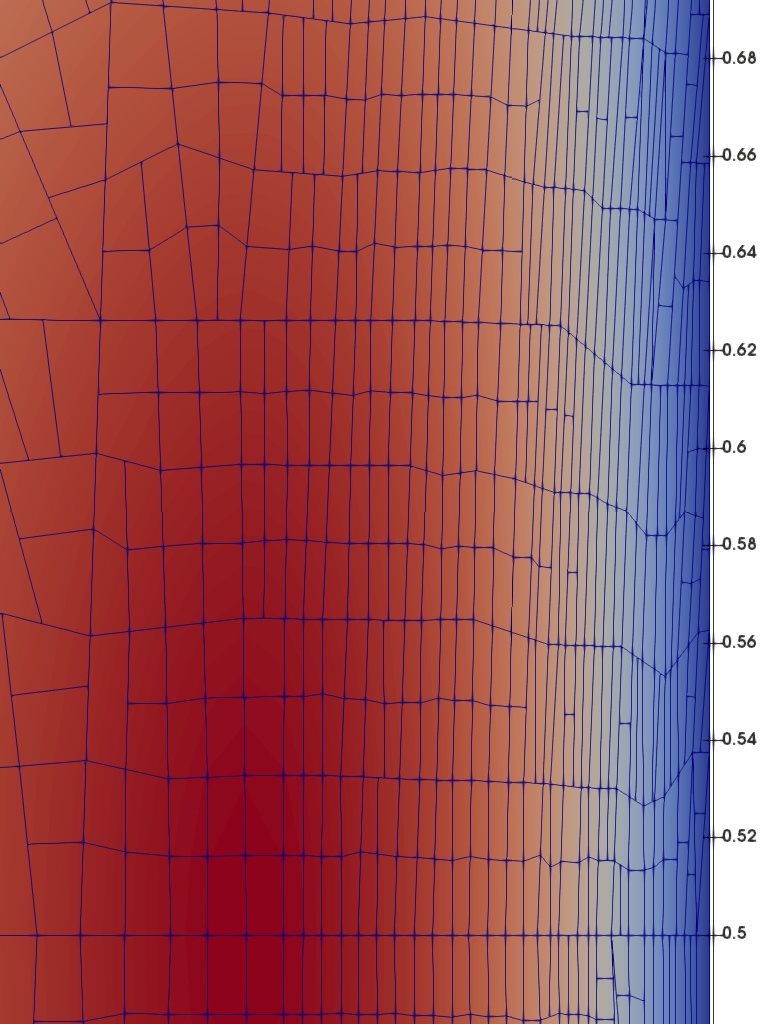}
    \caption{VEM order 1. Adaptive step n. 21.}
  \end{subfigure}
  \begin{subfigure}[b]{.32\linewidth}
    \centering
    \includegraphics[width=\linewidth]{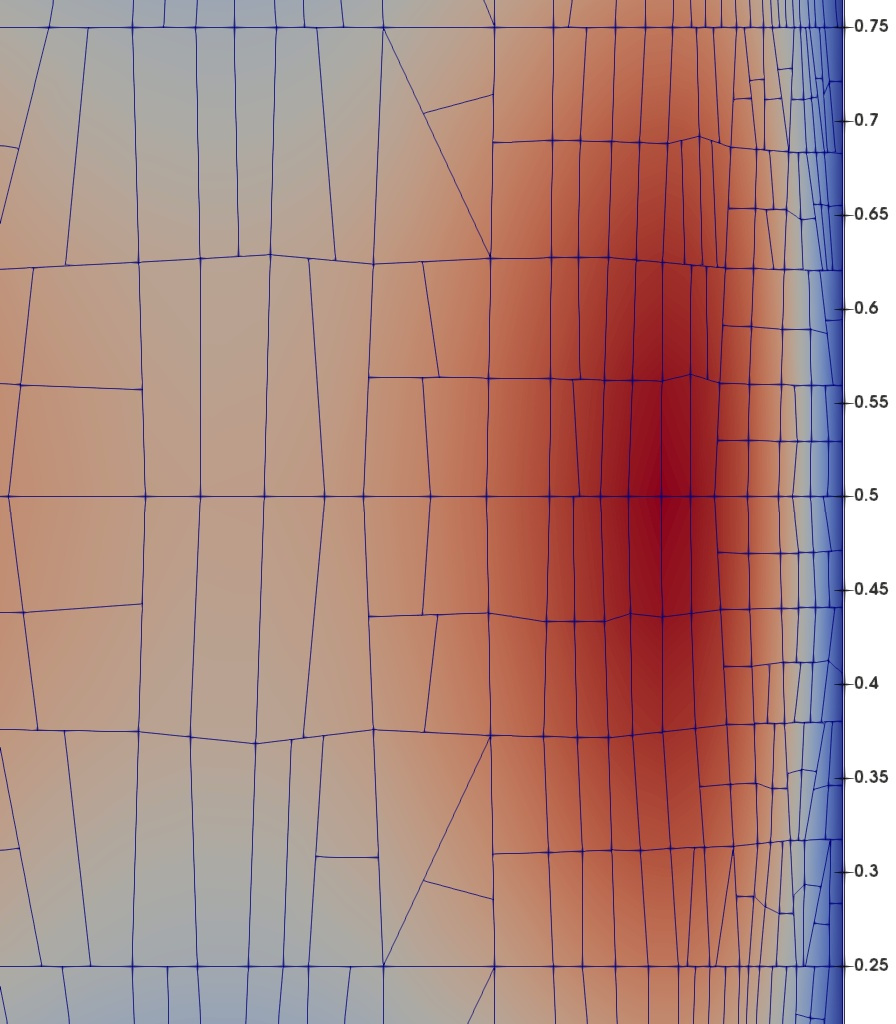}
    \caption{VEM order 2. Adaptive step n. 18}
  \end{subfigure}
  \caption{Test case 3. Zoom of the computed solutions and corresponding anisotropic grids at the final step of the adaptive algorithm based on employing the \emph{heuristically scaled estimator} $\sth$ defined in \eqref{eq:def_est_heur} to drive the adaptive algorithm.}
  \label{fig:test3-solutions-zoom}
\end{figure}
%%%%%%%%%%%%%%%%%%%
%%%%%%%%%%%%%%%%%%%
Finally, we compare the behavior of the computed estimator and of the error as a function of the number of the degrees of freedom.  In   Figure~\ref{fig:test3-analysis:totalerrors} we report the 
estimator $\sth$  and the  error $\tilde{e}$ versus the number of degrees of freedom, together with  the computed convergence
rates (obtained through a least square fitting). These results have been obtained with VEM of order 1, cf. Figure~\ref{fig:test3-analysis:order1:totalerrors} and with VEM of order 2, cf. Figure~\ref{fig:test3-analysis:order2:totalerrors}. As before, we compare these results with the analogous ones obtained with the isotropic error estimator $\eta^{\mathrm{iso}}_h$ defined in \eqref{eq:def_est_iso}. Again, as expected,  the adaptive algorithm based on employing the anisotropic estimator guarantees a lower error compared with the isotropic one.
%%%%%%%%%%%%
\begin{figure}
  \centering
  \begin{subfigure}{.49\linewidth}
    \centering \includegraphics[width =
    \linewidth]{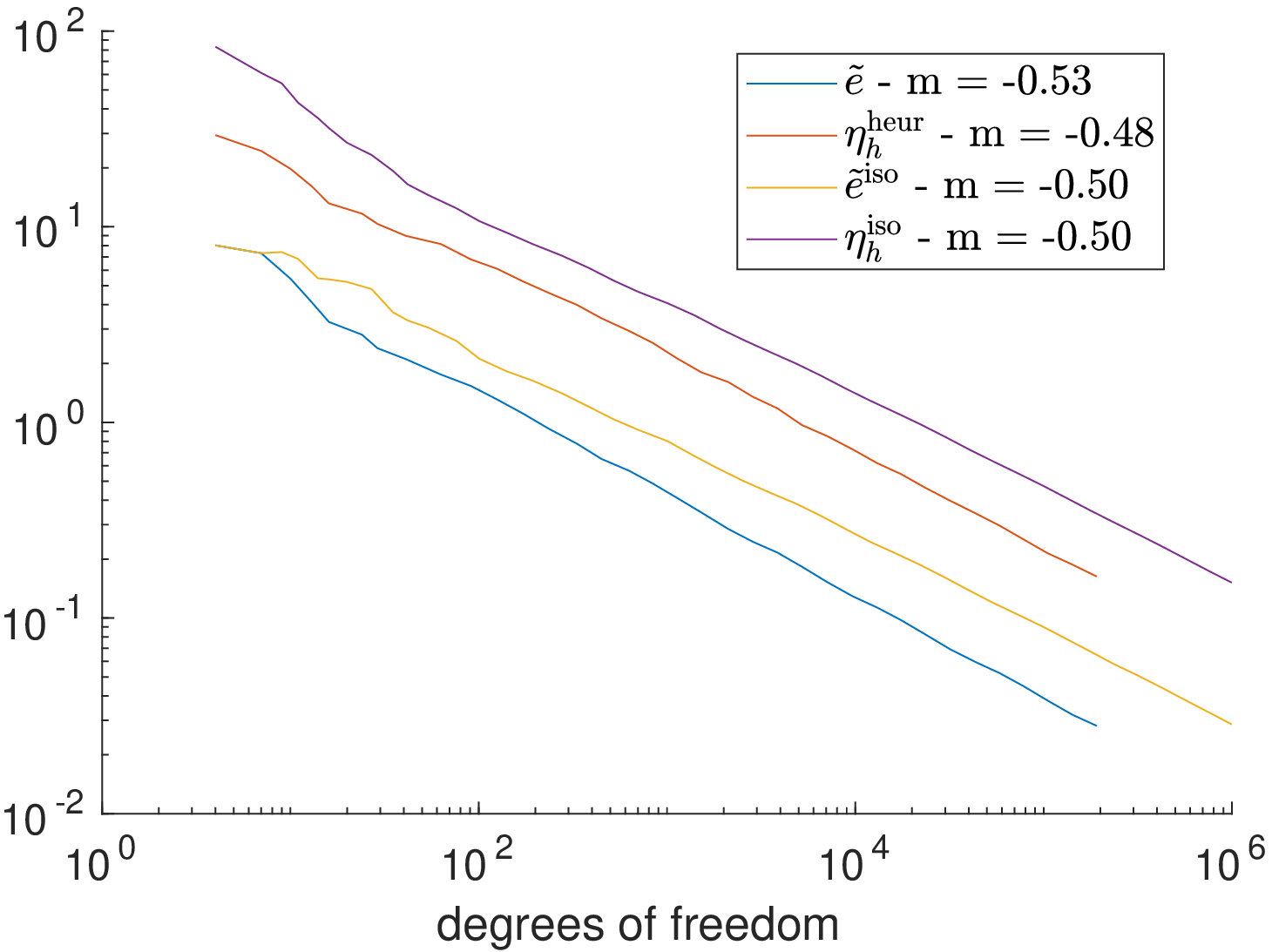}
    \caption{VEM of order 1.}
    \label{fig:test3-analysis:order1:totalerrors}
  \end{subfigure}
  \hfill
  \begin{subfigure}{.49\linewidth}
    \centering \includegraphics[width =
    \linewidth]{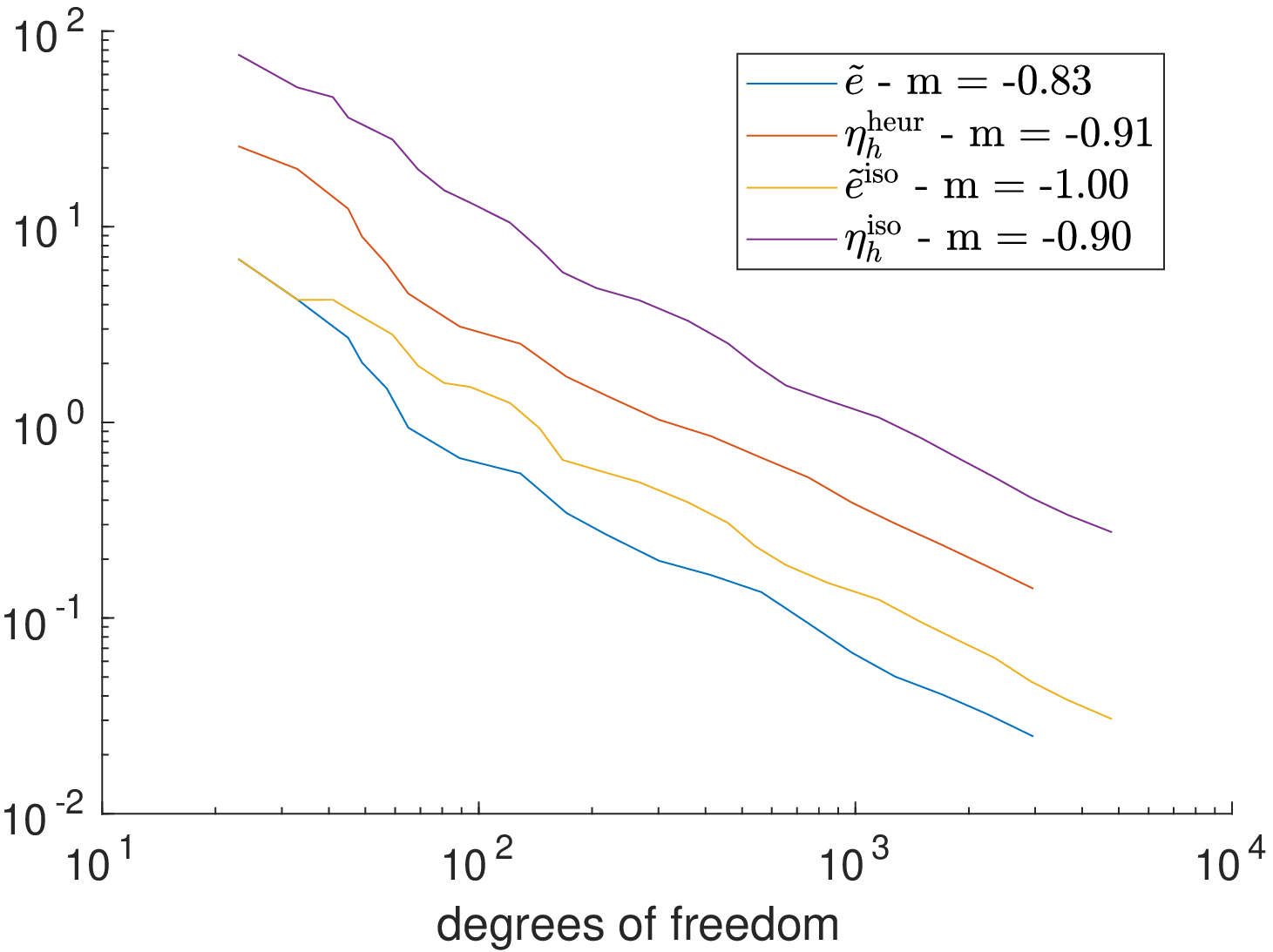}
    \caption{VEM of order  2.}
    \label{fig:test3-analysis:order2:totalerrors}
  \end{subfigure}
   \caption{Test case 3. Computed values of the estimator $\sth$, computed errors $\widetilde{e}$ based on employing the exact solution, and corresponding computed convergence rates \texttt{m} as a function of the number of degrees of freedom. The results are compared with the analogous quantities obtained based on employing the isotropic error estimator $\eta^{\mathrm{iso}}_h$ defined in \eqref{eq:def_est_iso}.}
\label{fig:test3-analysis:totalerrors}
\end{figure}
\FloatBarrier

%%%%%%%%%%%%%
%%%%%%%%%%%%%
%%%%%%%%%%%%%
%%%%%%%%%%%%%
%%%%%%%%%%%%%
%%%%%%%%%%%%%
%%%%%%%%%%%%%

\section*{Acknowledgments}
P.F.A., S.B., A.B., A.D., and M.V. acknowledge the financial support of MIUR through the PRIN grant n. 201744KLJL.
P.F.A., S.B., A.B., A.D., and M.V. have also been funded by INdAM-GNCS.
S.B., A.B., and A.D. also acknowledge the financial support of MIUR through the project ``Dipartimenti di Eccellenza 2018-2022'' (CUP E11G18000350001).

\bibliographystyle{abbrv}
\bibliography{literature}

\end{document}